\documentclass[a4paper]{amsart}
\usepackage{cases}
\usepackage[dvipsnames]{xcolor}
\usepackage{bigints, soul}
\usepackage{mathtools}
\usepackage[utf8]{inputenc}
\DeclareMathOperator*{\supp}{supp}
\DeclareMathOperator*{\sgn}{sgn}

\DeclareMathOperator*{\dist}{dist}
\usepackage[shortlabels]{enumitem}
\usepackage{cases}
\newtheorem{theorem}{Theorem}[section]
\newtheorem*{problem*}{Problem}
\newtheorem*{proposition*}{Proposition}
\newtheorem{lemma}[theorem]{Lemma}
\newtheorem{proposition}[theorem]{Proposition}
\newtheorem{question}{Question}

\newtheorem{corollary}[theorem]{Corollary}
\theoremstyle{remark}
\newtheorem*{remark*}{Remark}
\newtheorem*{lemma*}{Lemma}
\newtheorem{remark}[theorem]{Remark}
\theoremstyle{definition}
\newtheorem{definition}[theorem]{Definition}

\newtheorem*{definition*}{Definition}
\newtheorem{example}[theorem]{Example}
\newtheorem*{example*}{Example}
\newcommand{\floor}[1]{\left\lfloor #1 \right\rfloor}

\usepackage{hyperref}
\usepackage{bm}
\usepackage{MnSymbol}
\reversemarginpar
\newcommand{\F}{\mathbb{F}}
\newcommand{\E}{\mathbb{E}}
\newcommand{\K}{\mathbf{K}}
\newcommand{\C}{\mathbf{C}}
\def\N{\mathbb{N}}
\def\RR{\mathbb{R}}
\def\CC{\mathbb{C}}
\def\LL{\mathbb{L}}
\def\L{\mathbf{L}}
\def\M{\mathbf{M}}

\def\lam{\bm \lambda}

\def\xx{\mathbf{x}}

\def\yy{\mathbf{y}}
\def\zz{\mathbf{z}}

\def\WW{{\mathbf{W}}}
\def\KK{\mathbb{K}}
\def\G{{\mathcal G}}
\def\GS{{\mathcal {GS}}}
\def\GA{{\mathcal {GA}}}
\def\B{{\mathcal B}}
\def\X{{\mathbb X}}
\def\V{{\mathbb V}}
\def\Y{{\mathbb Y}}
\def\Z{{\mathbb Z}}
\def\F{{\mathbb F}}
\def\g{{\mathbf g}}
\def\k{{\mathbf k}}
\def\1{{\mathbf 1}}
\def\w{\mathbf{w}}
\newcommand{\EE}{\ensuremath{\mathcal{E}}}

\begin{document}

\title[Weak weight-semi-greedy Markushevich bases]{Weak weight-semi-greedy Markushevich bases}

\date{}

\author[M. Berasategui]{Miguel Berasategui}

\author[S. Lassalle]{Silvia Lassalle}

\address{IMAS--UBA--CONICET - Pab I,
Facultad de Cs. Exactas y Naturales, Universidad de Buenos
Aires, (1428) Buenos Aires, Argentina}
\email{mberasategui@dm.uba.ar}

\address{Departamento de Matem\'{a}tica y Ciencias, Universidad de San
Andr\'{e}s, Vito Dumas 284, (1644) Victoria, Buenos Aires,
Argentina  and IMAS--CONICET.}
\email{slassalle@udesa.edu.ar}

\thanks{This project was partially supported by CONICET PIP 0483 and ANPCyT PICT-2018-04104. The second author was also supported by PAI-UDESA 2020-2021.}


\begin{abstract}
The main purpose of this paper is to study weight-semi-greedy Markushevich bases, and in particular, find conditions under which such bases are weight-almost greedy. In this context, we prove that, for a large clase of weights, the two notions are equivalent. We also show that all weight semi-greedy bases are truncation quasi-greedy and weight-superdemocratic. In all of the above cases, we also bring to the context of weights the weak greedy and Chebyshev greedy algorithms -  which are frequently studied in the literature on greedy approximation. \\
In the course of our work a new property arises naturally and its relation with squeeze symmetric and bidemocratic bases is given. \\In addition, we study some parameters involving the weak thresholding and Chebyshevian greedy algorithms. Finally, we give examples of conditional bases with some of the weighted greedy-type conditions we study. 
\end{abstract}

\maketitle
\color{black}

\section{Introduction and background}\label{sectionintro}
\color{black}

Let $\X$ be an infinite dimensional separable Banach space over the real or complex field $\KK$. A sequence $\B=(\xx_i)_{i\in\N}$ is \emph{fundamental} if it generates the entire space,  that is $\X=\overline{[\xx_i: i\in \N]}$, and it is \emph{minimal} or a \emph{minimal system} if there is a (unique) sequence $\B^*=(\xx_i^*)_{i\in\N}$ in the dual space $\X^*$ (which we call the \emph{biorthogonal functionals}) such that $\xx_i^*(\xx_j)=\delta_{i,j}$ for all $i,j\in \N$.  If in addition $\B^*$ is total, that is if 
$$
\X^*=\overline{[\xx_i^*: i\in \N]}^{w^*},
$$
the sequence $\B$ is a \emph{Markushevich basis} for $\X$.  When there is $\C>0$ such that 
$$
\|\sum_{j=1}^{n}\xx_i^{*}(x)\xx_i\|\le \C\|x\|,\qquad\forall x\in \X,\ \forall n\in \N,
$$
the sequence $\B$ is a \emph{Schauder basis} and the minimum $\C>0$ for which the above inequality holds is the \emph{basis constant} of $\B$. A \emph{basic sequence} is a Schauder basis for the closure of its span. A Schauder basis $\B$ is $\C$-\emph{unconditional} if 
$$
\|\sum_{j=1}^{\infty}b_j\xx_j^*(x)\xx_j\|\le \C\|x\|,\qquad\forall x\in \X,\  \forall (b_j)_{j\in \N}\subset\KK\colon\ |b_j|\le 1\ \forall j\in \N,
$$
and it is $\C$-suppression unconditional if 
$$
\|P_A(x)\|\le \C\|x\|,\qquad\forall x\in \X,\ \forall A\in \N^{<\infty}
$$
where $P_A(x)$ denotes the projection of $x$ on $A$ (with respect to $\B$), that is
\begin{align*}
P_A(x)=\sum_{j\in A}\xx_j^*(x)\xx_j, 
\end{align*}
with the convention that the sum is zero if $A$ is empty. 
\\
It is well-known that $\C$-unconditionality entails $\C$-suppression unconditionality, whereas $\C$-suppression unconditionality entails $2\C$-unconditionality if $\KK=\RR$ and $4\C$-unconditionality if $\KK=\CC$. 
\\
In this paper, unless otherwise stated by a \emph{basis} $\B$ we mean a Markushevich basis with  biorthogonal functionals $\B^*$. We will refer to this sequence as the \emph{dual basis} of $\B$. Except in Section~\ref{sectionseparation}, we will assume that both $\B$ and $\B^*$ are bounded and we set 
\begin{align}
\lambda :=\sup_{i\in \N}\|\xx_i\|, \qquad  \lambda ':=\sup_{i\in \N}\|\xx_i^*\|, \qquad\lambda'' :=\sup_{i\in \N}\|\xx_i\|\|\xx_i^*\|, \label{constants}
\end{align}
a notation that we will use for all such bases. As usual, we use $\supp(x)$ to denote the support of $x\in \X$, that is the set $\{i\in \N: \xx^*_i(x)\ne 0\}$ and we set $\N^{<\infty}:=\{A\subset \N: |A|<\infty\}$.\\
Given a basis  $\B$ for $\X$, the \emph{Thresholding Greedy Algorithm} (TGA for short) gives approximations to vectors in $\X$ in terms of greedy sums, or equivalently, in terms of projections on greedy sets.
We will consider the more general concept of $t$-greedy sets, which are involved in approximations given by \emph{Weak Thresholding Greedy Algorithms} (WTGA for short).

\begin{definition}\label{definiciongreedyset} Let $0<t\le 1$. For each $m\in\N$, a set $A\subset \N$ is an \emph{$m$-$t$-greedy set for} $x\in \X$ if $|A|=m$ and
\begin{equation}
\min_{j\in A}\left|\xx_j^*\left(x\right)\right|\ge t\max_{j\in \N\setminus A}\left|\xx_j^*\left(x\right)\right|.
\end{equation}
If $t=1$, $A$ is called an \emph{$m$-greedy set for} $x$. By $\G(x,m,t)$ we denote the set of  all $m$-$t$-greedy sets for $x$, with  $\G(x,0,t)$ consisting only of the empty set, and we let $\G(x,t):=\bigcup_{m\in \N}\G(x,m,t)$. Also, by $\Lambda_m(x)$ we denote the element of $\G(x,m,1)$ with the property that for any $B\in \G(x,m,1)$  such that $B\not=\Lambda_m(x)$, we have 
$$
\max(\Lambda_m(x)\setminus B)<\min( B\setminus \Lambda_m(x)).
$$

\end{definition}

The TGA was introduced by Temlyakov  in~\cite{T1998} in the context of the trigonometric system, and extended by Konyagin and Temlyakov to general Banach spaces in~\cite{KT1999}, where the authors defined the concepts of \emph{greedy} and \emph{quasi-greedy} Schauder bases. A Schauder basis $\B$ for $\X$ is greedy with constant $\C>0$ (or \emph{$\C$-greedy}) if 
\begin{equation}
\|x-P_{\Lambda_m}(x)\|\le \C\sigma_{m}(x),\qquad\forall x\in \X,\ \forall m\in \N, \label{greedydefi1}
\end{equation}
where $\sigma_m(x)$ is the best  $m$-term approximation error (with respect to $\B$) given by 
\begin{equation}
\sigma_m(x)=\sigma_m(x)[\B,\X]=\inf_{\substack{y\in \X\\|\supp(y)|\le m}}\|x-y\|.\label{errorapprox}
\end{equation}
\begin{remark}\label{remarkallthesame} Note that, due to the continuity of the norm, it is equivalent to take $|\supp(y)|=m$ in \eqref{errorapprox}. Also, a standard small perturbation argument gives that \eqref{greedydefi1} is equivalent to 
\begin{equation}
\|x-P_A(x)\|\le \C\sigma_{m}(x),\qquad\forall x\in \X,\ \forall m\in \N,\ \forall A \in \G(x,m,1) \nonumber
\end{equation}
and  also to 
\begin{equation}
\|x-P_A(x)\|\le \C\sigma_{m}(x),\qquad \forall x\in \X,\ \forall m\in \N,\ \text{for some}\  A\in \G(x,m,1).  
\nonumber
\end{equation}
\end{remark}
Since their inception, greedy bases have been widely studied; see for example the book by Temlyakov \cite{T2011}, the more recent articles  \cite{AA2016}, \cite{AABW2021}, \cite{AAW2019}, \cite{BB2017} and the references therein. 
Greedy bases are a subclass of quasi-greedy ones, defined as follows: $\B$ is $\C$-quasi-greedy if 
\begin{equation}
\|P_{\Lambda_m}(x)\|\le \C\|x\|,\qquad\forall x\in \X,\ \forall m\in \N.\nonumber
\end{equation}
Following  \cite{AA2017}, we say that $\B$ is \textit{$\C$-suppression quasi-greedy} if 
\begin{equation}
\|x-P_{\Lambda_m}(x)\|\le \C\|x\|,\qquad\forall x\in \X,\ \forall m\in \N.\nonumber
\end{equation}
Intermediate structures between these two are almost greedy bases -introduced by Dilworth, Kalton, Kutzarova and Temlyakov in~\cite{DKKT2003}- and semi-greedy bases -~defined by Dilworth, Kalton and Kutzarova in \cite{DKK2003}-. A basis is almost greedy with constant $\C>0$ (or \textit{$\C$-almost greedy}) if 
\begin{equation}
\|x-P_{\Lambda_m}(x)\|\le \C\widetilde{\sigma}_{m}(x),\qquad\forall x\in \X,\ \forall m\in \N,\nonumber
\end{equation}
where $\widetilde{\sigma}_m(x)$ is the best $m$-term approximation error to $x$ via projections (with respect to $\B$), given by
\begin{equation}
\widetilde{\sigma}_m(x)=\widetilde{\sigma}_m[\B,\X](x)=\inf_{\substack{B\subset \N\\|B|=m}}\|x-P_B(x)\|.\label{projectionerror}
\end{equation}
Notice that, since $\B^*$ is weak star null, in order to compute \eqref{projectionerror}, it is equivalent to take $|B|=m$  or $|B|\le m$. 

On the other hand, $\B$ is semi-greedy with constant $\C>0$ (or \emph{$\C$-semi-greedy}) if 

\begin{equation}
\inf_{|\supp(y)|\subset \Lambda_m(x)}\|x-y\|\le \C\sigma_{m}(x),\qquad\forall x\in \X,\ \forall m\in \N. \nonumber
\end{equation}
The algorithm associated with a semi-greedy basis is called a \emph{Chebyshev Greedy Algorithm} (CGA for short). 

\begin{remark}\label{remarkallequal3} 
As in the case of greedy bases (Remark~\ref{remarkallthesame}), one may replace $\Lambda_m(x)$ by  all $A\in \G(x,m,1)$ or at least one $A\in \G(x,m,1)$ in the definitions of quasi-greedy, almost greedy and semi-greedy bases, obtaining equivalent notions, with the same constant $\C$.
\end{remark}
Almost greedy bases have been studied, among other papers, in \cite{AA2017, AADK2019, B2019,DHK2006, DKK2003}, while semi-greedy bases have been studied, for example, in \cite{BL2021, B2019, BBGHO2018, T2008}.
\\
Originally defined for Schauder bases, the concepts of quasi-, almost and semi-greediness were extended to and studied in the context of seminormalized minimal systems with seminormalized biorthogonal functionals (see for example \cite{AABW2021, BBGHO2018, DKO2015, W2000}). It is known that quasi-greedy systems are Markushevich bases \cite{AABW2021}, and that almost greediness and semi-greediness are equivalent concepts for Markushevich bases (see \cite[Theorem~4.2]{BL2021}, \cite[Theorem~1.10]{B2019} and \cite[Theorem~5.3]{DKK2003}), but not for general minimal systems \cite[Example~4.5]{BL2021}.  
\\
Weaker versions of the TGA and the CGA have also been studied. In the case of the WTGA, one may mention for example \cite{DKO2015, DKSW2012, G2009, KT2002, KT2003}, and for the WCGA see \cite{DGHKT2021, DKO2015, DKSW2012, T2015} among others. These algorithms consider approximations involving $t$-greedy sets, for some $0<t\le 1$. 
\\
Recently, Dilworth, Kutzarova, Temlyakov and Wallis extended the concepts almost greedy and semi-greedy Schauder bases to the context of sequences of weights \cite{DKTW2018}. In their work, the authors follow a similar extension for greedy Schauder bases previously introduced and studied by Kerkyacharian, Picard and  Temlyakov \cite{KPT2006}. In \cite{DKTW2018} it is shown that weight-almost greedy bases are weight-semi-greedy and that the converse holds for Schauder bases  if the Banach space has finite cotype. The cotype condition is removed in \cite{B2020}, where the author raises the question of extending the results to a more general class of Markushevich bases.  \\
In this context, we focus on weight-semi-greedy bases as well as a weak variant of them involving weak algorithms, and in particular the implication from (weak) weight-semi-greediness to weight-almost greediness, under different hypothesis. 
We prove that, for a large class of weights, (weak) weight-semi-greedy Markushevich bases are weight-almost greedy. Also, we prove that if we impose some mild conditions on the basis,  the above holds for any weight. Finally, in this context we show that for any weight $\w$, a weight semi-greedy basis with weight $\w$ is  truncation quasi-greedy and $\w$-superdemocratic. In the course of our study, some new properties arise naturaly, namely \emph{almost semi-greedy} Markushevich bases, and its weighted and weak counterparts. We show that this property turns out to be equivalent to squeeze symmetry, and can be used to characterize bidemocracy.  \\
Additionally, we also study parameters involving the WTGA in the classical context, that is for constant weights.  
\\
The paper is structured as follows: In  Section~\ref{sectionseparation}, we recall results from \cite{BL2021} about the finite dimensional separation property (FDSP) which is useful to replace arguments involving the Schauder basis constant when working with Markushevich bases. We also study a variant of this property for bounded uniformly discrete sequences that allows us to improve  some of the upper bounds that would be obtained by using the FDSP. In Section~\ref{sectionweightweaksemigreedy} we study weight-semi-greedy bases. We introduce the weak notions of weight-almost and semi-greedy bases and prove our main results:  Theorem~\ref{theoremswemigreedydisjointgreedynotc0}, and Theorem~\ref{theoremc0notl1norming}. 
Our proofs lead us to consider an intermediate notion that we call almost semi-greedy bases and the formally weaker versions thereof, as well as their weighted counterparts. These properties show a tight connection between the CTGA and two properties generally studied in connection to the TGA: bidemocracy and squeeze symmetry. This can be found in Section~\ref{sectionalmostsemigreedy}. To complete the framework  for our study of weighted bases, in Section~\ref{sectionweightweakalmostgreedy}, we deal with the formally weaker variant of weight-almost greedy bases. We show that, as it is the case of weak almost greedy bases  (see \cite[Proposition 2.3]{BL2021}, \cite[Theorem 6.4]{DKSW2012}) they are equivalent to the weight almost greedy property.  In Section~\ref{sectionquasi-greedyalmostgreedy}, we extend  some of the results on the WTGA and CTGA from \cite{BL2021}. In particular, we study parameters that allow us to estimate how the algorithms involved in said results perform with respect to bases that are not necessarily semi-greedy. Additionally, for bases that are (weak)-semi-greedy, we improve the known estimate for the quasi-greedy constant, and give a new estimate for the almost greedy one. Finally, in Section~\ref{sectionexamples}, we give some examples of bases with the properties we study. \\
Our general notation is standard. In addition to what was set before, unless otherwise stated, $\X$, $\Y$ and $\Z$  denote infinite-dimensional Banach spaces, whereas  $\E$ and $\F$ denote finite dimensional spaces and $\V$ stands for a Banach space without any restrictions on the dimension. Given a Banach space $\V$  over $\KK$, $S_{\V}$ denotes its unit sphere and $\V^*$ denotes its dual space. For $x\in \V$, $\widehat{x}$ denotes the image of $x$ in the bidual space $\V^{**}$, via the canonical inclusion. The same notation will be used for subsets of $\V$. The constant $\kappa$ is set as $\kappa=1$ if $\KK=\RR$, and $\kappa=2$ if $\KK=\CC$. 
\\
Given any set $A\subset \N$, we define
$$
\EE_A:=\{\varepsilon=(\varepsilon_i)_{i\in A}\colon\ |\varepsilon_i|=1\quad \forall i\in A\}
$$
with the convention that $\EE_A=\emptyset$ if $A=\emptyset$. When $A\subset \N$ is finite and $\B=(\xx_i)_{i\in\N}$ is a basis, for $\varepsilon\in\EE_A$, we denote 
$$
\1_{\varepsilon, A}:=\1_{\varepsilon, A,\B}=\sum_{i\in A}\varepsilon_i\xx_i 
$$
with the convention that any sum over the empty set is zero. Also, if $\varepsilon\in \EE_A$ and $B\subset A$, we write $\1_{\varepsilon,B}$ considering the natural restriction of $\varepsilon$ to $B$.  If $\varepsilon_j=1$ for all $j$, we write $\1_A$. If $\B^*=(\xx_i^*)_{i\in\N}$ is the dual basis of  a basis $\B$, the supremum norm of $x\in \V$ is $\|x\|_{\infty}:=\displaystyle \sup_{i\in \N}|\xx_i^*(x)|$, whereas for $1\le p<\infty$, $\|x\|_p$ denotes the usual $\ell_p$ norm, when it is defined. Finally, for each $x\in \X$ we define $\varepsilon(x)=(\sgn(\xx_i^*(x)))_{i\in \N}$, where $\sgn(0)=1$ and $\sgn(t)t=|t|$ for all $t\in \KK$. \\
All the remaining relevant terminology and preliminaries will be given in corresponding sections. 

\section{Separation properties}\label{sectionseparation}
In this section, we recall the definition of the finite dimensional separation property (FDSP) and some  results from \cite{BL2021} used to prove the implication from semi-greedy to almost greedy for Markushevich bases. Also,  we study a related property that allows us to improve some of  the upper bounds that would be obtained by using the FDSP.

\begin{definition} (\cite[Definition 3.1]{BL2021}) \label{definitionseparation} A sequence $(u_i)_{i\in\N}\subseteq \X$  has the \emph{finite dimensional separation property} (FDSP for short) if there is a positive constant $\M$ such that for every separable subspace $\Z\subset \X$ and every $\epsilon>0$, there is a basic subsequence $(u_{i_k})_{k\in\N}$ with basis constant no greater than $\M+\epsilon$ satisfying the following: For every finite dimensional subspace $\F \subset \Z$ there is $j_{\F,\epsilon}=j_{\F,((u_{i_k})_{k\in \N},\epsilon)}\in \N$ such that 
\begin{equation}
\|x\|\le (\M+\epsilon)\|x+z\|,\label{separation}
\end{equation}
for all $x\in \F$ and all $z\in \overline{[u_{i_k}:k> j_{\F, \epsilon}]}$. We call any such subsequence a \emph{finite dimensional separating sequence} for $(\Z, \M, \epsilon)$, and the minimum $\M$ for which this property holds will be called the \emph{finite dimensional separation constant} $\M_{fs}[ (u_i)_{i\in \N}, \X]$, leaving the sequence and the space implicit when it is clear.
\end{definition}

\begin{remark}Note that a subsequence $(u_{i_k})_{k\in\N}$ is finite dimensional separating for $(\Z, \M, \epsilon)$ if and only if \eqref{separation} holds for any $x\in S_{\F}$ and all $z\in [u_{i_k}:k> j_{\F,\epsilon}]$. 
\end{remark}
Recall that  a subspace $\Y\subset \X^*$ is said to be $r$-norming for $\X$, $0<r\le 1$,  if 
$$
r\|x\|\le \sup_{\substack{x^*\in S_{\Y}}}|x^*(x)|. 
$$
Additionally, we will say that a set $R\subset \X^*$ is $r$-norming if the subspace it spans in $\X^*$ is $r$-norming. 
\\
Also recall that a sequence $(v_i)_{i\in \N}$ is a \emph{block basis} of a Markushevich basis $(\xx_k)_{k\in \N}$ if there is a sequence of scalars $(b_k)_{k\in\N}$ and sequences of positive integers $(n_i)_{i\in\N}$, $(m_i)_{i\in\N}$ with $n_i\le m_i<n_{i+1}$ for all $i$  such that
$$
v_i=\sum\limits_{k=n_i}^{m_i}b_k\xx_k,
$$
with at least one nonzero $b_k$ for each $i\in \N$. In particular, any subsequence of a Markushevich basis is a block basis of it.

\begin{proposition} {\rm (\cite[{Proposition 3.11}]{BL2021})} \label{propositionseparation} Let $(v_i)_{i\in\N}\subset \X$ be a  block basis of a Markushevich basis $(\yy_k)_{k\in\N}$ for a subspace $\Y\subset \X$ with biorthogonal functionals $(\yy_k^{*})_{k\in\N}$. Let $(a_i)_{i\in\N}$ be a scalar sequence such that $(z_i:=a_i v_i)_{i\in\N}$ is seminormalized. The following hold: 
\begin{enumerate}[\upshape (i)]
\item \label{Markushevich} $(z_i)_{i\in\N}$ and $(v_i)_{i\in\N}$ have the finite dimensional separation property with the same constant, that is $\M_{fs}[(v_i)_{i\in\N}, \X]=\M_{fs}[(z_i)_{i\in\N}, \X]$.

\item \label{weakzero} If either $0\in \overline{\{z_i \}}^{w}_{i\in\N}$ or $\X$ is a dual space and $0\in \overline{\{z_i \}}^{w^{*}}_{i\in\N}$, then $\M_{fs}=1$. 

\item \label{noncompact} If $\overline{\{z_i \}}^{w}_{i\in\N}$ is not weakly compact, then 
\begin{equation}
\M_{fs} \le \Big( 2+\inf \Big\{ \frac{\|x^{**}\|}{\dist{(x^{**},\widehat{\X})}}\colon\ \ x^{**}\in \overline{\{\widehat{z}_{i\in \N}\}}^{w^{*}}_{i\in\N}\setminus \widehat{\X}\Big\}\Big)^2.\nonumber
\end{equation}

\item \label{norming} If $\Y=\X$ and $\overline{[\yy^*_k \colon k\in \N]}$ is $r$-norming, then $\M_{fs}\le r^{-1}$.  

\item \label{schauder} If $\Y=\X$ and $(\yy_k)_{k\in\N}$ is a Schauder basis for $\X$ with constant $\K_b$, then $\M_{fs}\le \K_{b}$. 
\end{enumerate}
\end{proposition}

\begin{remark}\label{remarkdisjointlysupported} Proposition~\ref{propositionseparation} remains valid if we replace a block basis with any sequence of nonzero vectors, that are pairwise disjointly supported.
\end{remark}

Next, we prove further results involving a similar property. Given $\delta>0$, we say that a set $S\subset \X$ is $\delta$-uniformly discrete if $\|x-y\|\ge \delta$ for all $x, y\in S$, $x\ne y$. For a sequence $(x_i)_{i\in \N}$ by $\delta$-uniformly discrete we mean that  $\|x_i-x_j\|\ge \delta$ for all $i\not=j$.  

\begin{lemma}\label{lemmaseparated1} Let $S\subset \X$ be a bounded uniformly discrete set and $\F\subset \X$ a finite-dimensional subspace. Given $\epsilon>0$, there are $x \ne y\in S$ such that for every $b \in \KK$ and every $z\in \F$, 
$$
\|z\|\le (1+\epsilon)\|z+b (x -y)\|.
$$
\end{lemma}
\begin{proof}
Choose $\delta>0$ so that $S$ is $\delta$-uniformly discrete, and $0<\epsilon'<\epsilon$ so that
\begin{equation}
0<\frac{1}{1-4\delta^{-1}\epsilon'-\epsilon'}<1+\epsilon.\label{epsilon'}
\end{equation}
Let $\{z_1,\dots,z_n\}$ be an $\epsilon'$-net in $S_{\F}$, and $\{z_1^*,\dots,z_n^*\}\subset S_{\X^*}$ so that $z_j^*(z_j)=1$ for all $1\le j\le n$. Since $S$ is bounded, there exists $z_0^{**}\in \X^{**}$ a $w^*$-accumulation point of $\widehat{S}\subset \X^{**}$. Hence, there are $x\ne y\in S$, such that for $1\le k\le n$, 
$$
|z_0^{**}(z_k^{*})-z_k^{*}(x)|\le \epsilon' \quad \text{and}\quad |z_0^{**}(z_k^{*})-z_k^{*}(y)|\le \epsilon'
$$
Fix $z\in S_{\F}$, and choose $1\le k\le n$ so that 
$$
\|z_k-z\|\le \epsilon'. 
$$
Now pick $b \in \KK$. If $|b|\le 2\delta^{-1}$, then
\begin{align*}
\|z+b(x - y)\|\ge& \|z_k+b(x - y)\|-\epsilon'\ge |z_{k}^{*}(z_k+b(x - y)|-\epsilon'\\
\ge& 1-|b||z_k^{*}(x - y)|-\epsilon'\\
\ge& 1-2\delta^{-1}|z_0^{**}(z_k)-z_k^{*}(x) + z_{k}^{*}(y) - z_0^{**}(z_k)|-\epsilon'\\
\ge& 1-4\delta^{-1}\epsilon'-\epsilon'. 
\end{align*}
Hence, by \eqref{epsilon'}, 
$$
\|z\|=1\le (1+\epsilon)\|z+b(x - y)\|.
$$
On the other hand, if $|b|>2\delta^{-1}$, then 
$$
\|z+b(x - y)\|\ge |b|\|x - y\|-\|z\|\ge 1=\|z\|. 
$$
This completes the proof for $z\in S_{\F}$, and hence by scaling for all $z\in \F$. 
\end{proof}

\begin{lemma}\label{lemmaxi-x2}
Let $\X$ be a Banach space, and $(u_j)_{j\in\N}\subseteq  \X$ a bounded uniformly discrete sequence. Then, for any separable subspace $\Z\subset \X$ and $\epsilon>0$ there is a subsequence $(u_{j_n} )_{n\in\N}$ such that the sequence $\{u_{j_{2n-1}}-u_{j_{2n}}\}_{n\in\N}$ is basic with basis constant no greater than $(1+\epsilon)$ and satisfies the following: For any finite dimensional subspace $\F\subset \Z$ and every $\xi>0$, there is $r_{\F,\xi}\in \N$ such that for all $y\in \F$ and all $v \in \overline{[u_{j_{2n-1}}-u_{j_{2n}}: n> r_{\F,\xi}]}$, 
$$
\|y\|\le (1+\xi)\|y+v \|.
$$
\end{lemma}
\begin{proof}
The argument is very similar to that of \cite[Lemma 3.5]{BL2021}; we give a proof for the sake of completeness. 

Fix $\Z\subset \X$ a separable subspace and  $\epsilon>0$. Choose a sequence $(v_j)_{j\in\N}\subset \Z$, dense in $\Z$, and a sequence of positive scalars $(\epsilon_j)_{j\in\N}$ so that $\prod\limits_{j=1}^{\infty}(1+\epsilon_j)\le (1+\epsilon)$. Let $j_0:=1$. Applying Lemma~\ref{lemmaseparated1} to the set $\{u_j\}_{j>j_0}$, we can find $j_0<j_1<j_2$ so that for all $y\in [v_k, u_k: 1\le k\le j_0 ]$ and all $b\in \KK$,
$$
\|y\|\le (1+\epsilon_1)\|y+b  (u_{j_1}-u_{j_2}) \|.
$$
Similarly, we can find $j_2<j_3<j_4$ so that for all $y\in [v_k, u_k: 1\le k \le j_2]$  and all $b\in \KK$,
$$
\|y\|\le (1+\epsilon_2)\|y+b  (u_{j_3}-u_{j_4}) \|.
$$
By an inductive argument, we obtain a strictly increasing sequence of positive integers $\{j_n\}_{n\in\N}$ such that for all  $y\in [v_k, u_k\colon 1\le k\le j_{2n-2}]$, $b\in \KK$ and $n\in\N$,
$$
\|y\|\le (1+\epsilon_n)\|y+b  (u_{j_{2n-1}}-u_{j_{2n}})\|.
$$
Then, for any  positive integers $m\le  l$, any $y\in [v_k, u_k\colon 1\le k\le j_{2m-2}]$ and any scalars $(a_n)_{m \le n\le l}$,
\begin{align*}
\|y\|\le& \prod\limits_{n=m}^{l}(1+\epsilon_{n})\|y+\sum\limits_{n=m}^{l}a_n (u_{j_{2n-1}}-u_{j_{2n}})\|\\
\le& \prod\limits_{n=m}^{\infty}(1+\epsilon_{n})\|y+\sum\limits_{n=m}^{l}a_n (u_{j_{2n-1}}-u_{j_{2n}})\|.
\end{align*}
In particular, $(u_{j_{2n-1}}-u_{j_{2n}})_{n\in \N}$ is basic with basis constant no greater than
$\prod\limits_{n=1}^{\infty}(1+\epsilon_n)\le 1+\epsilon$, and, given $\F \subset [ v_j\colon 1\le j\le n]$ for some $n\in \N$, we can pick $r_{\F,\xi}$ using the above computation. Now, standard density arguments allow us to obtain the result for any finite dimensional subspace of $\Z=\overline{[v_j \colon j\in\N]}$. 
\end{proof}
\medskip

\begin{corollary}\label{corollaryboth}
Let $\B=(\xx_i)_{i\in\N} \subset \X$ be a seminormalized Markushevich basis for $\Y\subset \X$, with finite dimensional separation constant no greater than $\M$. Then for every separable subspace $\Z\subset \X$ and every $\epsilon>0$, there is a basic subsequence $(\xx_{i_k})_{k\in\N}$ satisfying the following condition: For every finite dimensional subspace $\F\subset \Z$, there is $s_{\F,\epsilon}\in\N$ such that for every $x\in \F$, every $y \in \overline{[\xx_{i_k}: k\ge s_{\F,\epsilon}]}$ and every $z\in \overline{[\xx_{i_{2k-1}}-\xx_{i_{2k}}: k\ge s_{\F,\epsilon}]}$, 
\begin{equation}
\|x\|\le \min{\{(\M+\epsilon)\|x+y\|, (1+\epsilon)\|x+z\|\}}.\nonumber
\end{equation}
\end{corollary}
\begin{proof}
Fix $\Z\subset \X$ a separable subspace and  $\epsilon>0$. An application of Proposition~\ref{propositionseparation} gives a separating subsequence $(\xx_{i_l})_{l\in\N}$ for $(\Z,\M,\epsilon)$. Thus, for  any finite dimensional $\F\subset \Z$, every $x\in \F$ and every $y \in \overline{\left[\xx_{i_l}: l > j_{\F,\epsilon}\right]}$,
$$
\|x\|\le (\M+\epsilon)\|x+y\|. 
$$

Let $\B^*=(\xx_i^*)_{i\in\N}$ be the dual basis of $\B$. Since $(\xx^*_i)_{i\in \N}$ is bounded, $(\xx_{i_l})_{l\in\N}$ is uniformly discrete, then by Lemma~\ref{lemmaxi-x2} we obtain a further subsequence $(\xx_{i_{l_k}})_{k\in\N}$ such that for  any finite dimensional $\F\subset \Z$, every $x\in \F$ and every $z \in \overline{\left[\xx_{i_{l_{2k-1}}}-\xx_{i_{l_{2k}}}: k> r_{\F,\epsilon}\right]}$, 
$$
\|x\|\le (1+\epsilon)\|x+z\|. 
$$
Taking for each $\F$, $s_{\F,\epsilon}:=1+\max\{r_{\F,\epsilon}, j_{\F,\epsilon}\}$,  it is immediate from the above that  $(\xx_{i_{l_k}})_{k\in\N}$  has the desired properties. 
\end{proof}
\medskip

\section{Weak weight-semi-greedy bases}\label{sectionweightweaksemigreedy}
Let $\w=(w_i)_{i\in \N}$ be a sequence of positive numbers, and for each set $A\subset \N$, let 
$$
w(A):=\sum_{i\in A}w_i.
$$

The sequence $\w$ is called a \emph{weight} and, for $A\subset\N$, $w(A)$ is the $\w$-measure of $A$ (which might be infinite if $|A|=\infty$). In this section we introduce the weak weight-semi-greedy property, and study its relation with other notions studied in this context, in particular the weight-almost greedy property, introduced in \cite{DKTW2018}, which is a weaker version of the weight greedy property introduced and studied in \cite{KPT2006}. 

\begin{definition}\label{definitionweightsemigreedyetc}Let $\B$ be a basis for $\X$,  $\w$ a weight, and $\C>0$. Then: 
\begin{itemize}
\item $\B$ is \emph{weight-almost greedy with weight $\w$ and constant $\C$} (or $\C$-$\w$-almost greedy) if 
\begin{equation*}
\|x-P_{\Lambda_m(x)}\|\le \C   \inf_{\substack{B\in \N^{<\infty}\\ w(B)\le w(\Lambda_m)}}\|x-P_B(x)\|,  \qquad\forall x\in \X,\ \forall m\in \N. 
\end{equation*}
\item $\B$ is weight-semi-greedy with weight $\w$ and constant $\C$  (or $\C$-$\w$-semi-greedy) if 
\begin{equation*}
\inf_{\supp(y)\subset \Lambda_m(x)}\|x-y\|\le \C \inf_{\substack{z\in \X\\|\supp(z)|<\infty\\ w(\supp(z))\le w(\Lambda_m)}}\|x-z\|,   \qquad\forall x\in \X,\ \forall m\in \N. 
\end{equation*}
\end{itemize}
\end{definition}
As in the case of their regular counterparts, these weighted properties can be defined considering $\Lambda_m$, or all greedy sets, or at least one, obtaining equivalent notions. Also, for $\w$-greedy and $\w$-semi-greedy bases, it is equivalent to consider approximations using vectors with only finite support or any support, provided that the weight condition is kept. For the sake of completion, we give the proof for weight-semi-greedy bases; the proof for weight-greedy bases is similar. We will use a result that follows at once from the proof of \cite[Proposition 2.3]{B2020}, which is an extension of \cite[Proposition 4.5]{DKTW2018} without the Schauder hypothesis, and which uses the same definition of $\w$-semi-greedy bases given above. 

\begin{lemma}\label{lemmal1thenc0}Let $\B =(\xx_i)_{i\in \N}$  be a basis for $\X$ and $\w =(w_i)_{i\in \N}$ a weight.  If $\B$ is $\w$-semi-greedy and $(w_{i_k})_{k\in \N}$ is a subsequence such that $\sum_{k}w_{i_k}<\infty$, then $(\xx_{i_k})_{k\in \N}$ is a basic sequence equivalent to the canonical unit vector basis of $\mathtt{c}_0$.
\end{lemma}

\begin{remark}\label{remarkprojectionc0}\rm Note that under the conditions of Lemma~\ref{lemmal1thenc0}, if $A\subset \N$ is an infinite set such that $w(A)<\infty$, then the projections $P_A(x)$ are defined for each $x\in \X$, with unconditional convergence of the sums. Indeed, this follows at once from the fact that $\B^*$ is weak star null. Also, due to the totality condition, if $x$ has infinite support and finite $\w$-measure, then 
$$
x=\sum_{i\in \N}\xx_i^*(x)\xx_i,
$$
again with unconditional convergence. In particular, by \cite[Theorem 4.3]{DKTW2018}, this holds for $\w$-almost greedy bases. 
\end{remark}

\begin{lemma}\label{lemmaallthesameweightsemigreedy} Let $\B$ be a basis for $\X$, $\w$ a weight and $\C>0$. The following are equivalent:
\begin{enumerate}[\upshape (i)] 
\item \label{sgnallgredyset} For all $x\in \X$, $m\in \N$ and $A\in \G(x,m,1)$, there is $y\in \X$ with $\supp(y)\subset A$ such that
$$
\|x-y\|\le \C \inf_{\substack{z\in \X\\ w(\supp(z))\le w(A)}}\|x-z\|.
$$
\item \label{sgnallgredysetfinite}  For all $x\in \X$, $m\in \N$ and $A\in \G(x,m,1)$, there is $y\in \X$ with $\supp(y)\subset A$ such that
$$
\|x-y\|\le \C \inf_{\substack{z\in \X\\|\supp(z)|<\infty\\ w(\supp(z))\le w(A)}}\|x-z\|.     
$$
\item \label{sgnatgredyset} For all $x\in \X$ and $m\in \N$, there is $y\in \X$ with $\supp(y)\subset \Lambda_m(x)$ such that
$$
\|x-y\|\le \C \inf_{\substack{z\in \X\\ w(\supp(z))\le w(\Lambda_m)}}\|x-z\|.
$$
\item \label{sgnatgredysetfinite} For all $x\in \X$ and $m\in \N$, there is $y\in \X$ with $\supp(y)\subset \Lambda_m(x)$ such that
$$
\|x-y\|\le \C \inf_{\substack{z\in \X\\|\supp(z)|<\infty\\ w(\supp(z))\le w(\Lambda_m)}}\|x-z\|.    
$$
\item \label{sgnsomegredyset} For all $x\in \X$ and $m\in \N$, there is $A\in \G(x,m,1)$ and $y\in \X$ with $\supp(y)\subset A$ such that
$$
\|x-y\|\le \C \inf_{\substack{z\in \X\\ w(\supp(z))\le w(A)}}\|x-z\|.
$$
\item \label{sgnsomegredysetfinite}  For all $x\in \X$ and $m\in \N$, there is $A\in \G(x,m,1)$ and $y\in \X$ with $\supp(y)\subset A$ such that
$$
\|x-y\|\le \C \inf_{\substack{z\in \X\\|\supp(z)|<\infty\\ w(\supp(z))\le w(A)}}\|x-z\|.     
$$
\end{enumerate}
\end{lemma}

\begin{proof}The implications
\ref{sgnallgredyset}$ \Longrightarrow$ \ref{sgnatgredyset} $\Longrightarrow$ \ref{sgnsomegredyset} and \ref{sgnallgredysetfinite} $\Longrightarrow$ \ref{sgnatgredysetfinite} $\Longrightarrow$ \ref{sgnsomegredysetfinite}, as well as \ref{sgnallgredyset} $\Longrightarrow$ \ref{sgnallgredysetfinite}, together with \ref{sgnatgredyset} $\Longrightarrow$ \ref{sgnatgredysetfinite} and \ref{sgnsomegredyset} $\Longrightarrow$ \ref{sgnsomegredysetfinite} are immediate. 
\\
Let us prove \ref{sgnsomegredyset} $\Longrightarrow$ \ref{sgnallgredyset}: Fix $A\in \G(x,m,1)$. We may assume $x\not=P_A(x)$ (else, we take $y=x$), so $|\supp(x)|>m$. For each $l\in \N$ choose $z_l\in \X$ with $w(\supp(z_l))\le w(A)$ so that
$$
\|x-z_l\|\le (1+\frac{1}{l}) \inf_{\substack{z\in \X\\ w(\supp(z))\le w(A)}}\|x-z\|,
$$
which is possible because $A\subsetneq \supp(x)$ so the infimum above is not zero.  
\\
For each $n \in \N$, set $x_{n}:=x+ \frac{1}{n}P_A(x)$. As $\G(x_n,m,1)=\{A\}$ for each $n\in \N$, for each $l\in \N$ there is $y_{n,l}$ with $\supp(y_{n,l})\subset A$ such that
$$
\|x_n-y_{n,l}\|\le \C\|x_n-z_l\|. 
$$
Given that $A\in \N^{<\infty}$, for fixed $l\in \N$ there is  $y_l$ with $\supp(y_l)\subset A$ and a subsequence $(y_{n_{k,l},l})_{k\in \N}$ convergent to $y_l$. Letting $k\to \infty$
in 
$$
\|x_{n_{k,l}}-y_{n_{k,l},l}\|\le \C\|x_{n_{k,l}}-z_l\|,
$$
we obtain 
$$
\|x-y_l\|\le \C\|x-z_l\|\le \C (1+\frac{1}{l}) \inf_{\substack{z\in \X\\ w(\supp(z))\le w(A)}}\|x-z\|.
$$
Reasoning as before and taking a subsequence if necessary, we may assume that  $(y_l)_{l\in \N}$ is convergent to some $y$ with $\supp(y)\subset A$,  so we complete the step letting $l\rightarrow \infty$. \\
The implication \ref{sgnsomegredysetfinite} $\Longrightarrow$ \ref{sgnallgredysetfinite}  is proven by the same argument as that given above to prove \ref{sgnsomegredyset} $\Longrightarrow$ \ref{sgnallgredyset}.  \\
Finally, we show that \ref{sgnatgredysetfinite} $\Longrightarrow$ \ref{sgnatgredyset}. Fix $x\in \X$ and $m\in \N$. Suppose there is $z\in \X$ with $|\supp(z)|=\infty$ and $w(\supp(z))\le w(\Lambda_m(x))$. Then by Remark~\ref{remarkprojectionc0}, given $\epsilon>0$ there is a finite  set $B\subset \supp(z)$ such that 
$$
\|z-P_B(z)\|\le \epsilon. 
$$
It follows that 
$$
\inf_{\substack{z\in \X\\ w(\supp(z))\le w(\Lambda_m)}}\|x-z\|=\inf_{\substack{z\in \X\\|\supp(z)|<\infty\\ w(\supp(z))\le w(\Lambda_m)}}\|x-z\|, 
$$
so the proof is complete. 
\end{proof}

It was shown in \cite[Theorem 4.3]{DKTW2018} that every $\w$-almost greedy Schauder basis is $\w$-semi greedy (see also \cite[Theorem 1.11]{B2020}, which improves the bound for the $\w$-semi-greedy constant).  Both proofs are valid for Markushevich basis. The implication from $\w$-semi-greedy to $\w$-almost greedy was first proven for Schauder bases in spaces with finite cotype in \cite[Theorem 4.15]{DKTW2018}. The cotype condition in \cite[Theorem 4.15]{DKTW2018} was later removed in \cite[Theorem 1.11]{B2020}. In these papers, $\w$-almost greedy bases were characterized as those being quasi-greedy and $\w$-superdemocratic  or $\w$-disjoint-superdemocratic. Below, we find hypothesis weaker than the Schauder condition under which  $\w$-semi-greedy bases are $\w$-almost greedy. Additionally, we  prove that all $\w$-semi-greedy Markushevich bases are $\w$-superdemocratic and  that they have a property, called ``Property (C)'', with roots in \cite[Lemma 2.2]{DKKT2003} that was studied, for instance, in \cite{B2020} and \cite{BDKOW2019}. Before we go on, we give the relevant definitions, as well as some related notions. 

\begin{definition} Let $\B$ be a basis for $\X$, $\w$ a weight and $\C>0$.  Then: 
\begin{itemize}
\item $\B$ is \emph{weight superdemocratic with weight $\w$ and constant $\C$} (or $\C$-$\w$-superdemocratic) if 
\begin{align}
\|\1_{\varepsilon,A}\|\le \C \|\1_{\varepsilon',B}\|,\qquad \forall A, B \in \N^{<\infty},  w(A)\le w(B), \varepsilon\in \EE_A, \varepsilon'\in \EE_B. \nonumber
\end{align}
\item $\B$ is \emph{weight disjoint superdemocratic with weight $\w$ and constant $\C$} (or $\C$-$\w$-disjoint superdemocratic) if the above holds for  $A$ and $B$ disjoint sets. 
\end{itemize}
When  taking $\1_A$ and $\1_B$ instead of $\1_{\varepsilon,A}$ and $\1_{\varepsilon',B}$, the basis $\B$ is 
$\w$-democratic (see  \cite{DKTW2018, KPT2006}) and $\w$-disjoint democratic, respectively.
\end{definition}

\begin{definition}Let $\B$ be a basis for $\X$ with dual basis $\B^*=(\xx_i^*)_{i\in \N}$. We say that $\B$  \emph{has Property (C) with constant $\K>0$} if 
\begin{align}
\min_{i\in A}|\xx_i^*(x)|\|\1_{\varepsilon,A}\|\le \K\|x\|,\qquad\forall x\in \X,\, \forall A\in \N^{<\infty},\, \forall \varepsilon\in \EE_{A}.\label{propertyC}
\end{align}
\end{definition}

When \eqref{propertyC} holds only for $\varepsilon=\varepsilon(x)$, this property has been studied under the name ``truncation-quasi-greediness'' (see \cite{AABBL2021} and \cite{AABBL2021b}), which reflects the fact that the restricted truncation operator is bounded (see \cite{AABW2021}). These two definitions are equivalent (see \cite[Proposition 4.16]{AABW2021}). We will keep the latter terminology, though some of our proofs establish also upper bounds for the constant in \eqref{propertyC}. 

\begin{remark}\rm  \label{remarkpropc=tqg} 
Note that if $\B$ is $\C$-truncation quasi-greedy, it has the  $2\kappa \C^2$-property (C). Indeed, given $x\in \X$ with $\xx_i^*(x)\in \RR_{\ge 0}$ for all $i$,  $A\in \G(x,m,1)$ and $\varepsilon \in \EE_A$, then 
\begin{align*}
\min_{i\in A}|\xx_i^*(x)|\|\1_{\varepsilon, A}\|\le& \min_{i\in A}|\xx_i^*(x)|\kappa \max_{\substack{\varepsilon'\in \EE_A\\\varepsilon'_j\in \{-1,1\} \forall j}}\|\1_{\varepsilon', A}\|\le 2\kappa \min_{i\in A}|\xx_i^*(x)|\max_{B\subset A}\|\1_{B}\|\\
\le& 2\kappa \min_{i\in A}|\xx_i^*(x)| \C\|\1_{A}\|\le 2\kappa\C^2\|x\|.
\end{align*}
The general case follows because for every $x\in \X$, the basis $\B_x=(\yy_i:=\varepsilon_i \xx_i)$  where the $\varepsilon_i$ are chosen so that $\yy_i^*(x)\in \RR_{\ge 0}$ for all $i$ is also $\C$-truncation quasi-greedy. \\
Note that the above argument also shows that if $\B$ is $\C$-truncation quasi-greedy, $\|\1_{\varepsilon, A}\|\le 2\kappa \C\|\1_{\varepsilon', A}\|$ for any $A\in \N^{<\infty}$, $\varepsilon,\varepsilon'\in \EE_A$. 
\end{remark}

To prove the implication from $\w$-semi-greedy to $\w$-almost greedy bases, we use the concept of disjoint $\w$-almost greedy bases, which we define in terms of all greedy sets for convenience. 

\begin{definition}\label{definitiondisjointweightalmostgreedyetc}
Let $\B$ be a basis for $\X$,  $\w$ a weight, and $\C>0$. We say that $\B$ is \emph{weight-disjoint almost greedy with weight $\w$ and constant $\C$} (or $\C$-$\w$-disjoint almost greedy) if 
\begin{equation*}
\|x-P_A(x)\|\le \C   \inf_{\substack{B\in \N^{<\infty}\\ w(B)\le w(A)\\ B\cap A=\emptyset}}\|x-P_B(x)\|,  \qquad\forall x\in \X,\, \forall A \in \G(x,1). 
\end{equation*}
\end{definition}

We will use the following elementary result, which is a weighted variant of \cite[Lemma 6.2]{AABW2021}; we give a proof for the sake of completion.

\begin{lemma}\label{lemmadisjoint=}Let $\B$ be a basis for $\X$, $\w$ a weight and $\C>0$. Then $\B $ is $\C$-$\w$-disjoint almost greedy if and only if it is $\C$-$\w$-almost greedy. 
\end{lemma} 
\begin{proof}
Suppose $\B$ is $\C$-$\w$-disjoint almost greedy, fix $x\in \X$, $A\in \G(x,1)$, and $B\subset\N$ with $w(B)\le w(A)$. If $B\cap A=\emptyset$ or $A=B$, there is nothing to prove. Else, since $A\setminus B\in \G(x-P_{A\cap B}(x),1)$ and $w(B\setminus A)\le w(A\setminus B)$, we have 
\begin{align*}
\|x-P_A(x)\|=&\|x-P_{A\cap B}(x)- P_{A\setminus B}(x)\|\le \C \|x-P_{A\cap B}(x)- P_{B\setminus A}(x)\|\\
=&\C\|x-P_B(x)\|. 
\end{align*}
\end{proof}

\begin{remark}\rm In the definition of $\w$-almost greedy bases, the projections are taken over finite sets, which guarantees that they are well defined. But if a basis is $\w$-almost greedy, by Remark~\ref{remarkprojectionc0}, we have 
$$
 \inf_{\substack{B\in \N^{<\infty}\\ w(B)\le w(\Lambda_m)}}\|x-P_B(x)\| = \inf_{\substack{B\subset \N\\ w(B)\le w(\Lambda_m)}}\|x-P_B(x)\|,     \qquad\forall x\in \X,\, \forall m\in \N.
$$
Similar considerations hold for the infima taken over all sets $A\in \G(x,m,1)$.
\end{remark}
Next, we define a propety that can be seen as an extension of the weight-semi-greedy property as well as an extension, to the context of weights, of the weak semi-greedy property given in \cite[Definition 1.7]{BL2021} .

\begin{definition}\label{definitionweakweightsemigreedy} Let $\B$ be a basis for $\X$, $\w$ a weight, $\C>0$ and $0<s\le 1$. We say that $\B$ is \emph{weak weight-semi-greedy with parameter $s$, weight $\w$ and constant $\C$}  (or $\C$-$s$-$\w$-semi-greedy) if, for every $x\in \X$ and $m\in \N$, there is $A\in \G(x,m,s)$ and $y\in \X$ with $\supp(y)\subset A$ such that
\begin{equation}
\|x-y\|\le \C\inf_{\substack{z\in \X\\|\supp(z)|<\infty\\w(\supp(z))\le w(A)}}\|x-z\|.\label{weakweightsemigreedy}
\end{equation}
We denote by $\GS(x,m,s)$ the subset of $\G(x,m,s)$ for which the above holds.
\end{definition}

\begin{remark}\rm Note that the set $\GS(x,m,s)$ depends on $\C$, but we will leave that implicit if is no risk of ambiguity; else, we will write $\GS(x,m,s,\C)$. The same consideration applies to later definitions as  is the case with Definition~\ref{definitionweakweightalmostsemigreedy} and Definition~\ref{definitionweakweightalmostgreedy}. 
\end{remark}

It is known that a basis is greedy, quasi-greedy, or almost greedy if and only if the relevant definition holds for elements with finite support (see, for example,  \cite[Lemma 2.3]{Oikhberg2017} for quasi-greedy bases; similar arguments hold for almost greedy or greedy ones). We will show that this is also true for weak weight-semi-greedy bases. First, we need some auxiliary lemmas. 

\begin{lemma}\label{lemmaGxmtfinite}
Let $\B$ be a basis for $\X$ with dual basis $\B^*=(\xx_i)_{i\in \N}$ and $0<t\le 1$. Fix $x\in \X$ and $m\in \N$ such that  $|\supp(x)|\ge m$, then $\G(x,m,t)$ is a finite set and $\xx_i^*(x)\not=0$  for all  $i\in A$ with $A\in \G(x,m,t)$.
\end{lemma}
\begin{proof} 
Let 
$$
B:=\bigcup_{A\in \G(x,m,t)}A\qquad\text{and}\qquad b:=\inf_{i\in B}|\xx_i^*(x)|, 
$$
and pick any $j\in \supp(x)$. Since $\B^*$ is weak star null, there is $i_0\in \N$ such that 
$$
|\xx_i^*(x)|<t |\xx_j^*(x)|,\quad \forall i\ge i_0.
$$
Hence, $B\subset \{1,\dots, i_0\}$ and $\G(x,m,t)$ is a finite set, so $b$ is a minimum.  Now choose $A\in \G(x,m,t)$ and $n\in A$ with $|\xx_n^*(x)|=b$. If $b=0$, choose $j\in \supp(x)\setminus  A$. Then $0=b\ge t |\xx_j^*(x)|>0$, a contradiction. 
\end{proof}


Next, we strengthen \cite[Lemma 2.2]{Oikhberg2017} for Markushevich bases.

\begin{lemma}\label{lemmaallfinite}Let $\B$ be a  basis for $\X$, $x\in \X$, and $\epsilon>0$. The following hold:
\begin{enumerate}[ \rm (i)]
\item \label{lemmaallfinitenogreedy}Given $D\in \N^{<\infty}$, there is $y\in \X$ with finite support such that $\|x-y\|<\epsilon$ and $P_D(x)=P_D(y)$.
\item  \label{lemmaallfinitegreedy} Given $m_0\in \N$ and $0<t\le 1$,  there is $y\in \X$ with finite support such that $\|x-y\|<\epsilon$, and, for each  $1\le m\le m_0$, $\G(x,m,t)=\G(y,m,t)$ and $P_A(y)=P_A(x)$ for each $A\in \G(x,m,t)$. 
\end{enumerate}
\end{lemma}
\begin{proof}
To prove \ref{lemmaallfinitenogreedy} choose a finitely supported $z$ so that $\|x-P_D(x)-z\|< (1+\|P_D\|)^{-1}\epsilon$, and define $y:=z-P_D(z)+P_D(x)$. We have 
$$
\|x-y\|\le \|x-P_D(x)-z\|+\|P_D(z-x+P_D(x)\|\le (1+\|P_D\|)\|x-P_D(x)-z\|<\epsilon. 
$$
To prove \ref{lemmaallfinitegreedy} set $\B^*=(\xx^*_i)_{i\in \N}$ the dual basis of $\B$. We may assume that $x$ has infinite support, so the hypotheses of Lemma~\ref{lemmaGxmtfinite} hold for $m=m_0$. Let 
\begin{align*}
B:=&\bigcup_{A\in \G(x,m_0,t)}A, &&b:=\min_{i\in B}|\xx_i^*(x)|. 
\end{align*}
First note that if $m_0>1$, $1\le m< m_0$ and $A\in \G(x,m,t)$, then $A\cup A_1\in \G(x,m_0,t)$ for every $A_1\in \G(x-P_A(x),m_0-m,1)$. Hence, $A\subset B$. \\
Since $\B^*$ is weak star null, there is $i_0\in \N$ such that $|\xx_i^*(x)|< 2^{-1}tb$ for all $i\ge i_0$. 
Set
\begin{align*}
D:=&\{1,\dots,i_0\}, &&\epsilon_1:=\frac{t \min\{\epsilon, 1, b\}}{2(1+\lambda')},
\end{align*}
and let $y$ be obtained by an application of \ref{lemmaallfinitenogreedy} to $x$, $D$ and $\epsilon_1$. To see that $y$ has the desired properties, first note that for each $n\in D$ and each $k\not\in D$, 
\begin{align}
|\xx_k^*(y)|\le |\xx_k^*(x-y)|+|\xx_k^*(x)|\le \lambda'\|x-y\|+2^{-1}tb< tb\le t|\xx_{n}^*(x)|=t|\xx_{n}^*(y)|.\label{lemmaallfinitegreedyoutofD}
\end{align}
Since $i_0>m_0$, it follows that $\G(y,m,t)\subset D$ for all $1\le m\le m_0$. Now fix $1\le m\le m_0$, and choose  $A\in\G(x,m,t)$, $n\in A$ and $k\not\in A$. If $k\in D$, then 
$$
|\xx_n^*(y)|=|\xx_n^*(x)|\ge t |\xx_k^*(x)|=t|\xx_k^*(y)|, 
$$
whereas if $k\not \in D$, then $t |\xx_k^*(y)|\le |\xx_{n}^*(y)$ by \eqref{lemmaallfinitegreedyoutofD}. Therefore, $A\in \G(y,m,t)$. \\
Similarly, choose $A\in \G(y,m,t)$, $n\in A$ and $k\not\in A$. Since $n\in D$, the case $k\in D$ is handled as before but changing the roles of $x$ and $y$, whereas if $k\not \in D$, then $t|\xx_k^*(x)|<b\le |\xx_n^*(y)|=|\xx_n^*(x)|$. We conclude that $A \in \G(x,m,t)$.
\end{proof}

\begin{remark}\label{remarkpbanach}Note that the proof of Lemma~\ref{lemmaGxmtfinite} holds without changes for $p$-Banach spaces, for $0<p<1$, whereas that of Lemma~\ref{lemmaallfinite} holds as well, with only straightforward modifications: just choose $z$ so that $\|x-P_D(x)-z\|^p< (1+\|P_D\|^p)^{-1}\epsilon^p$, and use $p$-convexity. Thus, for Markushevich bases, our result strengthens \cite[Lemma 7.2]{BB2020} in addition to \cite[Lemma 2.2]{Oikhberg2017}. 
\end{remark}

\begin{remark}\label{remarkfinitesupport} Note also that the totality hypothesis does not play a role in the proofs of Lemmas~\ref{lemmaGxmtfinite} and Lemma~\ref{lemmaallfinite}. Moreover, in the proof of Lemma~\ref{lemmaallfinite}~\ref{lemmaallfinitenogreedy}, even if $\B^*$ is not total we can pick $z\in [\B]$ and obtain $y\in [\B]$, in other words we obtain $y$ as a finite linear combination of the $\xx_i$'s. Thus, if $|\supp(x)|\ge m_0$, we can also get $y\in [\B]$ in Lemma~\ref{lemmaallfinite}~\ref{lemmaallfinitegreedy}. However,  if $x\not\in [\B]$ and $|\supp(x)|<m_0$, without totality we cannot obtain $y\in [\B]$ in Lemma~\ref{lemmaallfinite}~\ref{lemmaallfinitegreedy}, because in that case, for every $A\in \G(x,m_0-1,t)$ and every $n\not\in A$, $\xx_n^*(x)=0$ and $A\cup \{n\}\in \G(x,m_0,t)$.  On the other hand, with only straightforward modifications the proofs of \cite[Lemma 7.2]{BB2020} and \cite[Lemma 2.2]{Oikhberg2017} yield $y\in [\B]$ even when $\B^*$ is not total. 
\end{remark}

Now we can prove the aforementioned equivalence. 

\begin{lemma}\label{lemmafinitesupportisenough}Let $\B$ be a basis for $\X$, $\w$ a weight, $\C>0$ and $0<s\le 1$. Suppose that the conditions of Definition~\ref{definitionweakweightsemigreedy} hold for $x$ with finite support. Then $\B$ is $\C$-$s$-$\w$-semi-greedy. Moreover, the conditions of Definition~\ref{definitionweakweightsemigreedy} hold even if the infimum is taken without the restriction $|\supp(z)|<\infty$.
\end{lemma}
\begin{proof}
Pick $x\in \X\setminus \{0\}$ and $m\in \N$. If $|\supp(x)|<m$, then $\supp(x) \subset A$ for every $A \in \G(x,m,t)$, and there is nothing to prove. Otherwise, for every $n\in \N$, by Lemma~\ref{lemmaallfinite} there is $x_n\in \X$ with finite support such that $\|x-x_n\|\le n^{-1}$ and $\G(x_n,m,t)=\G(x,m,t)$. By hypothesis, for each $n$ there are $A_n\in \G(x,m,t)$ and $y_n\in \X$ with $\supp(y_n)\subset A_n$ such that 
$$
\|x_n-y_n\|\le \C\inf_{\substack{z\in \X\\|\supp(z)|<\infty\\w(\supp(z))\le w(A_n)}}\|x_n-z\|.
$$
By Lemma~\ref{lemmaGxmtfinite}, $\G(x,m,t)$ is finite. Thus, passing to a subsequence, we may assume $A_n=A$, a fixed set. Passing to a further subsequence, we may also assume that there is $y$ supported in $A$ such that $\|y_n-y\|\le n^{-1}$ for all $n$. Now pick $z\in \X$ with $w(\supp(z))\le w(A)$. By Remark~\ref{remarkprojectionc0}, for each $n\in \N$ there is $z_n\in \X$ with finite support contained in $\supp(z)$ such that $\|z_n-z\|\le n^{-1}$. Hence, 
\begin{align*}
\|x-y\|\le&2n^{-1}+\|x_n-y_n\|\le 2n^{-1}+\C\|x_n-z_n\|\le \C\|x-z\|+2n^{-1}(1+\C), \forall n\in \N.
\end{align*}
As this holds for every $n\in \N$, we get $\|x-y\|\le \C\|x-z\|$. Now the proof is completed by taking infimum over all such $z$. 
\end{proof}

Our next result collects some general facts about weak $\w$-semi-greedy, $\w$-democratic and $\w$-superdemocratic bases. In particular, we give upper bounds for the norms of vectors of the form $\|\1_{\varepsilon,A}\|$ in terms of the $\w$-measure of $A$, and for the unconditionality parameter $\k_m=\k_m[\B,\X]$ (also known as \emph{conditionality constant} or parameter), defined by 
$$
\k_m:=\sup_{\substack{A\subset \N\\|A|\le m}}\|P_A\|, 
$$
which is used to measure how far a basis is from being unconditional  (see for example \cite{AAB2021}, \cite{AABBL2021}, \cite{AAGHR2015}, \cite{AAW2020}, \cite{BBG2017}, \cite{BBGHO2018} and \cite{DKO2015}). Below,  we appeal to the constants  $\lambda$ and $\lambda'$  of \eqref{constants}.

\begin{proposition}\label{propositionupperbounds}Let $\B=(\xx_i)_{i\in \N}$ be a $\C$-$s$-$\w$-semi-greedy  basis for $\X$, $\w=(w_i)_{i\in \N}$ a weight, $\C>0$ and $0<s\le 1$,  and let
$$
\C_1:=3\C s^{-1}(1+\lambda \lambda ')\lambda \max\{2\inf_{j\in \N} w_j^{-1},1\}.
$$
The following hold: 
\begin{enumerate}[\upshape (i)]
\item \label{boundwithweight} For every finite set $A\subset \N$ and $\varepsilon \in \EE_{A}$, 
$$
\|\1_{\varepsilon,A}\|\le \C_1\max\{w(A),1\}. 
$$
\item \label{functionals} For every $x^*\in S_{\X^*}$, 
$$
w(\{j\in \N: |x^*(\xx_j)|>\C_1 w_j\})\le 1. 
$$
\item \label{condparameter} For each $m\in \N$, 
$$
\k_m\le \C_1\lambda '\max\{\sup_{\substack{A\subset \N\\|A|\le m}}w(A),1\}.
$$
\end{enumerate}
If $\B$ is $\K$-$\w$-disjoint superdemocratic (in particular, if it is $\K$-$\w$-almost greedy), the above results hold if we replace $\C_1$ with
$$
\K_1:=4\K\lambda \max\{2\inf_{j\in \N}w_j^{-1},1\}.
$$
If $\B$ is $\K$-$\w$-disjoint democratic, then the above holds with $\C_1$ replaced by $2\kappa \K_1$.
\end{proposition}

\begin{proof}
To prove \ref{boundwithweight}, fix $0<\epsilon<1$, and choose $j_0\in \N$ so that 
$$
w_{j_0}^{-1}\le \inf_{j\in \N}w_j^{-1}+\epsilon.
$$
Given $A$ and $\varepsilon$ as in the statement, define the possibly empty set
$$
A_1:=\{i\in A: w_i<2^{-1}w_{j_0} \}. 
$$
We first estimate the norm  $\|\1_{\varepsilon,A\setminus A_1}\|$:
\begin{align}
\|\1_{\varepsilon,A\setminus A_1}\|\le& \lambda |A\setminus A_1|\le 2\lambda w_{j_0}^{-1}w(A)\le 2\lambda (1+\epsilon)\max\{\inf_{j\in \N}w_j^{-1},1\}\max\{w(A),1\}\nonumber\\
\le& \frac{1}{3}(1+\epsilon)\C_1\max\{w(A),1\}. \label{heavierelements}
\end{align}
If $A_1=\emptyset$, as $\epsilon$ is arbitrary there is nothing else to prove. Else, to estimate the norm $\|\1_{\varepsilon,A_1}\|$ choose a partition of $A_1$ as follows: First, pick a set $A_{1,1}\subset A_1$ of maximum cardinality such that $w(A_{1,1})\le w_{j_0}$. If $A_{1,1}\not=A_1$, then pick $A_{1,2}\subset A_1\setminus A_{1,1}$ of maximum cardinality such that $w(A_{1,2})\le w_{j_0}$, and so on. By this procedure, we get a partition of $A_1$ into finitely many sets $\{A_{1,k}\}_{1\le k\le k_1}$ with $w(A_{1,k})\le w_{j_0}$ for all $1\le k\le k_1$.  If $k_1>1$, then by construction, for every $1\le k\le k_1-1$, there is $i\in A_1\setminus A_{1,k}$ such that $w(A_{1,k})+w_i>w_{j_0}$, which implies that $w(A_{1,k})>2^{-1}w_{j_0}$. Thus, 
$$
w(A_1)> \sum\limits_{k=1}^{k_1-1}w(A_{1,k})\ge  2^{-1}(k_1-1)w_{j_0},
$$
so 
\begin{equation}
k_1\le 2w(A_1)w_{j_0}^{-1}+1\le  2w(A)w_{j_0}^{-1}+1.\label{upperboundk1}
\end{equation}
For each $1\le k\le k_1$, define 
$$
z_k:=(1+\epsilon)s^{-1}\xx_{j_0}+\1_{\varepsilon,A_{1,k}}.
$$
As $\G(z_k,1,s)=\{\{j_0\}\}$, there is $b_k\in \KK$ such that
$$
\|z_k-b_k\xx_{j_0}\|\le \C\inf_{\substack{|\supp(z)|<\infty\\w(\supp(z))\le w_{j_0}}}\|z_k-z\|\le \C(1+\epsilon)s^{-1}\|\xx_{j_0}\|.
$$
Hence, by the triangle inequality, 
$$
\|\1_{\varepsilon,A_{1,k}}\|\le \|z_k-b_k\xx_{j_0}\|+\|((1+\epsilon)s^{-1}-b_k)\xx_{j_0}\|\le (1+\epsilon)s^{-1}\C\|\xx_{j_0}\|+ |(1+\epsilon)s^{-1}-b_k|\|\xx_{j_0}\|. 
$$
Let $\B^*=(\xx^*_i)_{i\in \N}$ be the dual basis of $\B$. Since 
$$
|(1+\epsilon)s^{-1}-b_k|=|\xx_{j_0}^{*}(z_k-b_k\xx_{j_0})|\le \|\xx_{j_0}^{*}\|\|z_k-b_k\xx_{j_0}\|\le \C\|\xx_{j_0}^{*}\|(1+\epsilon)s^{-1}\|\xx_{j_0}\| 
$$
we obtain  
$$
\|\1_{\varepsilon,A_{1,k}}\|\le \C(1+\epsilon)s^{-1}(1+\|\xx_{j_0}\|\|\xx_{j_0}^{*}\|)\|\xx_{j_0}\|. 
$$
Using again the triangle inequality and \eqref{upperboundk1}, we get 
\begin{align*}
\|\1_{\varepsilon,A_1}\|\le& (2w(A)w_{j_0}^{-1}+1) \C(1+\epsilon)s^{-1}(1+\|\xx_{j_0}\|\|\xx_{j_0}^{*}\|)\|\xx_{j_0}\|\\
\le& 2\C(1+\epsilon)s^{-1}(1+\lambda \lambda ')\lambda  \max\{2w_{j_0}^{-1}w(A),1\}\\
=&\frac{2}{3} \C_1 (1+\epsilon)(\max\{2\inf_{j\in \N} w_j^{-1},1\})^{-1}\max\{2w_{j_0}^{-1}w(A),1\}\\
\le& \frac{2}{3}\C_1(1+\epsilon)\max\{w(A),1\}.
\end{align*}
Given that $\epsilon$ is arbitrary, the proof of \ref{boundwithweight} is completed combining  the above inequality with \eqref{heavierelements}. \\
Now suppose \ref{functionals} is false, and  choose $x^*\in S_{\X^*}$ for which the result does not hold. Then, there is  $A\subset \N$ finite with $w(A)>1$ such that 
$$
|x^*(\xx_j)|>\C_1 w_j,\qquad\forall j\in A. 
$$
Define $\varepsilon\in \EE_A$ by 
$$
\varepsilon_j:=\frac{|x^*(\xx_j)|}{x^*(\xx_j)}, \qquad \forall j\in A. 
$$
As $w(A)>1$, using \ref{boundwithweight} we get 
$$
\C_1w(A)\ge \|\1_{\varepsilon,A}\|\ge |x^*(\1_{\varepsilon,A})|=\sum_{j\in A}|x^*(\xx_j)|>\sum_{j\in A}\C_1 w_j=\C_1 w(A), 
$$
a contradiction.\\
To prove \ref{condparameter}, fix $x\in \X$, $m\in \N$, and $A\subset \N$ with $|A|\le m$. By \ref{boundwithweight}, 
\begin{align*}
\|P_A(x)\|\le& \|x\|_{\infty}\max_{\varepsilon\in \EE_A}\|\1_{\varepsilon,A}\|\le \lambda '\C_1\max\{w(A),1\}\|x\|. 
\end{align*}
so the proof is completed by taking supremum.\\
Now suppose that $\B$ is $\K$-$\w$-disjoint superdemocratic. Then all of the steps of the above proof hold with the only modification consisting in replacing $\C_1$ with $\K_1$, except for the bounds for $\|\1_{\varepsilon, A\setminus A_{1}}\|$ and $\|\1_{\varepsilon, A_{1}}\|$; we give bounds for these norms as follows: First, 
$$
\|\1_{\varepsilon, A_{1,k}}\|\le \K\|\xx_{j_0}\|\le \K\lambda,  \qquad \forall 1\le k\le k_1. 
$$
Thus, using \eqref{upperboundk1},
$$
\|\1_{\varepsilon, A_{1}}\|\le k_1\K\lambda \le \K\lambda (1+2w(A)w_{j_0}^{-1})\le \frac{1}{2}\K_1(1+2\epsilon)\max\{w(A),1\}. 
$$
On the other hand, arguing as in the proof of \eqref{heavierelements} we obtain
$$
\|\1_{\varepsilon,A\setminus A_1}\|\le  2\lambda (1+\epsilon)\max\{1,\inf_{j\in \N}w_j^{-1}\}\max\{w(A),1\}\le \frac{1}{2}(1+\epsilon)\K_1\max\{w(A),1\}, 
$$
and the result follows by the above inequalities. \\
If $\B$ is $\K$-$\w$-disjoint almost greedy, it is $\K$-$\w$-disjoint superdemocratic by Lemma~\ref{lemmadisjoint=} and \cite[Theorem 1.5]{B2020}.  \\
Finally, if $\B$ is $\K$-$\w$-disjoint democratic, by convexity we have
$$
\|\1_{\varepsilon, A_{1,k}}\|\le  2\kappa \sup_{B\subset A_{1,k}}\|\1_{B}\| \le 2\kappa \K\|\xx_{j_0}\|\le 2\kappa \K\lambda,  \qquad 1\le k\le k_1,
$$
and the rest of the proof is the same as that of the $\K$-$\w$-disjoint superdemocratic case. 
\end{proof}
\begin{remark}\label{remarkquantitativec0}\rm It is known that if $\B$-is $\w$-semi-greedy or $\w$-disjoint superdemocratic and $\w\in \ell_1$, $\B$ is equivalent to the canonical unit vector basis of $\mathtt{c}_{0}$ (see \cite[Proposition 2.3]{B2020},  \cite[Proposition 3.10]{BDKOW2019}, \cite[Proposition 4.5]{DKTW2018}), whereas the result for $\w$-disjoint democratic bases can be obtained via a straightforward modification of the proof of \cite[Proposition 3.10]{BDKOW2019}. Using  Proposition~\ref{propositionupperbounds} we can obtain some quantitative variants of these results, as well as a similar result for weak weight-semi-greedy bases. More precisely, if $x$ is finitely supported, then 
\begin{align}
\|x\|\le& \|x\|_{\infty}\max_{\varepsilon\in \EE_{\supp(x)}}\|\1_{\varepsilon,\supp(x)}\|\le \lambda '\C_2\max\{1,\|\w\|_{1}\}\|x\|\label{newupperbound}
\end{align}
where, using the notation of Proposition~\ref{propositionupperbounds},
\begin{equation*}
\C_2=\begin{cases}
\C_1 & \text{ if $\B$ is $\C$-$s$-$\w$-semi-greedy};\\
\K_1 & \text{ if $\B$ is $\K$-$\w$-disjoint-superdemocratic};\\
2\kappa \K_1 & \text{ if $\B$ is  $\K$-$\w$-disjoint-democratic}.
\end{cases}
\end{equation*}
Since the set of finitely supported elements is dense in $\X$, the bound on the right-hand side of \eqref{newupperbound} holds for any $x\in \X$. 
\end{remark}

\begin{remark}\label{remarkl1wdem}\rm It follows from Remark~\ref{remarkquantitativec0} that Lemma~\ref{lemmal1thenc0}  and Remark~\ref{remarkprojectionc0}  also hold if we replace the $\w$-semi-greedy property by $\w$-democracy. 
\end{remark}

Next, we prove that when $\w\not \in \mathtt{c}_0$, every $\w$-semi-greedy basis is $\w$-almost greedy.  While $\w$-almost greediness entails $\w$-superdemocracy (see \cite{DKTW2018}, \cite{B2020}), we give a direct proof of this result, as the upper bounds for the $\w$-superdemocracy constant might be of interest as well.  We also give an upper bound for the truncation quasi-greedy constant. First we prove an auxiliary result which will allow us to pick adequate $s$-greedy sets.

\begin{lemma}\label{lemmasgset}Let $\B$ be a $\C$-$s$-$\w$-semi-greedy basis for $\X$, $\w$ a weight, $\C>0$ and $0<s\le 1$, and let $\B^*=(\xx^*_i)_{i\in \N}$ be the dual basis of $\B$. For every $x\in \X$ and every nonempty finite set $A\subset \N$, there is $m\in \N$ and $B\in \GS(x,m,s)$ such that $A\subset B$ and for all $j\in B$,
$$
|\xx_j^*(x)|\ge s^{2} \min_{i\in A}|\xx_i^*(x)|.
$$
\end{lemma}
\begin{proof}
Let $c:=\min_{i\in A}|\xx_i^*(x)|$. Clearly we may assume $c>0$. Since $\B^*$ is weak star null, there is $n_0\in \N$ such that for each $n\ge n_0$, every set in $\G(x,n,s)$ contains $A$. Let 
$$
n_1:=\min\{{n\in  \N}: \exists B\in \GS(x,n,s): B\supset A\},
$$
and choose $B \in \GS(x,n_1,s)$ containing $A$. If $B=A$, there is nothing to prove. 
Otherwise, since $n_1>1$, we can choose $D \in \GS(x,n_1-1,s)$. By the minimality of $n_1$, it follows that 
$$
A\not\subset D.  
$$
Hence, for all $j\in D$,
\begin{equation}
|\xx_j^*(x)|\ge s c.\label{boundforthesmallerset}
\end{equation}
Thus, if there exists $j_0\in  D\setminus B$, it follows that for all $j\in B$,
\begin{equation}
|\xx_j^*(x)|\ge s |\xx_{j_0}^*(x)|\ge s^{2} c.\nonumber
\end{equation}
On the other hand, if $D\subseteq B$, given that $
A\not\subset  D$ and $A\subseteq B$, there is $i_1\in A$ such that
$$
B=D\cup \{ i_1\},
$$
which implies that \eqref{boundforthesmallerset} also holds for all $j\in B$. 
\end{proof}

Before going on, we note that the weak $\w$-semi-greedy and $\w$-disjoint and $\w$-super-democracy properties can be extended from the context of Markushevich bases to that of general minimal systems. Some minor and straightforward modifications are necessary to account for the fact that if the system is not a Markushevich basis, having finite support is not the same as being a finite linear combination of the elements of the system. In the more general context, the proofs of  Proposition~\ref{propositionupperbounds} and Lemma~\ref{lemmasgset} are valid as well. However, we continue working with Markushevich bases as the hypothesis of being a total system is used in the remaining proofs of this section, in particular to guarantee that we can appeal to the separation properties studied in Section~\ref{sectionseparation}.\\

Next, we prove our first case of the implication from weak weight-semi-greedy to weight-almost greedy bases. We will use the following notation: Given finite sets $A,B\subset \N$, we write $A<B$ to mean that $\max(A)<\min(B)$, and for $j\in \N$, we write $j<A$ to mean that $j<\min(A)$. We use similar conventions for ``$>$'', ``$\ge$'' and ``$\le$''. 

\begin{theorem}\label{theoremswemigreedydisjointgreedynotc0}
Let $\B=(\xx_i)_{i\in \N}$ be a $\C$-$s$-$\w$-semi-greedy basis for $\X$, $\w=(w_i)_{i\in \N}$ a weight, $\C>0$ and $0<s\le 1$. Suppose that $\w$ has a subsequence $(w_{i_k})_{k\in \N}$ that is bounded below, that is 
$$
\inf_{k\in \N}w_{i_k}>0. 
$$
Let $\M:=\M_{fs}((\xx_{i_k})_{k\in \N}, \X)$ and let $\B^*=(\xx^*_i)_{i\in \N}$ be the dual basis of $\B$. The following hold:
\begin{enumerate}[\upshape (i)]
\item \label{wdisjointgreedyest1}  For  every $x\in \X$, $m\in \N$, $0<t\le 1$, $A\in \G(x,m,t)$, and $y$ with $w(\supp(y))\le w(A)$ and $\supp(y)\cap A=\emptyset$,
$$
\|x-P_A(x)\|\le  \C\M \max\{1+8  t^{-1}s^{-1} \lambda \lambda ',1+6\C t^{-1}s^{-3} \} \|x-y\|.
$$
Thus, $\B$ is $\w$-almost greedy with constant as above taking $t=1$. 
\item \label{wdisjointgreedyest2a} For every $x\in \X$, every $A\in \N^{<\infty}$ and every $\varepsilon\in \EE_A$, if  
$$
w(A) \le w(\{i\in \N: |\xx_i^*(x)|\ge 1\})
$$
then
$$
\|\1_{\varepsilon,A}\|\le   2s^{-1}\C\M\max\{\lambda \lambda ',2s^{-2}\C \} \|x\|. 
$$
Thus, $\B$ is $\w$-superdemocratic and truncation quasi-greedy, in each case with constant as above. 
\end{enumerate}

\end{theorem}
\begin{proof}
To prove \ref{wdisjointgreedyest1}, choose $0<\epsilon<1$ and let $(\xx_{i_{k_j}})_{j\in \N}$ be a subsequence given by an application of Corollary~\ref{corollaryboth} to $(\xx_{i_k})_{k\in \N}$ and $(\X, \M, \epsilon)$. Fix $x$, $m$, $t$ and $A$ as in the statement, and $y$ such that $w(\supp(y))\le w(A)$ and $\supp(y)\cap A=\emptyset$. We assume first that both $x$ and $y$ have finite support, and we may also assume $x\not=P_A(x)$, so  
$$
a:=\min_{i\in A}|\xx_i^*(x)|
$$
is positive. Pick $i_0>\supp(x)\cup \supp(y)$, and set $
\F:=[ \xx_i: 1\le i\le i_0]$. We will consider two cases: \\
\paragraph{Case 1} Suppose that there is a set $E\subset \{i_{k_j}\}_{j\ge s_{\F,\epsilon}}$ such that $|E|\le 8$ and $2w(A)\le w(E)$. 
Define  
$$
z_1:=x-P_A(x)+a(1+\epsilon) t^{-1}s^{-1}\1_{E }.
$$
Notice that $|\xx_i^*(z_1)|=|\xx_i^*(x)|\le t^{-1}a$ for all $i\not \in E$, so $\G(z_1,|E|,s)=\{E\}$. Hence,  there is $z_2\in \X$ with $\supp(z_2) \subset E$ such that 
$$
\|z_1-z_2\|\le \C \inf_{\substack{|\supp(z)|<\infty\\w(\supp(z))\le w(E)}}\|z_1-z\|. 
$$
Given that $w(\supp(y))+w(A)\le w(E)$, we have
\begin{align*}
\|z_1-z_2\|\le& \C\|z_1+P_A(x)-y\|\le \C\|x-y\|+\C a(1+\epsilon)t^{-1}s^{-1}\|\1_E\|\\
\le&  \C\|x-y\|+8\C a(1+\epsilon)t^{-1}s^{-1} \lambda . 
\end{align*}
Pick any $i\in A$. Since $A\cap \supp(y)=\emptyset$, we have 
$$
a\le |\xx_i^*(x-y)|\le \lambda '\|x-y\|. 
$$
Thus, 
$$
\|z_1-z_2\|\le (\C+8\C(1+\epsilon)t^{-1}s^{-1} \lambda \lambda ')\|x-y\|. 
$$
Given that 
$$
\|x-P_A(x)\|\le (\M+\epsilon)\|z_1-z_2\|, 
$$
it follows that 
\begin{equation}
\|x-P_A(x)\|\le  \C (\M+\epsilon) (1+8(1+\epsilon)  t^{-1}s^{-1} \lambda \lambda ')\|x-y\|.\label{lightsgreedyset}
\end{equation}
\paragraph{Case 2} Suppose there is ${j_0}>s_{\F,\epsilon}$ such that 
\begin{equation}
w(A)>4w_{i_{k_{j}}},\qquad \forall j\ge {j_0}.\label{heavysgreedyset1}
\end{equation}
Choose $l_1, l_2\ge 3$ so that 
\begin{align*}
w\left(\{i_{k_{2({j_0}+d)-1}},i_{k_{2({j_0}+d)}}: 1\le d< l_1\}\right)\le& w(A)\le w\left(\{i_{k_{2({j_0}+d)-1}},i_{k_{2({j_0}+d)}}: 1\le d\le l_1\}\right);\\
w\left(\{i_{k_{2({j_0}+l_1+d)-1}},i_{k_{2({j_0}+l_1+d)}}: 1\le d< l_2\}\right)\le& w(A)\le w\left(\{i_{k_{2({j_0}+l_1+d)-1}},i_{k_{2({j_0}+l_1+d)}}: 1\le d\le l_2\}\right).
\end{align*}
Set
\begin{align*}
&E_{1,1}:=\{i_{k_{2({j_0}+d)-1}}: 1\le d<l_1\};  &&E_{1,2}:=\{i_{k_{2({j_0}+d)}}: 1\le d<l_1\};\\
&E_{2,1}:=\{i_{k_{2({j_0}+l_1+d)-1}}: 1\le d<l_2\};  &&E_{2,2}:=\{i_{k_{2({j_0}+l_1+d)}}: 1\le d<l_2\};\\
&E_{3,1}:=\{i_{k_{2({j_0}+l_1)-1}},i_{k_{2({j_0}+l_1+l_2)-1}}\}; &&E_{3,2}:=\{i_{k_{2({j_0}+l_1)}},i_{k_{2({j_0}+l_1+l_2)}}\};\\
&E:=\bigcup_{b=1}^{3}\bigcup_{d=1}^{2}E_{b,d}.
\end{align*}
It follows from our choices and \eqref{heavysgreedyset1} that 
\begin{align}
\max_{1\le b\le 3}w(E_{b,1}\cup E_{b,2})\le w(A)\le\frac{w(E)}{2}. \label{weightcloseabovebelow}
\end{align}
Now define 
$$
z_3:=x-P_A(x)-as^{-1}t^{-1}(1+\epsilon)(\sum_{b=1}^{3}\1_{E_{b,1}} -\1_{E_{b,2}}). 
$$
As before, we have $|\xx_i^*(z_3)|=|\xx_i^*(x-P_A(x))|\le a t^{-1}$ for all $i\not\in E$, so $\G(z_3,|E|,s)=\{E\}$. Hence, there is $z_4$ with $\supp(z_4)\subset E$ such that 
$$
\|z_3-z_4\|\le \C \inf_{\substack{z\in \X\\w(\supp(z))\le w(E)\\|\supp(z)|<\infty}}\|z_3-z\|. 
$$
Since $w(\supp(y))+w(A)\le 2w(A) \le w(E)$, we have
\begin{align}
\|x-P_A(x)\|\le& (\M+\epsilon)\|z_3-z_4\|\le \C(\M+\epsilon) \|z_3+P_A(x)-y\|\nonumber\\
\le& \C(\M+\epsilon)\|x-y\|+ \C(\M+\epsilon)as^{-1}t^{-1}(1+\epsilon)\|\sum_{b=1}^{3}\1_{E_{b,1}} -\1_{E_{b,2}}\|.\label{leftsidebound}
\end{align}
To estimate $\|\sum_{b=1}^{3}\1_{E_{b,1}} -\1_{E_{b,2}}\|$, set
$$
z_5:=x-y+a s^2(1-\epsilon)(\1_{E_{1,1}}-\1_{E_{1,2}}).
$$
By Lemma~\ref{lemmasgset} there is a set $D\supset A$ with $D\in \GS(z_5,|D|,s)$ such that
$$
\min_{j\in D}|\xx_j^*(z_5)|\ge s^2\min_{i\in A}|\xx_i^*(z_5)|=s^2\min_{i\in A}|\xx_i^*(x)|=s^2a,
$$
which implies that $D\subset \{1,\dots,i_0\}$. Choose $z_6$ with $\supp(z_6)\subset D$ so that
\begin{equation}
\|z_5-z_6\|\le \C \inf_{\substack{z\in \X\\w(\supp(z))\le w(D)\\|\supp(z)|<\infty}}\|z_5-z\|. \nonumber
\end{equation}
Given that $z_6\in \F$, using \eqref{weightcloseabovebelow} and the fact that $w(A)\le w(D)$ we infer that
\begin{align*}
\|a s^2(1-\epsilon)(\1_{E_{1,1}}-\1_{E_{1,2}})\|\le& \|z_5-z_6\|+\|x-y-z_6\|\le (2+\epsilon)\|z_5-z_6\|\\
\le& (2+\epsilon)\C\|z_5-a s^2(1-\epsilon)(\1_{E_{1,1}}-\1_{E_{1,2}})\|\\
=& (2+\epsilon)\C\|x-y\|.
\end{align*}
The same argument gives 
$$
\|a s^2(1-\epsilon)(\1_{E_{b,1}}-\1_{E_{b,2}})\|\le (2+\epsilon)\C\|x-y\,|\qquad\forall 2\le b\le 3. 
$$
Combining these inequalities with \eqref{leftsidebound}, by the triangle inequality we get 
\begin{equation}
\|x-P_A(x)\|\le \C(\M+\epsilon)(1+3(2+\epsilon)\C(1+\epsilon)(1-\epsilon)^{-1}t^{-1}s^{-3})\|x-y\|. \label{heavygreedyset2}
\end{equation}
As $\epsilon$ is arbitrary, a combination of \eqref{lightsgreedyset} and \eqref{heavygreedyset2} gives \ref{wdisjointgreedyest1} for $x, y\in \X$ with finite support. Now choose again $x,m, t, A$ and $y$ as in the statement, $x\not=P_A(x)$, fix $\delta>0$, and let $\K:=\C\M \max\{1+8  t^{-1}s^{-1} \lambda \lambda ',1+6\C t^{-1}s^{-3} \}$. By Lemma~\ref{lemmaallfinite}, there is $x_1\in \X$ with finite support such that $A\in \G(x_1,m,t)$, $P_A(x_1)=P_A(x)$, and $\|x-x_1\|\le \delta$. Also, Remark~\ref{remarkprojectionc0} gives $y_1\in \X$ with finite support such that $\supp(y_1)\subset \supp(y)$ and $\|y-y_1\|\le \delta$. Applying the result for vectors with finite support, we obtain 
\begin{align*}
\|x-P_A(x)\|\le&\|x_1-x\|+\|x_1-P_A(x_1)\|\le \delta+\K\|x_1-y_1\|\\
\le& \delta+\K\|x-y\|+\K\|x_1-x\|+\K\|y-y_1\|\le 3\delta+\K\|x-y\|.
\end{align*}
Since $\delta$ is arbitrary, the proof of \ref{wdisjointgreedyest1} is complete. 
\\
To prove \ref{wdisjointgreedyest2a}, choose $0<\epsilon<1$, let $(\xx_{i_{k_j}})_{j\in \N}$ be as in the proof of  \ref{wdisjointgreedyest1}, and let $A$, $\varepsilon$ and $x$ be as in the statement, and suppose $x$ has finite support. Choose $i_0>A\cup \supp(x)$, and set 
$$
\F:=[\xx_i: 1\le i\le i_0]. 
$$
We will consider again two cases: \\
\paragraph{Case 1} Suppose there are $j_2>j_1> s_{\F,\epsilon}$ such that 
$$
w(A)\le  w_{i_{k_{j_1}}}+w_{i_{k_{j_2}}}.
$$
Set
$$
E:=\{i_{k_{j_1}}, i_{k_{j_2}}\} \qquad \text{and} \qquad z_1:=\1_{\varepsilon, A} +(1+\epsilon) s^{-1}\1_{E }.
$$
Since $\G(z_1,2,s)=\{E\}$, there is $z_2\in \N$ with $\supp(z_2)\subset E$ such that 
$$
\|z_1-z_2\|\le \C \inf_{\substack{|\supp(z)|<\infty\\w(\supp(z))\le w(E)}}\|z_1-z\|.
$$
Hence, 
\begin{equation}
\|\1_{\varepsilon,A}\|\le (\M+\epsilon)\|z_1-z_2\|\le (\M+\epsilon) \C (1+\epsilon) s^{-1}\|\1_{E}\|.\nonumber
\end{equation}
Given that 
$$
\|\1_E\|\le 2\lambda \le 2\lambda \lambda '\|x\|, 
$$
we obtain 
\begin{equation}
\|\1_{\varepsilon,A}\|\le  2(\M+\epsilon)\C(1+\epsilon)s^{-1}\lambda \lambda '\|x\|. \label{lightsupport7}
\end{equation}
\paragraph{Case 2} Suppose that there is ${j_0}>s_{\F,\epsilon}$ such that 
$$
w(A)\ge  2 w(i_{k_{j}}),\qquad \forall j\ge {j_0}, 
$$
and choose $l_1\ge 2$ so that 
$$
w\left(\{i_{k_{2({j_0}+d)-1}},i_{k_{2({j_0}+d)}}: 1\le d< l_1\}\right)\le w(A)\le w\left(\{i_{k_{2({j_0}+d)-1}},i_{k_{2({j_0}+d)}}: 1\le d\le l_1\right\}).
$$
Define 
\begin{align*}
&E_{1,1}:=\{i_{k_{2({j_0}+d)-1}}: 1\le d<l_1\};  &&E_{1,2}:=\{i_{k_{2({j_0}+d)}}: 1\le d<l_1\};\\
&E_{2,1}:=\{i_{k_{2({j_0}+l_1)-1}} \}; &&E_{2,2}:=\{i_{k_{2({j_0}+l_1)}}\};\\
&E:=\bigcup_{b=1}^{2}\bigcup_{d=1}^{2}E_{b,d}.
\end{align*}
Note that $\max\{w(E_{1,1}\cup E_{1,2}),w(E_{2,1}\cup E_{2,2})\}\le w(A)\le w(E)$. Let
$$
B:=\{i\in \N: |\xx_i^*(x)|\ge 1\}\qquad\text{and}\qquad z_1:=x + s^2(1-\epsilon)(\1_{E_{1,1}}-\1_{E_{1,2}}).
$$
By Lemma~\ref{lemmasgset}, there is a set $D\supset B$ with $D\in \GS(z_1,|D|,s)$ and a vector $z_2$ with$\supp(z_2)\subset D$ such that
\begin{equation}
\min_{j\in D}|\xx_j^*(x)|\ge s^2,\nonumber
\end{equation}
and 
\begin{equation}
\|z_1-z_2\|\le \C \inf_{\substack{z\in \X\\w(\supp(z))\le w(D)\\|\supp(z)|<\infty}}\|z_1-z\|. \nonumber
\end{equation}
Since $D\subset\{1,\dots,i_0\}$,  considering that $w(D)\ge w(B)\ge w(A)\ge w(E_{1,1})+w(E_{1,2})$, we have
\begin{align*}
\| s^2(1-\epsilon)(\1_{E_{1,1}}-\1_{E_{1,2}})\|\le& \|z_1-z_2\|+\|x-z_2\|\le (2+\epsilon)\|z_1-z_2\|\\
\le& (2+\epsilon)\C\|x\|.
\end{align*}
The same argument gives 
$$
\| s^2(1-\epsilon)(\1_{E_{2,1}}-\1_{E_{2,2}})\|\le (2+\epsilon)\C\|x\|. 
$$
To finish the proof, define
$$
z_3:= \1_{\varepsilon, A}+ s^{-1}(1+\epsilon)(\sum_{b=1}^{2}\1_{E_{b,1}} -\1_{E_{b,2}}). 
$$
Since $\G(z_3,|E|,s)=\{E\}$, there is $z_4$ with $\supp(z_4)\subset E$ such that 
$$
\|z_3-z_4\|\le \C \inf_{\substack{z\in \X\\w(\supp(z))\le w(E)\\|\supp(z)|<\infty}}\|z_3-z\|. 
$$
Therefore, 
\begin{align*}
\|\1_{\varepsilon, A}\|\le&  (\M+\epsilon)\|z_3-z_4\|\le \C(\M+\epsilon)  s^{-1}(1+\epsilon)\|\sum_{b=1}^{2}\1_{E_{b,1}} -\1_{E_{b,2}}\|\\
\le&  2\C^2(\M+\epsilon)  s^{-3}(1+\epsilon)(1-\epsilon)^{-1}(2+\epsilon)\|x\|.
\end{align*}
Now the proof is completed combining the above result with \eqref{lightsupport7}, and letting $\epsilon$ tend to zero. 
\end{proof}

We have yet to consider the case of $\w\in \mathtt{c}_0\setminus\ell_1$, which (to us) is more intricate.  The main obstacle is that we are not able to take a separating sequence of $\B$ given by an application of Proposition~\ref{propositionseparation} or Corollary~\ref{corollaryboth} to prove that (weak) $\w$-semi-greedy bases are $\w$-almost greedy, because we cannot guarantee that there is one such sequence $(\xx_{i_k})_{k\in\N}$ with $\w\left(\{i_k\}_{k\in\N}\right)=\infty$, and if the $\w$-measure were finite, there would be greedy sets with arbitrarily greater $\w$-measure than the entire sequence, precluding the kind of approximation we have used in our proofs so far.  Even so, if $\B^*$ is $r$-norming for some $0<r\le 1$, we can still prove that weak $\w$-semi-greedy bases are $\w$-almost greedy. Our next result handles this case. 

\begin{theorem}\label{theoremc0notl1norming}
Let $\B$ be $\C$-$s$-$\w$-semi-greedy a basis for $\X$, $\w=(w_i)_{i\in \N}$ a weight, $\C>0$, and $0<s\le 1$. Suppose that $\w$ has a subsequence $(w_{i_k})_{k\in \N}\in \mathtt{c_0}\setminus \ell_1$ and that there is $\M>0$ such that 
$$
\M_{fs}((v_k)_{k\in \N},\X)\le \M,
$$
for every block basis $(v_k)_{k\in \N}$ of $\B$ with the property that $(w(\supp(v_k)))_{k\in\N}$ is bounded. Then, the following hold: 
\begin{enumerate}[\upshape (i)]
\item \label{wdisjointgreedyest3}  For all $x\in \X$, $m\in \N$, $0<t\le 1$, $A\in \G(x,m,t)$, and $y\in \X$ with $w(\supp(y))\le w(A)$,
$$
\|x-P_A(x)\|\le  \C  \M(1+6\C t^{-1}s^{-3}) \|x-y\|.
$$
In particular, $\B$ is $\w$-almost greedy with constant as above taking $t=1$. 
\item \label{wdisjointgreedyest4} {Let $\B^*=( \xx^*_i)_{i\in\N}$  be the dual basis of $\B$.} For every nonempty finite set $A\subset \N$, every $\varepsilon\in \EE_A$ and every $x\in \X$, if  
$$
 w(A) \le w(\{i\in \N: |\xx_i^*(x)|\ge 1\}) 
$$
then 
$$
\|\1_{\varepsilon,A}\|\le   4s^{-3}\C^2\M \|x\|. 
$$
\end{enumerate}
Thus, $\B$ is $\w$-superdemocratic and truncation quasi-greedy, in each case with constant as above. \\
In addition, if $\B^*$ is $r$-norming for some $0<r\le 1$, the above conditions hold with $\M=r^{-1}$.
\end{theorem}

\begin{proof}
The statement for the case of an $r$-norming dual basis follows from Proposition~\ref{propositionseparation}, so we need to prove 
\ref{wdisjointgreedyest3}  and \ref{wdisjointgreedyest4}. To prove the former, fix $x$, $m$, $t$ and $A$ as in the statement, and $y$ with $w(\supp(y))\le w(A)$. We may  assume $x\not=P_A(x)$, and we will also assume that both $x$ and $y$ have finite support. 
Set 
$$
a:=\min_{i\in A}|\xx_i^*(x)|, 
$$
choose $i_0>\supp(x)\cup \supp(y)$, and set  $
\F:=[ \xx_i: 1\le i\le i_0]$. Given that $(w_{i_k})_{k\in \N}\in \mathtt{c_0}\setminus \ell_1$, there is a sequence of finite sets of positive integers $(A_k)_{k\in \N}$ such that for all $k\in \N$, $i_0<A_k<A_{k+1}$ and
\begin{equation}
\frac{w(A)}{3}\le w(A_k)\le \frac{w(A)}{2}. \label{part1sets}
\end{equation}
By hypothesis $\M_{fs}((\1_{A_k})_{k\in \N},\X)\le \M$. To simplify our notation, we may assume that $(\1_{A_k})_{k\in \N}$ is already a separating sequence for $(\X,\M,\epsilon)$. By Proposition~\ref{propositionupperbounds}, $(\1_{A_k})_{k\in \N}$ is bounded. Since it is also $\lambda '^{-1}$-uniformly discrete, we may choose $0<\epsilon<1$ and apply Lemma~\ref{lemmaxi-x2} to $(\1_{A_k})_{k\in \N}$. Again, we assume that $(\1_{A_k})_{k\in \N}$ is already the subsequence given by the lemma. Set 
$$
\LL:=\{k\in 5\N: k> \max\{r_{\F,(\1_{A_n})_{n\in \N},\epsilon}, j_{\F,(\1_{A_n})_{n\in \N},\epsilon}, \}\},
$$
and for every $k\in \LL$, define
\begin{align*}
z_{1,k}:=&x-P_A(x)+s^{-1}t^{-1}a(1+\epsilon)\sum_{l=1}^{3}(\1_{A_{2k+2l-1}}-\1_{A_{2k+2l}});\\
z_{2,k,l}:=&x-y-s^2(1-\epsilon)a(\1_{A_{2k+2l-1}}-\1_{A_{2k+2l}}), \qquad\forall 1\le l\le 3.
\end{align*}
Fix $k$ as above and $1\le l\le 3$. By Lemma~\ref{lemmasgset}, there is a set $D\supset A$ such that $D\in \GS(z_{2,k,l},|D|,s)$ and 
$$
\min_{j\in D}|\xx_j^*(z_{2,k,l})|\ge s^2\min_{j\in A}|\xx_j^*(z_{2,k,l})|=s^2\min_{j\in A}|\xx_j^*(x)|=s^2 a. 
$$
It follows that $D\subset \{1,\dots,i_0\}$, so there is $z_{3,k,l}\in \F$ such that 
$$
\|z_{2,k,l}-z_{3,k,l}\|\le \C \inf_{\substack{z\in \X\\|\supp(z)|<\infty\\w(\supp(z))\le w(D)}}\|z_{2,k,l}-z\|. 
$$
As $D\supset A$, using the above inequality and \eqref{part1sets} we get 
$$
\|z_{2,k,l}-z_{3,k,l}\|\le \C\|x-y\|.  
$$
Hence, 
\begin{align}
\|s^2(1-\epsilon)a(\1_{A_{2k+2l-1}}-\1_{A_{2k+2l}})\|\le& \|z_{2,k,l}-z_{3,k,l}\|+\|x-y-z_{3,k,l}\|\nonumber\\
\le& (2+\epsilon) \|z_{2,k,l}-z_{3,k,l}\|\le \C(2+\epsilon)\|x-y\|.\label{firstleftboundikl}
\end{align}
Now we consider $z_{1,k}$: set 
$$
B_k:=\bigcup_{l=1}^{3}A_{2k+2l-1}\cup A_{2k+2l}.
$$
Notice that $\G(z_{1,k},|B_k|,s)=\{B_k\}$. Hence, there is $z_{4,k}$ with $\supp(z_{4,k})\subset B_k$ such that 
$$
\|z_{1,k}-z_{4,k} \|\le \C \inf_{\substack{z\in \X\\|\supp(z)|<\infty\\w(\supp(z))\le w(B_k)}}\|z_{1,k}-z\|. 
$$
By \eqref{part1sets}, $2w(A)\le w(B_k)$. Since $w(\supp(y))\le w(A)$, this gives 
\begin{align}
\|z_{1,k}-z_{4,k} \|& \le \C \|z_{1,k}+P_A(x)-y\| \nonumber \\
& \le \C\|x-y\|+\C a(1+\epsilon)s^{-1}t^{-1}\|\sum_{l=1}^{3}(\1_{A_{2k+2l-1}}-\1_{A_{2k+2l}})\|.\label{firstrightboundikl}
\end{align}
Note that the sequence 
$$
\big(u_k:=a(1+\epsilon)s^{-1}t^{-1}\sum_{l=1}^{3}(\1_{A_{2k+2l-1}}-\1_{A_{2k+2l}})-z_{4,k}\big)_{k\in \LL}
$$
has the finite dimensional separation property with constant $\le \M$. Indeed, if $u_k=0$ for infinitely many values of $k$, the sequence has this property with constant $1$, whereas if this is not the case, there is $k_0\in \LL$ such that the subsequence beginning in $k_0$ is a block basis of $\B$, so we have this bound by hypothesis.  In particular,  it follows that there is $k\in \LL$ such that 
\begin{align*}
\|x-P_A(x)\|\le& (\M+\epsilon)\|x-P_A(x) +a(1+\epsilon)s^{-1}t^{-1}\sum_{l=1}^{3}(\1_{A_{2k+2l-1}}-\1_{A_{2k+2l}})-z_{4,k}\|\\
=&(\M+\epsilon)\|z_{1,k}-z_{4,k}\|,
\end{align*}
which, when combined with \eqref{firstleftboundikl}, \eqref{firstrightboundikl} and the triangle inequality gives 
$$
\|x-P_A(x)\|\le \C(\M+\epsilon)(1+3\C(1+\epsilon)(1-\epsilon)^{-1}(2+\epsilon)t^{-1}s^{-3})\|x-y\|. 
$$
As $\epsilon$ is arbitrary, this completes the proof of \ref{wdisjointgreedyest3} for $x$ and $y$ with finite support, and the general case is proven by the argument given in the proof of Theorem~\ref{theoremswemigreedydisjointgreedynotc0}. \\
The proof of \ref{wdisjointgreedyest4} is similar: Fix $A$, $\varepsilon$ and $x$ as in the statement, set 
$$
B:=\{i\in \N: |\xx_i^*(x)|\ge 1\}, 
$$
choose $0<\epsilon<1$ and $i_0>A\cup \supp(x)$. Now choose a sequence of sets of positive integers $(A_k)_{k\in \N}$ so that for all $k$, $i_0<A_k<A_{k+1}$ and
\begin{equation}
\frac{w(A)}{4}\le w(A_k) \le \frac{w(A)}{2}.\label{part2sets}
\end{equation}
As before, we assume that $(\1_{A_k})_{k\in\N}$ is already a separating sequence for $(\X,\M,\epsilon)$, and that we have applied Lemma~\ref{lemmaxi-x2}. Set
\begin{align*}
\F:=&[ \xx_i: 1\le i\le i_0];\\
\LL:=&\{k\in 5\N: k >\max\{r_{\F,(\1_{A_n})_{n\in \N},\epsilon}, j_{\F,(\1_{A_n})_{n\in \N},\epsilon},\} \}.
\end{align*}
For each $k\in \LL$, define 
\begin{align*}
z_{1,k}:=&\1_{\varepsilon,A}+s^{-1} (1+\epsilon)\sum_{l=1}^{2}(\1_{A_{2k+2l-1}}-\1_{A_{2k+2l}});\\
z_{2,k,l}:=&x-s^2(1-\epsilon) (\1_{A_{2k+2l-1}}-\1_{A_{2k+2l}}),\qquad\forall 1\le l\le 2.
\end{align*}
Fix $k\in \LL$ and $1\le l\le 2$. By Lemma~\ref{lemmasgset}, there is a set $D\supset B$ such that $D\in \GS(z_{2,k,l},|D|,s)$ and 
$$
\min_{j\in D}|\xx_j^*(z_{2,k,l})|\ge s^2\min_{j\in B}|\xx_j^*(z_{2,k,l})|=s^2\min_{j\in B}|\xx_j^*(x)|\ge s^2,
$$
which implies that $D\subset \{1,\dots,i_0\}$. Hence, there is $z_{3,k,l}\in \F$ such that 
$$
\|z_{2,k,l}-z_{3,k,l}\|\le \C \inf_{\substack{z\in \X\\|\supp(z)|<\infty\\w(\supp(z))\le w(D)}}\|z_{2,k,l}-z\|. 
$$
As $D\supset B$, using the above inequality, \eqref{part2sets} and the fact that $w(A)\le w(B)$ we get 
$$
\|z_{2,k,l}-z_{3,k,l}\|\le \C\|x\|.  
$$
Thus, the property of Lemma~\ref{lemmaxi-x2} gives
\begin{align}
\|s^2(1-\epsilon)(\1_{A_{2k+2l-1}}-\1_{A_{2k+2l}})\|\le& \|z_{2,k,l}-z_{3,k,l}\|+\|x-z_{3,k,l}\|\nonumber\\
\le& (2+\epsilon) \|z_{2,k,l}-z_{3,k,l}\|\le \C(2+\epsilon)\|x\|.\label{firstleftboundikl2}
\end{align}
For fixed $k\in \LL$, set
$$
B_k:=\bigcup_{l=1}^{2}A_{2k+2l-1}\cup A_{2k+2l}.
$$
Note that $\G(z_{1,k},|B_k|,s)=\{B_k\}$. Thus, there is $z_{4,k}$ with support contained in $B_k$ such that 
$$
\|z_{1,k}-z_{4,k} \|\le \C \inf_{\substack{z\in \X\\|\supp(z)|<\infty\\w(\supp(z))\le w(B_k)}}\|z_{1,k}-z\|. 
$$
By \eqref{part2sets}, $w(A)\le w(B_k)$, so
\begin{equation}
\|z_{1,k}-z_{4,k} \|\le \C \|z_{1,k} -\1_{\varepsilon,A} \|=\C (1+\epsilon)s^{-1}\|\sum_{l=1}^{2}(\1_{A_{2k+2l-1}}-\1_{A_{2k+2l}})\|.\label{firstrightboundikl2}
\end{equation}
As before, the sequence 
$$
\big((1+\epsilon)s^{-1}\sum_{l=1}^{2}(\1_{A_{2k+2l-1}}-\1_{A_{2k+2l}})-z_{4,k}\big)_{k\in \LL} 
$$
has the finite dimensional separation property with constant $\le \M$. In particular, there is $k\in \LL$ such that 
\begin{align*}
\|\1_{\varepsilon,A}\|\le& (\M+\epsilon)\|z_{1,k}-z_{4,k}\|
\end{align*}
which, when combined with \eqref{firstleftboundikl2}, \eqref{firstrightboundikl2} and the triangle inequality gives 
$$
\|\1_{\varepsilon,A}\|\le 2\C^2(\M+\epsilon)(1+\epsilon)(1-\epsilon)^{-1}(2+\epsilon)s^{-3}\|x\|.
$$
As $\epsilon$ is arbitrary, the proof is complete. 
\end{proof}

It remains to study the general case $\w\in \mathtt{c}_0\setminus \ell_1$. We do not know whether all (weak) $\w$-semi-greedy bases are quasi-greedy, but we can prove that they are  $\w$-superdemocratic and truncation quasi-greedy. To do so, we only need to address the cases that do not meet the conditions of Theorem~\ref{theoremswemigreedydisjointgreedynotc0} or Theorem~\ref{theoremc0notl1norming}. The following result covers all such cases. 

\begin{proposition}\label{propositionc0notl1} Let $\B$ be a $\C$-$s$-$\w$-semi-greedy basis for $\X$, $\w$ a weight, $\C>0$ and $0<s\le 1$. Let $\B^*=(\xx^*_i)_{i\in \N}$ be the dual basis of $\B$ and suppose that $\B$ has a seminormalized block basis $(v_k)_{k\in \N}$ with $(w(\supp(v_k))_{k\in \N}$ bounded that is not weakly null. Then, there is $\K>0$ such that 
\begin{enumerate}[\upshape (i)]
\item \label{anotherone}For all $x\in \X$ and all  $A\in \N^{<\infty}$, 
$$
\min_{i\in A}|\xx_i^*(x)| w(A)\le  \K \|x\|.
$$
\item \label{normlikeweight}For all  $A\in \N^{<\infty}$ and all $\varepsilon\in \EE_{A}$, 
$$
\max\{w(A),1\}\K^{-1}\le \|\1_{\varepsilon,A}\|\le \K\max\{w(A),1\}.
$$
\item \label{propC}For all $x\in \X$, all  $A\in \N^{<\infty}$ and all $\varepsilon\in \EE_{A}$, 
$$
\min_{i\in A}|\xx_i^*(x)|\|\1_{\varepsilon,A}\|\le \K\|x\|.
$$
\item \label{superdem}For all $A, B\in \N^{<\infty}$ with $w(A)\le w(B)$, and all  $\varepsilon\in \EE_{A}$,  $\varepsilon'\in \EE_{B}$, 
$$
\|\1_{\varepsilon,A}\|\le  \K\|\1_{\varepsilon',B}\|. 
$$
\end{enumerate}
\end{proposition}

\begin{proof}
As this is not a quantitative result, we will not keep track of the constants (even so, note that the right-hand side of the inequality in \ref{normlikeweight} was estimated in Proposition~\ref{propositionupperbounds}). \\
Note that \ref{normlikeweight} follows from \ref{anotherone}, Proposition~\ref{propositionupperbounds}, and the fact that $\|x\|_{\infty}\le \lambda '\|x\|$ for all $x\in \X$;  \ref{propC} follows from that fact together with \ref{anotherone} and \ref{normlikeweight}, whereas \ref{superdem} follows from \ref{normlikeweight}. Thus, we only need to prove \ref{anotherone}. Moreover, by Lemma~\ref{lemmaallfinite}, it is sufficient to prove \ref{anotherone} for $x$ with finite support. \\
Let $\B=( \xx_i)_{i\in\N}$. Since $(v_k)_{k\in \N}$ is not weakly null, passing to a subsequence we may assume there is $\epsilon>0$ and $x^*\in S_{X^*}$ such that
$$
|x^*(v_k)|\ge \epsilon, \qquad\forall k\in\N. 
$$
For each $k\in \N$, set $A_k:=\supp(v_k)$, and define $c_1:=\max\{1,\| (\|v_k\|_{\infty})_{k\in\N}\|_{\infty}\}$.\\
Note that $c_1$ is a well-defined positive number because $\B^*$ and $(v_k)_{k\in \N}$ are both  bounded. For each $k\in \N$, choose $\varepsilon^{(k)}\in \EE_{A_k}$ so that 
$$
\varepsilon_j^{(k)} x^*(\xx_j)\ge 0, \qquad \forall j\in A_k. 
$$ 
Note that $(\1_{\varepsilon^{(k)},A_k})_{k\in \N}$ is bounded by Proposition~\ref{propositionupperbounds}. Let $\epsilon_1:=\epsilon c_1^{-1}$. For each $k\in \N$, we have
\begin{align}
x^*(\1_{\varepsilon^{(k)},A_k})=&\sum_{j\in A_k}|x^*(\xx_j)|\ge c_1^{-1}\sum_{j\in A_k}|\xx_j^*(v_k)| |x^*(\xx_j)|\ge c_1^{-1}|x^*(\sum_{j\in A_k}\xx_j^*(v_j)\xx_j)|\nonumber\\
=& c_1^{-1}|x^*(v_k)| \ge  c_1^{-1}\epsilon =\epsilon_1>0. \label{notweaklynull}
\end{align}
Set 
$$
a=\liminf_{k\to \infty}w(A_k). 
$$
By hypothesis, $a$ is a nonnegative real number. We claim that $a>0$. Otherwise, there would be a subsequence  $(\1_{\varepsilon^{(k_j)},A_{k_j}})_{j\in \N}$ such that $(w(A_{k_j}))_{j\in \N}\in \ell_1$. By Lemma~\ref{lemmal1thenc0}, $((\xx_i)_{i\in A_{k_j}})_{j\in  \N}$  would be equivalent to the canonical unit vector basis of $\mathtt{c}_{0}$, so all of its bounded block bases would be weakly null, contradicting \eqref{notweaklynull}. Thus, passing to a subsequence if necessary we may assume that 
\begin{equation}
\frac{a}{2}\le  w(A_k)\le 2a, \qquad\forall k\in\N. \label{similarweights5}
\end{equation}
Let $\M:=\M_{fs}((\1_{\varepsilon^{(k)},A_k})_{k\in \N},\X)$. Applying Proposition~\ref{propositionseparation}, again we may assume that $(\1_{\varepsilon^{(k)},A_k})_{k\in \N}$ is already a separating sequence with the properties of Definition~\ref{definitionseparation} for $(\X,\M,1)$. \\
Now fix $x\in \X$  with finite support , $A$ a finite nonempty subset of $\N$, and $\varepsilon\in \EE_{A}$. We may assume that $A\subset \supp(x)$.  If $w(A)\le 3a$, pick any $i\in A$. We have 
\begin{equation}
\min_{i\in A}|\xx_i^*(x)| w(A)\le 3a |\xx_{i}^*(x)|\le 3a \lambda '\|x\|. \label{lightset5}
\end{equation}
On the other hand, if $w(A)>3a$, define $b:=\min_{i\in A}|\xx_i^*(x)|$, 
and set 
$$
\F:=\supp(x). 
$$

By \eqref{similarweights5}, there is $B>j_{\F,1}$ such that  
\begin{equation}
\sum_{k\in B}w(A_k)\le w(A)\le 2 \sum_{k\in B}w(A_k). \label{approaxingAinweight}
\end{equation}
Set 
$$
z_1:=x+ \frac{bs^{2}}{2}\sum_{k\in B}\1_{\varepsilon^{(k)},A_k}. 
$$
By Lemma~\ref{lemmasgset}, there is $D\supset A$ such that $D\in \GS(z_1,|D|,s)$ and 
$$
\min_{j\in D}|\xx_j^*(z_1)|\ge s^2 \min_{j\in A}|\xx_j^*(z_1)|=s^2 \min_{j\in A}|\xx_j^*(x)|=bs^2. 
$$
It follows that $D\subset \supp(x)$, so there is $z_2\in \F$ such that 
$$
\|z_1-z_2\|\le \C  \inf_{\substack{z\in \X\\|\supp(z)|<\infty\\w(\supp(z))\le w(D)}}\|z_{1}-z\|. 
$$
Thus, using \eqref{approaxingAinweight} and the separating condition on $(\1_{\varepsilon^{(k)},A_k})_{k\in \N}$ we deduce that 
\begin{align*}
\|\frac{b s^{2}}{2} \sum_{k\in B}\1_{\varepsilon^{(k)},A_k}\|\le& \|z_1-z_2\|+\|x-z_2\|\le (\M+2)\|z_1-z_2\|\le (\M+2)\C\|x\|. 
\end{align*}
On the other hand, by \eqref{notweaklynull},  \eqref{similarweights5} and \eqref{approaxingAinweight}, 
\begin{align*}
\|\sum_{k\in B}\1_{\varepsilon^{(k)},A_k}\|\ge& |x^*(\sum_{k\in B}\1_{\varepsilon^{(k)},A_k})|\ge \epsilon_1 |B|\ge \epsilon_1 \sum_{k\in B}\frac{w(A_k)}{2a}\ge \frac{\epsilon_1}{4a} w(A). 
\end{align*}
Hence, 
$$
\min_{i\in A}|\xx_i^*(x)|w(A)=bw(A) \le 8 a \epsilon_1^{-1} s^{-2}\C(\M+2)\|x\|. 
$$
The proof is completed combining the above inequality and \eqref{lightset5}. 
\end{proof}

In \cite{BDKOW2019}, the authors introduced and studied the \emph{weighted Property (A)},  extending to the weight setting a property that arises naturally in the context of the TGA. Property (A) has been studied for example in \cite{AA2017}, \cite{AW2006} and  \cite{BB2017}. 

\begin{definition}(\cite[Definition~1.3]{BDKOW2019})
Let $\B$ be a basis for $\X$, $\w$ a weight and $\C>0$. We say that
$\B$ \emph{has the $\C$-$\w$-Property (A)} if
$$
\|x+\1_{\varepsilon,A}\|\le \C \|x+\1_{\varepsilon',B}\|,
$$
for any $ x\in \X, \|x\|_{\infty}\le 1$, for any $A,B\in \N^{<\infty}$ such that  $w(A)\le w(B)$ with $A\cap B=\emptyset$ and 
$\supp(x)\cap (A\cup B)=\emptyset,$ and for any $\varepsilon\in \EE_A, \varepsilon'\in \EE_B$.
\end{definition}

From our previous results, we have the following corollary, which extends \cite[Theorem 5.2]{BDKOW2019}. 

\begin{corollary}\label{corollaryPropCwA} Let $\B$ be a basis for $\X$, $\w$ a weight and $0<s\le1$. If $\B$ is $s$-$\w$-semi-greedy, it is truncation quasi-greedy and $\w$-superdemocratic. Hence, it has the $\w$-Property (A). 
\end{corollary}
\begin{proof}
Note that if $\w\in \mathtt{c}_0\setminus \ell_1$ and the hypotheses of Proposition~\ref{propositionc0notl1} do not hold, then by Proposition~\ref{propositionseparation} the hypotheses of Theorem~\ref{theoremc0notl1norming} do, with $\M=1$. Hence, it follows from Theorems~\ref{theoremswemigreedydisjointgreedynotc0}, ~\ref{theoremc0notl1norming}, Proposition~\ref{propositionc0notl1} and Lemma~\ref{lemmal1thenc0} that if  $\B$ is $s$-$\w$-semi-greedy, it is truncation quasi-greedy and $\w$-superdemocratic. Then, by \cite[Proposition 3.13]{BDKOW2019}, it has the $\w$-Property (A). 
\end{proof}

Before we end this section, there are two questions about  weak weight-semi-greedy bases that need to be addressed. First, in Definition~\ref{definitionweakweightsemigreedy} we require the existence of \emph{one} $s$-greedy set for which \eqref{weakweightsemigreedy} holds; this is in line with similar definitions in \cite{BL2021} and \cite{DKSW2012}. An alternative would have been to require that \eqref{weakweightsemigreedy} holds for \emph{all} $s$-greedy sets; this would have been similar to the definition of $\mathbf{n}$-$s$-quasi greedy bases in \cite{Oikhberg2017}. So, one question is whether the two approaches are equivalent. 
 The second is whether a basis that is $s$-$\w$-semi-greedy for some $0<s\le 1$ and weight $\w$, is also $t$-$\w$-semi-greedy for all $0<t\le 1$. After proving that weak weight-semi-greedy bases are truncation quasi-greedy, we are able to tackle and answer both questions in the affimative. In order to do so, we will use the following definition (see e.g. \cite{AABW2021}).

\begin{definition}A basis $\B$ is \emph{suppression unconditional for constant coefficients} with constant $\C$ ($\C$-SUCC) if 
$$
\|\1_{\varepsilon, B}\|\le \C\|\1_{\varepsilon,A}\|
$$
for all $B\subset A\in \N^{<\infty}$ and all $\varepsilon\in \EE_A$. 
\end{definition} 
Clearly a basis that is either $\C$-truncation quasi-greedy or $\C$-$\w$-superdemocratic for some weight $\w$ is also $\C$-SUCC. Hence, by Corollary~\ref{corollaryPropCwA}, any weak weight-semi-greedy basis is SUCC. 
 
\begin{theorem}\label{theoremalltgreedysetsallt}Let $\B$ be a basis for $\X$, $\w$ a weight and $0<s\le 1$. If $\B$ is $s$-$\w$-semi-greedy, then for every $0<t\le 1$  there is $\C(s,t)>0$ such that for every $x\in \X$, $m\in \N$, and $A\in\G(x,m,t)$, there is $y\in \X$ with $\supp(y)\subset A$ such that
\begin{align}
\|x-y\|\le \C(s,t)\inf_{\substack{z\in \X\\w(\supp(z))\le w(A)}}\|x-z\|.\label{theoremalltgreedysetsalltallsets}
\end{align}
In particular, $\B$ is $t$-$\w$-semi-greedy for all $0<t\le 1$.
\end{theorem}
\begin{proof} Set $\B=(\xx_i)_{i\in \N}$ and $\B^*=(\xx^*_i)_{i\in \N}$ its dual basis. 
Suppose $\B$ is $\C_1$-$s$-$\w$-semi-greedy. By the previous remarks and Corollary~\ref{corollaryPropCwA}, $\B$ is $\C_2$-SUCC and $\C_3$-truncation quasi-greedy for some positive constants $\C_2\le \C_3$. 

Set $x,m,t, A$ as in the statement. 
If $x=P_A(x)$, there is nothing to prove. Otherwise, take $a:=\min_{i\in A}|\xx_i^*(x)|>0$ because $A\in \G(x,m,t)$. Choose $y\in [\xx_i: i\in A]$ so that 
\begin{align}
\|x-y\|=\min_{\substack{z\in \X\\ \supp(z)\subset A }} \|x-z\|, \label{theoremalltgreedysetsalltmbestapprox}
\end{align}
pick $\epsilon>0$, and set $$u:=x+at^{-1}s^{-1}(1+\epsilon)\1_{\varepsilon(x), A}.$$
For every $i\in A$ and $k\not \in A$, 
\begin{align*}
|\xx_k^*(u)|=&|\xx_k^*(x)|\le t^{-1}a<s |\xx_i^*(u)|.
\end{align*}
Hence, $\G(u,m,s)=\{A\}$. By Lemma~\ref{lemmafinitesupportisenough}, there is $y_1\in [\xx_i:i\in A]$ such that 
\begin{align*}
\|u-y_1\|\le \C_1\inf_{\substack{z\in \X\\w(\supp(z))\le w(A)}}\|u-z\|.
\end{align*}
Thus, setting $$y_2:=-at^{-1}s^{-1}(1+\epsilon)\1_{\varepsilon(x), A}+y_1,$$ we get 
\begin{align}
\|x-y_2\|=&\|u-y_1\|\le \C_1 \inf_{\substack{z\in \X\\w(\supp(z))\le w(A)}}\|x-z +at^{-1}s^{-1}(1+\epsilon)\1_{\varepsilon(x), A}\|.\label{theoremalltgreedysetsalltmbestapprox2}
\end{align}
Now choose $v\in \X$ with $w(\supp(v))\le w(A)$. If $\supp(v)\subset A$, then $\|x-y\|\le \|x-v\|$. Otherwise, set
\begin{align*}
&A_1:=A\cap \supp(v), &&A_2:=A\setminus \supp(v).
\end{align*}
By \eqref{theoremalltgreedysetsalltmbestapprox2}, 
\begin{align}
\|x-y_2\|\le& \C_1\|x-v-at^{-1}s^{-1}(1+\epsilon)\1_{\varepsilon(x), A_1}+at^{-1}s^{-1}(1+\epsilon)\1_{\varepsilon(x), A} \|\nonumber\\
\le&\C_1\|x-v\|+a\C_1t^{-1}s^{-1}(1+\epsilon)\|\1_{\varepsilon(x), A_2}\|.\label{theoremalltgreedysetsalltmbestapprox3}
\end{align}
Note that $A_2\not=\emptyset$. Let $A_3$ be a greedy set for $x-v$ of minimum cardinality containing $A_2$. We have
\begin{align*}
\min_{i\in A_3}|\xx_i^*(x-v)|=&\min_{i\in A_2}|\xx_i^*(x-v)|=\min_{i\in A_2}|\xx_i^*(x)|\ge \min_{i\in A}|\xx_i^*(x)|=a. 
\end{align*}
Since $A_2\subset A_3$, 
\begin{align*}
a\|\1_{\varepsilon(x), A_2}\|=& a\|\1_{\varepsilon(x-v), A_2}\|\le \C_2a\|\1_{\varepsilon(x-v),A_3}\|\le \C_2\C_3\|x-v\|. 
\end{align*}
Combining the above with \eqref{theoremalltgreedysetsalltmbestapprox} and \eqref{theoremalltgreedysetsalltmbestapprox3} it follows that 
$$
\|x-y\|\le \C_1(1+\C_2\C_3t^{-1}s^{-1}(1+\epsilon))\|x-v\|. 
$$
Taking infimum over all such $v$ and letting $\epsilon\rightarrow 0$, we conclude that \eqref{theoremalltgreedysetsalltallsets} holds for $\C(s,t)=\C_1(1+\C_2\C_3 t^{-1}s^{-1})$. 
\end{proof}

\bigskip
\section{Weak weight-almost semi-greedy bases.}\label{sectionalmostsemigreedy}

An examination of the proofs of Proposition~\ref{propositionupperbounds}, Theorem~\ref{theoremswemigreedydisjointgreedynotc0}\ref{wdisjointgreedyest2a}, Theorem~\ref{theoremc0notl1norming}\ref{wdisjointgreedyest4}, Proposition~\ref{propositionc0notl1} and Corollary~\ref{corollaryPropCwA} shows that these results do not need the full strength of the $\C$-$s$-$\w$-semi-greedy property, but can be obtained using approximations by projections. This suggests an ``almost semi-greedy'' property, and its corresponding weak and weighted versions. In this section we study the position of such bases with respect to the already known ones.

\begin{definition}\label{definitionweakweightalmostsemigreedy} Let $\B$ be a basis for $\X$, $\w$ a weight, $\C>0$, and $0<s\le 1$. We say that $\B$ is \emph{weak weight-almost semi-greedy with parameter $s$, weight $\w$ and constant $\C$}  (or $\C$-$s$-$\w$-almost semi-greedy) if, for every $x\in \X$ and $m\in \N$, there is $A\in \G(x,m,s)$ and $y\in \X$ with $\supp(y)\subset A$ such that
\begin{equation}
\|x-y\|\le \C\inf_{\substack{B\in \N^{<\infty}\\ w(B)\le w(A)}}\|x-P_B(x)\|.\label{definitionweakweightalmostsemigreedyb}
\end{equation}
We denote by $\mathcal{GAS}(x,m,s)$ the subset of $\G(x,m,s)$ for which the above holds. In case $s=1$ and $w_n=1$ for all $n\in \N$, we say that $\B$ is \emph{$\C$-almost semi-greedy}. 
\end{definition}

As pointed out,  the same proofs of the aforementioned results hold verbatim under this weaker hypothesis, obtaining the same estimates when we replace the $s$-$\w$-semi-greedy constant with the $s$-$\w$-almost semi-greedy one. Thus, in particular, any  $s$-$\w$-almost semi-greedy basis is truncation quasi-greedy and $\w$-superdemocratic. It turns out that truncation quasi-greediness and $\w$-superdemocracy characterize $\w$-almost semi-greediness and its weak variant, as we prove next. 

\begin{proposition}\label{propositionequivalencesalmostsemigreedy}
Let $\B$ be a basis for $\X$ and $\w$ a weight. The following are equivalent:
\begin{enumerate}[ \rm (i)]
\item \label{propositionequivalencesalmostsemigreedytqg+wsd} $\B$ is truncation quasi-greedy and $\w$-superdemocratic. 
\item \label{propositionequivalencesalmostsemigreedytqg+wd}$\B$ is truncation quasi-greedy and $\w$-democratic. 
\item \label{propositionequivalencesalmostsemigreedyalmostsemiall+} 
For every $0<s\le 1$, $\B$ is $s$-$\w$-almost semi-greedy.\\
Moreover, there is $\C(s)>0$ such that for each $x\in \X$ and $m\in \N$, $\mathcal{GAS}(x,m,s, \C(s))=\G(x,m,s)$ and, for each $A\in \G(x,m,s)$, there is $D\subset A$ such that \eqref{definitionweakweightalmostsemigreedyb} holds for $\C=\C(s)$ and $y=P_D(x)$.
\item \label{propositionequivalencesalmostsemigreedyalmostsemis}There is $0<s\le 1$ such that $\B$ is $s$-$\w$-almost semi-greedy. 
\item \label{propositionequivalencesalmostsemigreedydisjoint}For each $0<s\le 1$ there is $\C(s)$ such that
\begin{align*}
\|x\|\le& \C(s)\|x-P_B(x)\|,
\end{align*}
for all $x\in \X$ and $B\in \N^{<\infty}$ such that there is $m\in \N$ and $A\in \G(x,m,s)$ disjoint from $B$ with $w(B)\le w(A)$. 
\end{enumerate}

\end{proposition}
\begin{proof}
The implications \ref{propositionequivalencesalmostsemigreedytqg+wsd} $\Longrightarrow$ \ref{propositionequivalencesalmostsemigreedytqg+wd} and  \ref{propositionequivalencesalmostsemigreedyalmostsemiall+} $\Longrightarrow$ \ref{propositionequivalencesalmostsemigreedyalmostsemis} are immediate, whereas \ref{propositionequivalencesalmostsemigreedyalmostsemis} $\Longrightarrow$ \ref{propositionequivalencesalmostsemigreedytqg+wsd} is the counterpart of Corollary~\ref{corollaryPropCwA}.\\
To prove \ref{propositionequivalencesalmostsemigreedydisjoint} $\Longrightarrow$ \ref{propositionequivalencesalmostsemigreedyalmostsemiall+}, fix $0<s\le 1$, $x\in \X$, $m\in \N$, $A\in \G(x,m,s)$, and $B\in \N^{<\infty}$ so that $w(B)\le w(A)$. Let $E:=A\cap B$, $A_1:=A\setminus B$, $B_1:=B\setminus A$ and $D\subset A$ so that $\|x-P_D(x)\|\le \|x-P_S(x)\|$ for all $S\subset A$. Since $A_1\in \G(x-P_E(x),|A_1|,s)$ and $w(B_1)\le w(A_1)$, we have
$$
\|x-P_{D}(x)\|\le \|x-P_{E}(x)\|\le \C(s)\|x-P_{E}(x)-P_{B_1}(x-P_E(x))\|=\|x-P_B(x)\|. 
$$
Now the result follows by taking infimum. \\
To prove \ref{propositionequivalencesalmostsemigreedytqg+wd} $\Longrightarrow$ \ref{propositionequivalencesalmostsemigreedytqg+wsd}, suppose that $\B$ is $\C_1$-truncation quasi-greedy and $\C_2$-$\w$-democratic, and fix $A,B\in \N^{<\infty}$ with $w(A)\le w(B)$, $\varepsilon\in \EE_A$, $\varepsilon'\in \EE_B$. By 
Remark~\ref{remarkpropc=tqg}, 
$$
\|\1_{\varepsilon,A}\|\le 2\kappa \C_1\|\1_{A}\|\le 2\kappa \C_1\C_2\|\1_{B}\|\le 4\kappa^2\C_1^2\C_2\|\1_{\varepsilon', B}\|.
$$
Finally, set $\B^*=(\xx^*_i)_{i\in \N}$ the dual basis of $\B$ to  prove \ref{propositionequivalencesalmostsemigreedytqg+wsd}$\Longrightarrow$\ref{propositionequivalencesalmostsemigreedydisjoint}. Suppose $\B$ is $\C_1$-truncation quasi-greedy and $\C_2$-$\w$-superdemocratic, and choose $s,x,m, B$ and $A$ as in the statement. Let $A_1$ be a  greedy set for $x-P_B(x)$ containing $A$ of minimum cardinality. Given that $w(B)\le w(A_1)$, $A\in \G(x-P_B(x),m,s)$ and 
$$
\min_{i\in A}|\xx_i^*(x)|=\min_{i\in A}|\xx_i^*(x-P_B(x))|=\min_{i\in A_1}|\xx_i^*(x-P_B(x))|,
$$
we have 
\begin{align*}
\|x\|\le& \|x-P_B(x)\|+\|P_B(x)\|\le \|x-P_B(x)\|+\max_{i\in B}|\xx_i^*(x)|\max_{\varepsilon\in \EE_{B}}\|\1_{\varepsilon,B}\|\\
\le& \|x-P_B(x)\|+s^{-1}\min_{i\in A_1}|\xx_i^*(x-P_B(x))|\C_2\|\1_{\varepsilon(x-P_B(x)), A_1}\|\\
\le&(1+s^{-1}\C_1\C_2)\|x-P_B(x)\|. 
\end{align*}
\end{proof}

\begin{remark}\label{remarkonlymarkushevich}\rm Note that the proof of the implication \ref{propositionequivalencesalmostsemigreedyalmostsemis} $\Longrightarrow$ \ref{propositionequivalencesalmostsemigreedytqg+wsd} of Proposition~\ref{propositionequivalencesalmostsemigreedy} is the only one in which the totality condition on $\B^*$ is used. For that implication, the condition is essential: indeed, \cite[Example 4.5]{BL2021} shows that even semi-greedy systems need not be democratic. In fact, it is easily shown that the system of \cite[Example 4.5]{BL2021} is not truncation quasi-greedy, either, or even unconditional for constant coefficients (see \cite[Definition 3]{W2000}). 
\end{remark}

Next, we give an analogue of Lemma~\ref{lemmafinitesupportisenough} for  weak weight-almost semi-greedy bases, which is proved in a similar manner. 
\begin{lemma}\label{lemmafinitesupportisenoughasg}Let $\B$ be a basis for $\X$, $\w$ a weight and $0<s\le 1$. Suppose that there is $\C>0$ such that the conditions of Definition~\ref{definitionweakweightalmostsemigreedy} hold for $x$ with finite support. Then $\B$ is $\C$-$s$-$\w$-almost semi-greedy. Moreover, the conditions of Definition~\ref{definitionweakweightalmostsemigreedy} hold even if the infimum is taken without the restriction $|B|<\infty$.
\end{lemma}
\begin{proof}
Pick $x\in \X\setminus \{0\}$ and $m\in \N$. As in the proof of Lemma~\ref{lemmafinitesupportisenough}, we assume $|\supp(x)|\ge m$ and, for every $n\in \N$, we choose $x_n\in \X$ with finite support such that $\|x-x_n\|\le n^{-1}$ and $\G(x_n,m,t)=\G(x,m,t)$. By hypothesis, for each $n$ there are $A_n\in \G(x,m,t)$ and $y_n\in \X$ with $\supp(y_n)\subset A_n$ such that 
$$
\|x_n-y_n\|\le \C\inf_{\substack{D\in \N^{<\infty}\\w(D)\le w(A_n)}}\|x_n-P_D(x_n)\|.
$$
Passing to a subsequence, we assume $A_n=A$ and there is $y$ supported in $A$ such that $\|y_n-y\|\le n^{-1}$ for all $n$. Given $B\in \N^{<\infty}$ with $w(B)\le w(A)$,
\begin{align*}
\|x-y\|\le&2n^{-1}+\|x_n-y_n\|\le 2n^{-1}+\C\|x_n-P_B(x_n)\|\\
\le& n^{-1}(2+\C+\C\|P_B\|)+\C\|x-P_B(x)\|. 
\end{align*}

As this holds for every $n\in \N$, it follows that $\|x-y\|\le \C\|x-P_B(x)\|$. Thus, taking infimum over all such $B$, we obtain that $\B$ is $\C$-$\w$-almost semi-greedy. \\
Now set $x, m$ as before, choose $A\in \G(x,m,t)$ and $y\in \X$ with $\supp(y)\subset A$ so that \eqref{definitionweakweightalmostsemigreedyb} holds, and suppose there is $B\subset \N$ with $|B|=\infty$ and $w(B)\le w(A)$. By Proposition~\ref{propositionequivalencesalmostsemigreedy}, $\B$ is $\w$-superdemocratic. Hence, by Remarks~\ref{remarkprojectionc0} and~\ref{remarkl1wdem}, $P_B(x)$ is well-defined and, for each $n\in \N$, there is $B_n\subset B$ such that $B_n\in \N^{<\infty}$ and $\|P_B(x)-P_{B_n}(x)\|\le n^{-1}$. Thus, 
\begin{align*}
\|x-y\|\le \C\|x-P_{B_n}(x)\|\le \C\|x-P_{B}(x)\|+\C n^{-1}.
\end{align*}
As $n$ is arbitrary, we conclude that $\|x-y\|\le \C\|x-P_{B}(x)\|$, and the proof is completed by taking supremum. 
\end{proof}

In many cases equivalent weights $\w \approx \w'$ (those satisfying $\C_1 w_n\le w_n'\le \C_2 w_n$ for all $n\in \N$, for some $0<\C_1 \le \C_2<\infty$) produce the same weighted bases. This happens for
the $\w$-Property (A), $\w$-superdemocracy and $\w$-democracy which are equivalent to their respective $\w'$-counterparts (see \cite[Proposition 3.5, Remark 3.6]{BDKOW2019}). This extends to $\w$-almost greedy bases characterized as quasi-greedy and $\w$-democratic \cite[Theorem 2.6]{DKTW2018}. Proposition~\ref{propositionequivalencesalmostsemigreedy} combined with the aforementioned result for $\w$-democracy gives the following natural result.

\begin{corollary}\label{corollaryequivalentweightsalmostsemi}Let $\B$ be a basis for $\X$, $0<s\le 1$. If $\w$ and $\w'$ are equivalent weights, then $\B$ is $s$-$\w$-almost semi-greedy if and only if it is $s$-$\w'$-almost semi-greedy. 
\end{corollary}

In the case of constant weights, Proposition~\ref{propositionequivalencesalmostsemigreedy}   sheds light on some of the connections between the Chebyshevian Thresholding Greedy algorithm with bidemocratic and  squeeze symmetric bases, which are defined below.  

Recall that squeeze symmetric bases are those that can be sandwiched between  two symmetric bases of quasi-Banach spaces with equivalent fundamental functions (see \cite{AAB2021} and \cite{AABW2021} for further definitions and results). 
\begin{definition}\label{definitionsqueezesymmetric}Let $\B$ be a basis for $\X$. We say that $\B$ is \emph{squeeze symmetric} if there are quasi-Banach spaces $\Y$, $\Z$ with symmetric bases $\B_1=(\yy_n)_{n\in\N}$ and $\B_2=(\zz_n)_{n\in\N}$ and  bounded linear operators $T: \Y\rightarrow \X$ and $S: \X\rightarrow \Z$ such that by $T(\yy_n)=\xx_n$ and $S(\xx_n)=\zz_n$, and there is $\C>0$ such that 
$$
\|\1_{\varepsilon,A}[\B_1]\|\le \C\|\1_{\varepsilon,A}[\B_2]\|,\qquad\forall A\in \N^{<\infty},\ \forall \varepsilon\in \EE_A.
$$ 
\end{definition}
It was proven in \cite[Proposition 9.4, Corollary 9.15]{AABW2021} that a  basis is squeeze symmetric if and only if it is truncation quasi-greedy and superdemocratic (or democratic,  \cite[Proposition 4.16]{AABW2021}). In our context, that is when $\B$ is Markushevich  basis for a Banach space $\X$, a combination of the aforementioned results with Proposition~\ref{propositionequivalencesalmostsemigreedy}   gives the following. 
\begin{corollary}Let $\B$ be a basis for $\X$. The following are equivalent: 
\begin{itemize}
\item $\B$ is almost semi-greedy. 
\item $\B$ is truncation quasi-greedy and superdemocratic. 
\item $\B$ is truncation quasi-greedy and democratic. 
\item $\B$ is squeeze symmetric. 
\end{itemize}
\end{corollary}

Bidemocratic basis were introduced in \cite{DKKT2003} to study duality in connection to the TGA. A central result is that a quasi-greedy basis $\B$ is bidemocratic if and only if both $\B$ and $\B^*$ are almost greedy, or partially greedy \cite[Theorem 5.4]{DKKT2003}. 

\begin{definition}\label{definitionbidemocratic}A basis $\B$ is \emph{bidemocratic} if there is $\C>0$ such that, for each $m\in \N$, 
\begin{align*}
\sup_{\substack{A\subset \N\\|A|\le m}}\sup_{\substack{B\subset \N\\|B|\le m}}\|\1_{A}\|\|\1_{B}^{*}\| \le& \C m. 
\end{align*}
\end{definition}
Bidemocratic bases have also been studied for their own sake in \cite{AABBL2021}, where it was proven that  a basis $\B$ for a Banach space $\X$ is bidemocratic if and only if both $\B$ and $\B^*$ are truncation quasi-greedy and (super)democratic \cite[Corollary 2.6]{AABBL2021}. In the case of Markushevich bases, we can combine this result with Proposition~\ref{propositionequivalencesalmostsemigreedy} to obtain the following. 
\begin{corollary}\label{corollarybidem}A basis $\B$ is bidemocratic if and only if both $\B$ and $\B^*$ are almost semi-greedy. 
\end{corollary}

\section{Weak weight-almost greedy bases} \label{sectionweightweakalmostgreedy}
It is known that in the standard case -equivalently, in the case of constant weights-, almost greediness with respect to the weak algorithm is equivalent to almost greediness \cite[Proposition 2.3]{BL2021} and \cite[Theorem 6.4]{DKSW2012}. Thus, it is natural to ask whether this equivalence has an extension to general weights. In this section, we answer this question in the affirmative. To do so, we introduce and briefly study a weighted extension of the notion of weak almost greedy bases from \cite{BL2021} - which can also be seen as a weak-algorithm extension of the concept of weight-almost greedy bases from \cite{DKTW2018} (see remarks after Definition~\ref{definitionweightsemigreedyetc}).

\begin{definition}\label{definitionweakweightalmostgreedy}
Let $\B$ be a basis for $\X$, $\w$ a weight, $\C>0$, and $0<s\le 1$. We say that $\B$ is \emph{weak weight  almost greedy with parameter $s$, weight $\w$ and constant $\C$} (or  $\C$-$s$-$\w$-almost-greedy) if, for every $x\in \X$ and $m\in \N$, there is $A\in \G(x,m,s)$ such that 
\begin{equation}
\|x-P_A(x)\|\le \C\inf_{\substack{B\in \N^{<\infty}\\w(B)\le w(A)}}\|x-P_B(x)\|.\nonumber
\end{equation}
We denote by $\GA(x,m,s)$ the subset of $\G(x,m,s)$ for which the above bound holds. 
\end{definition}
\smallskip

\begin{remark}\rm \label{remarkdisjointsuperdem}
First, note that every $\C$-$s$-$\w$-almost-greedy is  $\C s^{-1}$-$\w$-disjoint superdemocratic. Indeed, given disjoint nonempty sets $A,B\in \N^{<\infty}$ with $w(A)\le w(B)$, $\varepsilon\in \EE_A, \varepsilon'\in \EE_B$,  for every $\epsilon>0$ we have
$$
\|\1_{\varepsilon, A}\|=\|\1_{\varepsilon, A}+(1+\epsilon)s^{-1}\1_{\varepsilon',B}-(1+\epsilon)s^{-1}\1_{\varepsilon',B}\|\le \C(1+\epsilon)s^{-1}\|\1_{\varepsilon',B}\|,
$$
where we used that $\G(\1_{\varepsilon, A}+(1+\epsilon)s^{-1}\1_{\varepsilon',B}, |B|,s)=\{B\}$. 
Finally, as in addition $\B$ is $\w$- semi greedy, Proposition~\ref{propositionupperbounds} applies to these bases. In particular, if $\w\in \ell_1$, then $\B$ is equivalent to the canonical unit vector basis of $\mathtt{c}_{0}$. 
\end{remark}

To prove that $\w$-$s$-almost greedy bases are quasi-greedy, we need an auxiliary result analogous to Lemma~\ref{lemmasgset}. 

\begin{lemma}\label{lemmaagset}Let $\B$ be a $\C$-$s$-$\w$-almost greedy  basis for $\X$ and let $\B^*=( \xx^*_i)_{i\in\N}$ be the dual basis of $\B$. For every $x\in \X$ and every nonempty set $A\in \N^{<\infty}$, there is $m\in \N$ and $E\in \GA(x,m,s)$ such that $A\subset E$ and for all $j\in E$,
$$
|\xx_j^*(x)|\ge s^{2} \min_{i\in A}|\xx_i^*(x)|.
$$
\end{lemma}

\begin{proof}
This is proven by the same argument as Lemma~\ref{lemmasgset}.
\end{proof}

Now we can prove the main result of this section. The proof is based on that of \cite[Proposition 2.3]{BL2021}, which in turn is based on the arguments of the proof of \cite[Proposition 4.4]{DKSW2012}.

\begin{theorem}\label{theoremswalmostgreedy}Let $\B=(\xx_i)_{i\in\N}$ be a $\C$-$s$-$\w$-almost greedy basis for $\X$, with $\C>0$, $0<s\le 1$ and $\w$ a weight such that $\w\not \in \ell_1$, and set
$$
\K:=\C s^{-1} \max\{2\C s^{-1},\lambda \lambda '\}.
$$
The following hold:
\begin{enumerate}[\upshape (i)]
\item \label{whyperdem} For all $A,B \in \N^{<\infty}$ with $w(A)\le w(B)$, if $(a_i)_{i\in A}$, $(b_i)_{i\in B}$ are scalars such that 
$$
\max_{i\in A}|a_i|\le \min_{i\in B}|b_i|, 
$$
then 
$$
\|\sum_{i\in A}a_i\xx_i\|\le \K\|\sum_{i\in B}b_i\xx_i\|.
$$
\item \label{qg} For every $0<t\le 1$, $x\in \X$, $m\in \N$ and $A\in \G(x,m,t)$, 
$$
\|P_A(x)\|\le (1+ t^{-1}s^{-2}\K)(1+\C) \|x\|.
$$
In particular, $\B$ is quasi-greedy with constant as above taking $t=1$. 
\end{enumerate}
\end{theorem}

\begin{proof}
To prove \ref{whyperdem}, fix $A,B, (a_i)_{i\in A}$, $(b_i)_{i\in B}$ as in the statement with 
$$
a:=\max_{i\in A}|a_i|>0,
$$
and consider two cases: \\
First, suppose there is $j>A$ such that $ w_{j}\ge w(A)$. As $a\le  \lambda'\|\sum_{i\in B}b_i\xx_i\|$, by 
 Remark~\ref{remarkdisjointsuperdem} and convexity 
\begin{equation}
\|\sum_{i\in A}a_i\xx_i\|\le a \sup_{\varepsilon\in \EE_A}
\|\1_{\varepsilon,A}\|\le a s^{-1}\C\|\xx_{j}\|\le  a s^{-1}\C\lambda \le  s^{-1}\C\lambda \lambda '\|\sum_{i\in B}b_i\xx_i\|.\label{lightweight9v2}
\end{equation}
On the other hand, if there is no such $j$, since $\w\not\in \ell_1$ we can choose two finite sets of positive integers $A_2>A_1>A \cup B$ so that 
$$
\max\{w(A_1),w(A_2)\}\le w(A)\le w(A_1)+w(A_2).
$$
As before, by Remark~\ref{remarkdisjointsuperdem},  convexity and the triangle inequality, 
\begin{equation}
\|\sum_{i\in A}a_i\xx_i\|\le  a s^{-1}\C\|\1_{A_1}+\1_{A_2}\|\le  a s^{-1}\C\|\1_{A_1}\|+ a s^{-1}\C\|\1_{A_2}\|.\label{oneside6v2}
\end{equation}
Fix $0<\epsilon<1$, and set  
$$
z_1:=\sum_{i\in B}b_i\xx_i +(1-\epsilon)a s \1_{A_1}.
$$
Given that $\G(z_1,|B|,s)=\{B\}$ and $w(A_1)\le w(B)$, we have 
$$
\|(1-\epsilon)a s \1_{A_1}\|=\|z_1-P_B(z_1)\|\le \C\|z_1-P_{A_1}(z_1)\|=\C\|\sum_{i\in B}b_i\xx_i\|.
$$
By the same argument, the above bound also holds for $\|(1-\epsilon)a s \1_{A_2}\|$. From this and \eqref{oneside6v2} we deduce that
$$
\|\sum_{i\in A}a_i\xx_i\|\le 2\C^2s^{-2}(1-\epsilon)^{-1}\|\sum_{i\in B}b_i\xx_i\|.
$$
As $\epsilon$ is arbitrary, the proof of \ref{whyperdem} is completed combining the above estimate and \eqref{lightweight9v2}. \\
Now, set $\B^*=(\xx^*_i)_{i\in\N}$ the dual basis of $\B$. Fix $t$, $x$, $m$ and $A$ as in \ref{qg}, assuming $x\not=P_A(x)$, and set 
$$
a:=\min_{i\in A}|\xx_i^*(x)|. 
$$
By Lemma~\ref{lemmaagset}, there is $A\subset D\in \GA(x,|D|,s)$ such that 
$$
\min_{i\in D}|\xx_i^*(x)|\ge s^2a. 
$$
Let $B:=D\setminus A$. If $B=\emptyset$, then $\|P_A(x)\|\le \|x\|+\|x-P_D(x)\|\le (\C+1)\|x\|$. Otherwise, 
$$
\max_{i\in B}|\xx_i^*(x)|\le t^{-1}a\le t^{-1}s^{-2}\min_{i\in D}\|\xx_i^*(x)\|. 
$$
Hence, by \ref{whyperdem}  
$$
\|P_B(x)\|\le t^{-1}s^{-2}\K\|P_D(x)\|. 
$$
Thus, by the triangle inequality,
\begin{align*}
\|P_A(x)\|\le& \|P_B(x)\|+\|P_D(x)\|\le (1+ t^{-1}s^{-2}\K)\|P_D(x)\|\\
\le& (1+ t^{-1}s^{-2}\K)(\|x\|+\|x-P_D(x)\|)\le (1+ t^{-1}s^{-2}\K)(1+\C)\|x\|.
\end{align*}
\end{proof}
Finally, we have:

\begin{remark}\label{remarkequivalentweights}\rm If $\w\approx \w'$ and $0<s\le 1$, then a basis $\B$ is $s$-$\w$-almost greedy if and only if it is $s$-$\w'$-almost greedy. This follows by Theorem~\ref{theoremswalmostgreedy} and the remarks before Corollary~\ref{corollaryequivalentweightsalmostsemi}. 
\end{remark}

\section{Lebesgue and Chebyshevian Lebesgue parameters}\label{sectionquasi-greedyalmostgreedy}

In this section, we study some parameters that involve the (weak) TGA and CGA, and also improve the known bounds for the quasi-greedy and almost greedy constants of weak semi-greedy bases. 

We will use the following auxiliary definitions.

\begin{definition}\label{definitiontquasigreedyparameter}Let $\B$ be a basis for $\X$, $0<t\le 1$ and $m\in\N$.
\begin{itemize}
\item The \emph{$t$-quasi-greedy parameter} $\overline{\g}(m,t)$ is defined by 
\begin{equation}
\overline{\g}(m,t):=\inf_{\C>0} \Big\{ \|P_A(x)\|\le \C\|x\|, \qquad\forall x\in \X,\; \forall A \in \G(x,m,t)\Big\}. \nonumber
\end{equation}
\item The \emph{suppression $t$-quasi-greedy parameter} $\widehat{\g}(m,t)$ is defined by 
\begin{equation}
\widehat{\g}(m,t):=\inf_{\C>0}\ \Big\{ \|x-P_A(x)\|\le \C\|x\|, \qquad\forall x\in \X,\; \forall A \in \G(x,m,t)\Big\}. \nonumber
\end{equation}
\end{itemize}
\end{definition}
\noindent For $t=1$, the parameter $\overline{\g}(m,1)$ has been considered in \cite{AABBL2021}, and the quasi-greedy parameter given by $\g_m:=\max_{n\le m}\overline{\g}(n,1)$ has been studied as well, for example in \cite{BBGHO2018}. Next we define the Chebyshevian Lebesgue  parameters associated to weak tresholding greedy algorithms.

\begin{definition}\label{definitionChebyshevianlebesguetype} Let $\B$ be a basis for $\X$,   $0<t\le 1$ and $m\in  \N$.
\begin{itemize}
\item The \emph{upper Chebyshevian Lebesgue parameter},  $\L_{ch}^u(m,t)=\L_{ch}^u(m,t)[\B,\X]$, is given by

\begin{eqnarray*}
\L_{ch}^u(m,t):=\inf_{\C>0} \left\lbrace \|x-y\|\le \C\sigma_m(x): \begin{array}{ll}
\forall x\in \X,\ \forall A\in \G(x,m,t),\\  
 \exists  y\in \X: \supp(y)\subset A 
\end{array}
\right\rbrace.
\end{eqnarray*}

\item The \emph{lower Chebyshevian Lebesgue parameter},  $\L_{ch}^l(m,t)=\L_{ch}^l(m,t)[\B,\X]$, is given by

\begin{eqnarray*}
\L_{ch}^l(m,t):=\inf_{\C>0} \left\lbrace \|x-y\|\le \C\sigma_m(x): \begin{array}{ll}
\forall x\in \X,\ \exists A\in \G(x,m,t),\\  
 \exists  y\in \X: \supp(y)\subset A 
\end{array}
\right\rbrace.
\end{eqnarray*}
\end{itemize}
As usual, to simplify our notation we leave the space and basis implicit when clear. 
\end{definition}

\begin{remark}\rm Under the conditions of Definition~\ref{definitionChebyshevianlebesguetype}, it is clear that there is a smallest $\C>0$ in the definition of $\L_{ch}^{u}(m,t)$, so the parameter is attained.  The same holds for  $\L_{ch}^{l}(m,t)$, since for each $x\in \X$ with $|\supp(x)|>m$, the set $\G(x,m,t)$ is finite. 
\end{remark}

\begin{remark}\rm 
The upper Chebyshevian Lebesgue parameter is the Chebyshevian Lebesgue constant introduced in \cite{DKO2015} and also studied in \cite{BBGHO2018}. 
\end{remark}

Next, we study a relation between $\widehat{\g}$, $\M_{fs}$ and $\L_{ch}^l$ which, in particular gives a slight improvement of the known bound  for the suppression quasi-greedy constant of semi-greedy bases. More precisely, \cite[Theorem 4.2]{BL2021} gives that if $\B$ is $\C$-$s$-weak semi-greedy, it is $\K$-suppression quasi-greedy with $\K\le \M_{fs}\C(1+(\M_{fs}+1)\C s^{-2})$, whereas Theorem~\ref{theoremswemigreedydisjointgreedynotc0} avoids the quadratic factor $\M_{fs}^2$ at the cost of involving $\lambda$ and $\lambda'$, and gives $\K\lesssim \M_{fs}\C s^{-1} \max\{\C s^{-2}, \lambda \lambda'\}$. Our next result gives $\K\le \M_{fs}\C(1+ 2\C s^{-2})$. 

\begin{proposition}\label{propositionchevyshevianqg}Let $\B$ be a basis for $\X$. For all $0<s\le 1$, $0<t\le 1$ and $m\in \N$,  we have the following estimates: 
\begin{equation}
\widehat{\g}(m,t)\le \M_{fs} \L_{ch}^l(2\floor{\frac{m+1}{2}},s) (1+ 2\L_{ch}^l(2\floor{\frac{m+1}{2}},s) t^{-1}s^{-2}). \label{semigreedytoquasigreedy}
\end{equation}
Hence, if $\B$ is $\C$-$s$-weak semi-greedy, it is $\M_{fs}\C(1+2\C s^{-2})$-suppression quasi-greedy.
\end{proposition}

\begin{proof} Set $\B=( \xx_i)_{i\in\N}$ and $\B^*=( \xx^*_i)_{i\in\N}$.  
Choose $0<\epsilon<1$, and let $(\xx_{i_k})_{k\in \N}$ be a subsequence given by an application of Corollary~\ref{corollaryboth} to $\B$ and $(\X, \M_{fs}, \epsilon)$. \\
Fix $x\in \X$ with finite support, and $A \in \G(x,m,t)$. We may assume $x\not=P_A(x)$. Set 
$$
a:=\max_{i\not\in A}|\xx_i^{*}(x)|,
$$
and pick $i_0>\supp(x)$. Set 
\begin{align*}
n:=&\floor{\frac{m+1}{2}};\\
\F:=&[\xx_i: 1\le i\le i_0];\\
E_1:=&\{i_{2(s_{\F,\epsilon}+j)-1}:1\le j\le n\};\\
E_2:=&\{i_{2(s_{\F,\epsilon}+j)}:1\le j\le n\};\\
z:=&x+ ats(1-\epsilon)\left(\1_{E_1}-\1_{E_2}\right).
\end{align*}
Note that $s_{\F,\epsilon}> i_0$ (othewise, $\xx_{s_{\F,\epsilon}}\in \F$ and
$\|\xx_{s_{\F,\epsilon}}\|\le (\M_{fs}+\epsilon)\|\xx_{s_{\F,\epsilon}}-\xx_{s_{\F,\epsilon}}\|=0)$. So, the sets $E_1$, $E_2$ and $A$ are pairwise disjoint.  Since $A\in \G(x,m,t)$, we have
$$
|\{1\le i\le i_0 : |\xx_i^{*}(z)|\ge at \}|\ge |A|+1=m+1\ge 2n. 
$$
Thus, 
$$
D \subset  \{1,\dots,i_0\},\qquad \forall D \in \G(z,2n,s). 
$$
It follows that there is $u\in \F$ such that
$$
\|z-u\|\le \L_{ch}^l(2n,s) \sigma_{2n}(z)\le \L_{ch}^l(2n,s)\|x\|.
$$
Hence, by Corollary~\ref{corollaryboth},
\begin{align}
ats(1-\epsilon)\|\1_{E_1}-\1_{E_2}\|\le& \|z-u\|+\|x-u\|\le (2+\epsilon)\|z-u\|\nonumber\\
\le&(2+\epsilon)\L_{ch}^l(2n,s)\|x\|. \label{newpart}
\end{align}
Now set 
$$
y:=x-P_A(x)-s^{-1}(1+\epsilon)a \left(\1_{E_1}-\1_{E_2}\right). 
$$
Given that 
$$|\xx_i^*(x-P_A(x)) |\le a,\qquad\forall i \in \N,$$
we have $\G(y,2n,s)=\left\lbrace E_1\cup E_2\right\rbrace$. Thus, there is $v\in \X$ with $\supp(v)\subset E_1\cup E_2$ such that 
\begin{align}
\|y-v\|\le&  \L_{ch}^l(2n,s) \sigma_{2n}(y)\nonumber\\
\le&  \L_{ch}^l(2n,s) \|x\|+ \L_{ch}^l(2n,s) s^{-1}(1+\epsilon)a \| \1_{E_1}-\1_{E_2}\|.\label{oldpart}
\end{align}
Again, by Corollary~\ref{corollaryboth}, combining \eqref{newpart} and \eqref{oldpart} we obtain 
\begin{align*}
\|x-P_A(x)\|&\le (\M_{fs} +\epsilon)\|y-v\|\\
&\le (\M_{fs} +\epsilon) \L_{ch}^l(2n,s)  (1+\L_{ch}^l(2n,s) t^{-1}s^{-2}(1+\epsilon)(2+\epsilon)(1-\epsilon)^{-1})\|x\|. 
\end{align*}
As $\epsilon$ is arbitrary, the proof of \eqref{semigreedytoquasigreedy} is complete for $x$ with finite support. The general case is handled as in the proof of Theorem~\ref{theoremswemigreedydisjointgreedynotc0}. \\Now if $\B$ is $\C$-$s$-weak semi-greedy, then 
\begin{align*}
\sup_{m\in \N}\widehat{\g}(m,1)\le& \M_{fs}\sup_{\substack{n\in \N\\n\text{ is even}}}\L_{ch}^l(n,s)(1+2\sup_{\substack{n\in \N\\n\text{ is even}}}\L_{ch}^l(n,s) s^{-2}) \le \M_{fs}\C(1+2\C s^{-2}). 
\end{align*}
\end{proof}

\begin{remark}\label{remarklinearqg}\rm 
Note that the linear factor $\M_{fs}$ in \eqref{semigreedytoquasigreedy} is necessary. Indeed, \cite[Example 4.4]{BL2021} shows that for each $\alpha>0$, there is a  basis $\B$ for a subspace $\X$ of $\ell_1$ that has quasi-greedy constant $\C_{\alpha}> \alpha$ and semi-greedy constant $\K_{\alpha}\le 4$;
the first computation of the proof implies that $\B$ is a Schauder basis equivalent to the canonical unit vector basis of $\ell_1$  with basis constant no greater than $3+2\alpha$, and the constant $\M_{fs}$ is no greater than the basis constant by Proposition~\ref{propositionseparation}\ref{schauder}. 
\end{remark}
Next, we consider the following Lebesgue-type parameters.

\begin{definition}\label{definitionlebesguedisjoint}Let $\B$ be a basis for $\X$,  $0<t\le 1$ and $m\in \N$. 
\begin{itemize}
\item The \emph{Lebesgue parameter}, $\L(m,t)=\L(m,t)[\B,\X]$, is given by 

\begin{eqnarray*}
\L(m,t):=\inf_{\C>0} \left\lbrace \|x-P_A(x)\|\le \C \|x-y\|: \begin{array}{ll}
\forall x\in \X,\ \forall A\in \G(x,m,t),\\  
 \forall y\in \X: |\supp(y)|\le m
\end{array}
\right\rbrace.
\end{eqnarray*}

\item The \emph{disjoint Lebesgue parameter}, $\L_d(m,t) =\L_d(m,t)[\B,\X]$, is given by

$$
\L_d(m,t):=\displaystyle \inf_{\C>0} \left\lbrace \|x-P_A(x)\|\le \C \|x-y\|: \begin{array}{ll}
\forall x\in \X,\ \forall A\in \G(x,m,t),  \forall y\in \X:\\  
 |\supp(y)|\le m,\ \supp(y)\cap A=\emptyset
\end{array}
\right\rbrace.
$$
\end{itemize}
\end{definition}

The Lebesgue parameter for $t=1$ has been widely been studied, in the literature, for example in \cite{BBG2017}, \cite{BBGHO2018}, \cite{GHO2013}, \cite{Oswald2001}, among others. The parameter involving the WTGA for any $0<t\le 1$ has been studied for example in \cite{DKO2015} and \cite{BBGHO2018}. The disjoint parameters are weaker variants suggested by  Theorem~\ref{theoremswemigreedydisjointgreedynotc0}. Our next result gives upper bounds for  the disjoint Lebesgue parameter in terms of the lower Chebyshevian Lebesgue parameter. 

\begin{proposition}\label{propositionsemigreedydisjointgreedy}Let $\B$ be a basis for $\X$ and $\M:=\M_{fs}(\B, \X)$. For all $0<s\le 1$, $0<t\le 1$ and $m\in \N$, we have the following estimates: 
\begin{enumerate}[\rm (i)]
\item \label{lebesgue1}
 $\L_d(m,t)\le \M\L_{ch}^l(2m,s)(1+2(\M+1)\L_{ch}^l(m,s)t^{-1}s^{-2})$.
\item \label{lebesgue1v} If $m$ is even, in addition we have: 
$$
\L_d(m,t)\le \M\L_{ch}^l(2m,s)(1+4\L_{ch}^l(m,s)t^{-1}s^{-2}).
$$ 
\item \label{lebesgue1v2} If $m$ is odd, $m>1$, in addition we have:
$$
\L_d(m,t)\le \M\L_{ch}^l(2m,s)(1+(4\L_{ch}^l(m-1,s)+2\L_{ch}^{l}(2,s))t^{-1}s^{-2}).
$$
\end{enumerate}
In particular, if $\B$ is $\C$-$s$-weak semi-greedy, it is $\K$-almost greedy with 

$$
\K\le\min\left\lbrace \M\C(1+2\C(\M+1) s^{-2}, \max\{ \M\C(1+6\C s^{-2}), 1+\lambda\lambda'+\lambda'' \}\right\rbrace. 
$$
\end{proposition}

\begin{proof}  Set $\B=( \xx_i)_{i\in\N}$ and $\B^*=( \xx^*_i)_{i\in\N}$.  
Choose $0<\epsilon<1$, and let $(\xx_{i_k})_{k\in \N}$ be a subsequence given by Corollary~\ref{corollaryboth} for $(\X, \M, \epsilon)$.  To prove \ref{lebesgue1}, fix $x\in \X$ with finite support and $A\in \G(x,m,t)$. We may assume that $P_A(x)\not=x$, so $|\supp(x)|>m$. Pick $y\in \X$  so that
$$
|\supp(y)|\le m\qquad\text{and} \qquad \supp(y)\cap A=\emptyset,
$$
set $a:=\min_{i \in A}|\xx_i^*(x)|$ and choose $i_0>\supp(x)\cup \supp(y)$. Now define
\begin{align*}
\F:=&[ \xx_i: 1\le i\le i_0];\\
E_l:=&\{i_{2s_{\F,\epsilon}+j+lm }: 1\le j\le m\}, \quad \forall  l\in \{0,1\} ;\\
z:=&x-P_A(x)+a(1+\epsilon)s^{-1}t^{-1}(\1_{E_0}+\1_{E_1}).
\end{align*}
We have
$$
|\xx_i^*(z)|=|\xx_i^*(x-P_A(x))|\le t^{-1}a,\qquad \forall i\not\in E_0\cup E_1.  \label{tgreedycondition0}
$$
Thus, 
$$
\G(z,2m,s)=\{E_0\cup E_1\}. 
$$
It follows that there is $v\in \X$ with $\supp(v)\subset E_0\cup E_1$ such that 
\begin{equation*}
\|z-v\|\le \L_{ch}^l(2m,s)\sigma_{2m}(z). 
\end{equation*}
Hence,
\begin{align}
\|x-P_A(x)\| & \le (\M+\epsilon)\|z-v\|\le (\M+\epsilon)\L_{ch}^l(2m,s)\|z+P_A(x)-y\|\nonumber\\
&\le (\M+\epsilon)\L_{ch}^l(2m,s)(\|x-y\|+a(1+\epsilon)s^{-1}t^{-1}\|\1_{E_0}+\1_{E_1}\|).\label{leftbound}
\end{align}
Set  
$$
u_l:=x-y-a(1-\epsilon)s\1_{E_l},\qquad \forall l\in  \{0,1\}.
$$
Given that $\supp(y)\cap A=\emptyset$, we have
$$
|\{1\le i\le i_0: |\xx_i^*(x-y)|\ge a\}|\ge |A|= m. 
$$
It follows that for $l\in \{0,1\}$, every element of $\G(u_l,m,s)$ is contained in $\{1,\dots,i_0\}$, so there is $v_l\in \F$ such that
\begin{equation*}
\|u_l-v_l\|\le \L_{ch}^l(m,s) \sigma_m(u_l),\qquad \forall  l\in  \{0,1\}. 
\end{equation*}
This entails that
\begin{align}
a(1-\epsilon)s\|\1_{E_l}\|\le& \|u_l-v_l\|+\|x-y-v_l\|\le (1+\M+\epsilon)\|u_l-v_l\|\nonumber\\
\le& (1+\M+\epsilon)\L_{ch}^l(m,s) \sigma_m(u_l)\nonumber\\
\le& (1+\M+\epsilon)\L_{ch}^l(m,s) \|x-y\|,\qquad \forall l\in \{0,1\}. \label{foreven}
\end{align}
Combining the above estimate with \eqref{leftbound}, we get 
\begin{align*}
& \|x-P_A(x)\|\nonumber\\
&\le  (\M+\epsilon)\L_{ch}^l(2m,s)(1+2(1+\M+\epsilon)\L_{ch}^l(m,s)(1+\epsilon)(1-\epsilon)^{-1}s^{-2}t^{-1})\|x-y\|.
\end{align*}
As $\epsilon$ is arbitrary, the proof of \ref{lebesgue1} is complete.\\
To prove \ref{lebesgue1v}, we use the same argument, with only the following modifications: For $l\in \{0,1\}$ and $n\in \N$, define $\varepsilon_n:=(-1)^{n+1}$, and substitute $\1_{\varepsilon, E_l}$ for $\1_{E_l}$ in the definitions of $z$ and $u_l$. \\
When $|E_l|=m$ is even, $\1_{\varepsilon,E_l}$ is a sum that meets the conditions of Corollary~\ref{corollaryboth},  so instead of \eqref{foreven} we obtain 
\begin{align*}
a(1-\epsilon)s\|\1_{\varepsilon, E_l}\|\le& \|u_l-v_l\|+\|x-y-v_l\|\le (1+1+\epsilon)\|u_l-v_l\|\nonumber\\
\le& (2+\epsilon)\L_{ch}^l(m,s) \sigma_m(u_l)\le (2+\epsilon)\L_{ch}^l(m,s) \|x-y\|\quad \forall l\in \{0,1\},  \nonumber
\end{align*}
and the result follows as before. \\
Finally, suppose $m$ is odd, $m>1$. The proof is as that for $m$ even, with the following modifications: For $l\in \{0,1\}$, let 
\begin{align*}
E_l:=&\{i_{2s_{\F,\epsilon}+j+l(m-1) }: 1\le j\le m-1\},
\end{align*}
and let $E_2:=\{i_{2s_{\F,\epsilon}+2m-1}, i_{2s_{\F,\epsilon}+2m}\}$, with $\varepsilon_n:=(-1)^{n+1}$ as before. Set 
$$
z:=x-P_A(x)+a(1+\epsilon)s^{-1}t^{-1}(\1_{\varepsilon,E_0}+\1_{\varepsilon, E_1}+\1_{\varepsilon,E_2}).
$$
Since $|E_{0}|=|E_{1}|=m-1$, for $l\in \{0,1\}$ we get 
\begin{align}
a(1-\epsilon)s\|\1_{\varepsilon, E_l}\|\le& \|u_l-v_l\|+\|x-y-v_l\|\le (1+1+\epsilon)\|u_l-v_l\|\nonumber\\
\le& (2+\epsilon)\L_{ch}^l(m-1,s) \sigma_{m-1}(u_l)\le (2+\epsilon)\L_{ch}^l(m-1,s) \|x-y\|,\nonumber
\end{align}
whereas $|E_2|=2$, so 
\begin{align}
a(1-\epsilon)s\|\1_{\varepsilon, E_2}\|\le& \|u_2-v_2\|+\|x-y-v_2\|\le (1+1+\epsilon)\|u_2-v_2\|\nonumber\\
\le& (2+\epsilon)\L_{ch}^l(2,s) \sigma_2(u_2)\le (2+\epsilon)\L_{ch}^l(2,s) \|x-y\|. \nonumber
\end{align}
and the proof is completed as in the even case. \\
Finally, suppose that $\B$ is $\C$-$s$-weak semi-greedy. Combining \ref{lebesgue1}, \ref{lebesgue1v}, \ref{lebesgue1v2} we get that 
$$
\L_d(m,1)\le \M\C \min\{ 1+2\C(\M+1) s^{-2}, 1+6\C s^{-2} \} 
$$
for all $m\ge 2$. Hence, using Lemma~\ref{lemmadisjoint=}, to complete the proof of the remaining inequality, we only need to prove that if $A=\{k\}\in \G(x,1,1)$ and $j\not =k$, then 
$$
\|x-P_A(x)\|\le (1+\lambda\lambda'+\lambda'') \|x-\xx_{j}^*(x)\xx_j\|.
$$
But this follows easily from the triangle inequality and the fact that $A$ is a greedy set for $x$:
\begin{align*}
\|x-P_A(x)\|\le& \|x-\xx_{j}^*(x)\xx_j\|+\|P_A(x)\|+\|\xx_{j}^*(x)\xx_j\|\\
=&\|x-\xx_{j}^*(x)\xx_j\|+\|P_A(x-\xx_{j}^*(x)\xx_j)\|+|\xx_{j}^*(x)|\|\xx_j\|\\
\le& (1+\lambda'')\|x-\xx_{j}^*(x)\xx_j\|+|\xx_{k}^*(x-\xx_{j}^*(x)\xx_j)|\lambda\\
\le& (1+\lambda''+\lambda\lambda')\|x-\xx_{j}^*(x)\xx_j\|.
\end{align*}
\end{proof}

Next, we consider parameters  involving only projections rather than arbitrary vectors. 

\begin{definition}\label{definitionlebesguedisjointalmost}Let $\B$ be a basis for $\X$, $0<t\le 1$ and $m\in \N$.
\begin{itemize}
\item The \emph{almost greedy parameter}, $\L_{a}(m,t) =\L_{a}(m,t)[\B,\X]$, is given by 
$$
\L_{a}(m,t) :=\displaystyle \inf_{\C>0} \left\lbrace \|x-P_A(x)\|\le \C \|x-P_B(x)\|: \begin{array}{ll}
\forall x\in \X,\ \forall A\in \G(x,m,t),\\  
 \forall B\in \subset \N, |B|= m
\end{array}
\right\rbrace.
$$
\item The \emph{disjoint almost greedy parameter}, $\L_{a,d}(m, t)=\L_{a,d}(m,t)[\B,\X]$, is given by 

\begin{align*}
\L_{a,d}(m, t):=\inf_{\C>0} \left\lbrace \|x-P_A(x)\|\le \C \|x-P_B(x)\|:\begin{array}{ll} \forall x\in \X,\ \forall A\in \G(x,m,t),\\  B\subset \N, |B|= m: B\cap A=\emptyset
\end{array}
\right\rbrace.
\end{align*}
\end{itemize}
\end{definition}

In the case $t=1$, the almost greedy parameter has been studied for example in \cite{AAB2021}, \cite{BBG2017} and \cite{GHO2013}. The disjoint variants are weaker versions naturally suggested by Lemma~\ref{lemmadisjoint=}. \\
It is immediate that $\L_{a,d}(m, t)\le \min\{\L_{a}(m, t),\L_d(m,t)\}$, and the arguments of Lemma~\ref{lemmadisjoint=} show that
$$
\L_{a}(m, t)\le \max_{1\le j\le m}\L_{a,d}(m, t). 
$$
Indeed, if $x\in \X$, $A\in \G(x,m,t)$ and $B\subset \N$, with $|B|=m$, are such that $0<\min\{|A\cap B|,|A\setminus B|\}$, then $A\setminus B\in \G(x-P_{A\cap B}(x), |A\setminus B|,t)$, so 
\begin{align*}
\|x-P_A(x)\|=&\|x-P_{A\cap B}(x)-P_{A\setminus B}(x-P_{A\cap B}(x))\|\\
\le& \L_{a,d}(|A\setminus B|, t)\|x-P_{A\cap B}(x)-P_{B\setminus A}(x-P_{A\cap B}(x))\|\\
\le&  \max_{1\le j\le m}\L_{a,d}(j, t)\|x-P_B(x)\|.
\end{align*}
Thus, Proposition~\ref{propositionsemigreedydisjointgreedy} also can be used to obtain bounds for $\L_{a}(m,t)$.  In the case $t=1$, another route to find such upper bounds is to combine Proposition~\ref{propositionchevyshevianqg} and the results of Section~\ref{sectionseparation} with  \cite[Theorem 3.3]{AAB2021}. First, we find estimates for the squeeze symmetry parameter, defined as follows. 

\begin{definition}\cite[Lemma~1.4(iii)]{AAB2021} \label{definitionsqueezeparameter} Let $\B$ be a basis for $\X$ with dual basis  $\B^*=( \xx^*_i)_{i\in\N}$ and  $m\in \N$. The $m$-\emph{squeeze symmetry parameter}, $\lam_m=\lam_m[\B,\X]$, is given by 

\begin{align*}
\lam_m:=\inf_{\C>0} \left\lbrace \min_{n\in A}|\xx_n^*(x)|\|1_{\varepsilon,B}\|\le \C\|x\|:
\begin{array}{ll} 
\forall x\in \X,\ \forall A\in \G(x,m,1),\\  \forall B\subset \N: |B|=m, \varepsilon\in \EE_B
\end{array}
\right\rbrace.
\end{align*}
\end{definition}
\begin{proposition}\label{propositionsqueezeparameter}Let $\B$ be a basis for $\X$ and  $\M:=\M_{fs}(\B, \X)$. For all $0<s\le 1$ and  $m\in \N$ we have the following estimates: 
\begin{enumerate}[\rm (i)]
\item $\lam_m\le (\L_{ch}^l(m,s))^2\M(1+\M)s^{-2},\qquad\forall m\in \N$. \label{propositionsqueezeparameter0}
\item $\lam_m\le 2(\L_{ch}^l(m,s))^2\M s^{-2},\quad \qquad \qquad\forall m\in 2\N$.\label{propositionsqueezeparameter2}
\item $\lam_m\le 2(\L_{ch}^l(m-1,s))^2\M s^{-2}+\lambda\lambda',\quad\forall m\in 2\N-1,$ where $\L_{ch}(0,s):=0$. 
\label{propositionsqueezeparameter1}
\end{enumerate}
\end{proposition}

\begin{proof}
The proof is very similar to those of Proposition~\ref{propositionsemigreedydisjointgreedy} and Theorem~\ref{theoremswemigreedydisjointgreedynotc0}\ref{wdisjointgreedyest2a} but simpler, so we shall be brief.  Set $\B=( \xx_i)_{i\in\N}$ and $\B^*=( \xx^*_i)_{i\in\N}$.
Choose $0<\epsilon<1$, and let $(\xx_{i_k})_{k\in \N}$ be a subsequence given by Corollary~\ref{corollaryboth} for $(\X, \M, \epsilon)$. 
Now fix $x\in \X$ with finite support, $A\in \G(x,m,1)$, and $B\subset \N, |B|=m$. We may assume $a:=\min_{i \in A}|\xx_i^*(x)|>0$. Choose $i_0>B\cup \supp(x)$, and  set
\begin{align*}
\F:=&[\xx_i: 1\le i\le i_0];\quad E_0:=\{i_{2s_{\F,\epsilon}+j}: 1\le j\le m\};\quad z:=a\1_{\varepsilon,B} +a(1+\epsilon)s^{-1}\1_{E_0}.
\end{align*}
Since $\G(z,m,s)=\{E_0\}$, there is $v\in \X$ with $\supp(v)\subset E_0$ such that $\|z-v\|\le \L_{ch}^l(m,s)\sigma_{m}(z)$. Hence,
\begin{align*}
\|a\1_{\varepsilon,B}\|\le& (\M+\epsilon)\|z-v\|\le (\M+\epsilon)\L_{ch}^l(m,s)\|z-a\1_{\varepsilon,B}\|\\
=&(\M+\epsilon)\L_{ch}^l(m,s)a(1+\epsilon)s^{-1}\|\1_{E_0}\|.\nonumber
\end{align*}
Let $u:= x-a(1-\epsilon)s\1_{E_0}$. Since every element of $\G(u,m,s)$ is contained in $\{1,\dots,i_0\}$, there is $v\in \F$ such that
\begin{equation*}
\|u-v\|\le \L_{ch}^l(m,s) \sigma_m(u).
\end{equation*}
This entails that
\begin{align}
a(1-\epsilon)s\|\1_{E_0}\|\le& \|u-v\|+\|x-v\|\le (1+\M+\epsilon)\|u-v\|\nonumber\\
\le& (1+\M+\epsilon)\L_{ch}^l(m,s) \|x\|.\label{propositionsqueezeparameter0generalcase}
\end{align}
Combining the above estimates and taking into account that $\epsilon$ is arbitrary, we get 
\begin{align*}
\|a\1_{\varepsilon,B}\|\le (\L_{ch}^l(m,s))^2\M(1+\M)s^{-2}    \|x\|, 
\end{align*}
so the proof of \ref{propositionsqueezeparameter0} for finitely supported vectors is complete. The general case follows using some of the arguments of Theorem~\ref{theoremswemigreedydisjointgreedynotc0}.\\
The proof of \ref{propositionsqueezeparameter2} is like that of \ref{propositionsqueezeparameter0}, with only the following differences: the set $E_0$ has a partition $E_1\cupdot E_2$ with $|E_1|=|E_2|=\frac{m}{2}$ resulting from an application of Corollary~\ref{corollaryboth}, and instead of $\1_{E_0}$ we define $z$ and $u$ using $\1_{E_1}-\1_{E_2}$, so that instead of \eqref{propositionsqueezeparameter0generalcase} we get 
\begin{align*}
a(1-\epsilon)s\|\1_{E_1}-\1_{E_2}\|\le& \|u-v\|+\|x-v\|\le (1+1+\epsilon)\|u-v\|\nonumber\\
\le& (2+\epsilon)\L_{ch}^l(m,s) \|x\|.
\end{align*}
Finally, to prove \ref{propositionsqueezeparameter1}, write $B=B_0\cupdot \{j_0\}$ and $A=A_{0}\cup \{k_0\}$. Then 
$$
a\|\xx_{j_0}\|\le \lambda\lambda'\|x\|
$$
and, if $m>1$, by \ref{propositionsqueezeparameter2} we have
$$
a\|\1_{\varepsilon, B_0}\|\le \min_{n\in A_0}|\xx_n^*(x)|\|\1_{\varepsilon, _0}\|\le  2(\L_{ch}^l(m-1,s))^2\M s^{-2}\|x\|.
$$
Applying the triangle inequality we get the desired result. 
\end{proof}

Combining Propositions~\ref{propositionchevyshevianqg} and~\ref{propositionsqueezeparameter} with \cite[Theorem 3.3]{AAB2021}, we obtain the following.
\begin{corollary}\label{corollaryalternative} Let $\B$ be a basis for $\X$ and  $\M:=\M_{fs}(\B, \X)$. For $0<s\le 1$ and  $m\in \N$ let 
$$
p_m:=\max_{1\le j\le m} \L_{ch}^l(2\floor{\frac{j+1}{2}},s) (1+ 2\L_{ch}^l(2\floor{\frac{j+1}{2}},s) s^{-2}).
$$
The following hold: 
\begin{align*}
\L_{a}(m,1)\le& \M p_{m}+ (\L_{ch}^l(m,s))^2\M(1+\M)s^{-2}&&\forall m\in \N;\\
\L_{a}(m,1)\le& \M p_{m}+ 2(\L_{ch}^l(m,s))^2\M s^{-2}&&\forall m\in 2\N;\\
\L_{a}(m,1)\le& \M p_{m}+2(\L_{ch}^l(m-1,s))^2\M s^{-2}+\lambda\lambda'&&\forall m\in 2\N-1.
\end{align*}
In particular, if $\B$ is $\C$-$s$-weak semi-greedy, it is $\K$-almost greedy with 
$$
\K\le \M \C(1+2\C s^{-2})+\min\{\C^2\M(1+\M)s^{-2}, 2\C^2\M s^{-2}+\lambda\lambda'\}.
$$
\end{corollary}

\section{Examples}\label{sectionexamples}

In this section, we construct bases with some of the properties we have studied. We leave aside the case $\w\in \ell_1$ because in that case any $\w$-democratic basis is equivalent to the canonical unit vector basis of $\mathtt{c}_0$ (see  Remark~\ref{remarkquantitativec0}). Similar considerations apply to the case $\w\not \in \ell_{\infty}$ (see \cite{B2020} and  \cite{DKTW2018}). Also,  it is known that if $\w\in\ell_{\infty}$ is a weight, there is a $\w$-greedy basis which is not equivalent to the canonical basis of $\mathtt{c}_0$. Indeed, for $1<p<\infty$, if $\X_p$ is the completion of $\mathtt{c}_{00}$ with the norm  
\begin{align}
\|(a_n)_{n\in\N}\|=&\max\{\|(a_n)_{n\in\N}\|_{\infty},(\sum_{n\in \N}w_n|a_n|^p)^{\frac{1}{p}}\},\label{Xp}
\end{align}
the canonical vector basis $\B_p$ of $\X_p$ is $1$-unconditional, and $\|\1_{\varepsilon, A}\|=\max\{1,(w(A))^{\frac{1}{p}}\}$ for all $A\in \N^{<\infty}$ (see \cite[Remark 4.10]{DKTW2018}) for $p=2$). For this reason, we will focus on constructing conditional bases. When $\w$ is seminormalized, the weighted properties are equivalent to their standard counterparts - that is, those involving constant weights -, and there are many examples of conditional bases showing a broad overview of the particularities of the different greedy-type bases (see e.g., \cite{AABBL2021}, \cite{AABW2021}, \cite{AADK2019}, \cite{DHK2006}, \cite{DKK2003},  \cite{KT1999} and  \cite{W2000}, among others). On the other hand, we are not aware of any examples in the literature of conditional $\w$-almost greedy bases in either of the following cases: $\w\in \mathtt{c}_{0}\setminus \ell_1$, and  $\w\in \ell_{\infty}\setminus \mathtt{c}_{0}$ with $(w_{n}^{-1})_{n\in\N}\not\in \ell_{\infty}$. Our purpose, in this section, is to give examples of such bases,  as well as examples of $\w$-almost semi-greedy bases that are  not $\w$-almost greedy. We begin with the weight in $\mathtt{c}_{0}\setminus \ell_1$ defined by
$$
\w^{(1)}:=(n^{-\frac{1}{2}}\log(n+1))_{n\in\N}.
$$

Our first task is to construct a conditional $\w^{(1)}$-almost greedy basis or, equivalently by \cite[Theorem 2.6]{DKTW2018}, quasi-greedy and $\w^{(1)}$-democratic. We will need an elementary lemma. 
\begin{lemma}\label{lemmalogcomputation}For all $A\in \N^{<\infty}$,
\begin{align*}
\sum_{n\in A}n^{-\frac{3}{4}}\le& 4 (\sum_{n\in A}n^{-\frac{1}{2}}\log(n+1))^{\frac{1}{2}}.
\end{align*}
\end{lemma}

\begin{proof}
By induction on $|A|$. If $|A|=1$, the result follows by an immediate computation. Fix $n_0\in \N$, and suppose the result holds for $1\le |A|\le n_0$. If $|A|=n_0+1$, let $n_1:=\max A$, and $A_1:=A\setminus\{n_1\}$. Using the inductive hypothesis, we obtain 
\begin{align*}
&\big(\sum_{n\in A}n^{-\frac{3}{4}}\big)^2=\big(\sum_{n\in A_1}n^{-\frac{3}{4}}\big)^2\ +\ n_1^{-\frac{3}{2}} + 2 n_1^{-\frac{3}{4}}(\sum_{n\in A_1}n^{-\frac{3}{4}})\\
\le & 16 \sum_{n\in A_1}n^{-\frac{1}{2}}\log(n+1)\ +\ n_1^{-\frac{3}{2}}\ +\ 8 n_1^{-\frac{3}{4}}\Big(\sum_{n\in A_1}n^{-\frac{1}{2}}\log(n+1)\Big)^{\frac{1}{2}}\\
= & 16\sum_{n\in A}n^{-\frac{1}{2}}\log(n+1)\ +\ 8 n_1^{-\frac{3}{4}}\Big(\sum_{n\in A_1}n^{-\frac{1}{2}}\log(n+1)\Big)^{\frac{1}{2}}\ + n_1^{-\frac{3}{2}}-16n_1^{-\frac{1}{2}}\log(n_1+1) \\
\le &16 \sum_{n\in A}n^{-\frac{1}{2}}\log(n+1)\ +\ 8 n_1^{-\frac{3}{4}}\Big(\sum_{n\in A_1}n^{-\frac{1}{2}}\log(n+1)\Big)^{\frac{1}{2}}-15n_1^{-\frac{1}{2}}\log(n_1+1)\\
\le &  16 \sum_{n\in A}n^{-\frac{1}{2}}\log(n+1)\ +\ n_1^{-\frac{1}{2}}\log^{\frac{1}{2}}(n_1+1)\Big(8 n_1^{-\frac{1}{4}}\big(\sum_{n\in A_1}n^{-\frac{1}{2}}\big)^{\frac{1}{2}}-15\log^{\frac{1}{2}}(n_1+1)\Big)\\
\le& 16 \sum_{n\in A}n^{-\frac{1}{2}}\log(n+1)\ +\ n_1^{-\frac{1}{2}}\log^{\frac{1}{2}}(n_1+1)\big(8\sqrt{2}-15 \log^{\frac{1}{2}}(3)\big)\\
<& 16 \sum_{n\in A}n^{-\frac{1}{2}}\log(n+1),
\end{align*}
and the proof is complete. 
\end{proof}
Now we can construct the example, which is a variant of one given in \cite{KT1999} (see also, \cite[página 35]{T2011} or \cite[página 266]{T2008}).

\begin{example}\label{examplewalmostc0}For each $(a_n)_{n\in\N}\in \mathtt{c}_{00}$, define 
\begin{align*}
&\|(a_n)_{n\in\N}\|_{\diamond}:=\|( (w^{(1)}_n)^{\frac{1}{2}} a_n)_{n\in\N}\|_2 &\rm and \qquad &\|(a_n)_{n\in\N}\|_{\circ}:=\sup_{m\in \N}|\sum_{n=1}^{m}n^{-\frac{3}{4}}a_n|.
\end{align*}
Let $\X$ be the completion of $\mathtt{c}_{00}$ with the norm 
$$
\|(a_n)_{n\in\N}\|:=\max\{\|(a_n)_{n\in\N}\|_{\diamond}, \|(a_n)_{n\in\N}\|_{\circ},\|(a_n)_{n\in\N}\|_{\infty}\},
$$
and let $\B$ be the canonical vector basis and $\B^*$ its dual basis. Then $\B$ is a normalized monotone conditional $\w^{(1)}$-almost greedy Schauder basis. 
\end{example}

\begin{proof}
It is clear that $\B$ is a normalized monotone Schauder basis. Now fix $A\in \N^{<\infty}$ and $\varepsilon\in \EE_A$. We have 
$$
\|\1_{\varepsilon, A}\|_{\diamond}=(w^{(1)}(A))^{\frac{1}{2}}
$$
and, by Lemma~\ref{lemmalogcomputation}, 
$$
\|\1_{\varepsilon, A}\|_{\circ}\le \sum_{n\in A}n^{-\frac{3}{4}}\le 4 (w^{(1)}(A))^{\frac{1}{2}}. 
$$
Thus, 
$$
\max\{1,(w^{(1)}(A))^{\frac{1}{2}}\}\le \|\1_{\varepsilon, A}\|\le \max\{1, 4(w^{(1)}(A))^{\frac{1}{2}}\}. 
$$
It follows that $\B$ is $4$-$\w^{(1)}$-superdemocratic. To prove that it is quasi-greedy, fix $x\in \X$ with $\|x\|=1$, $m\in \N$, $A\in \G(x,m,1)$, and define $a:=\min_{n\in A}|\xx_n^*(x)|$. We  may assume $a>0$. Let $k_0:=\floor{ a^{-4}}$ and $B:=\{1,\dots, k_0\}$. We have
\begin{align*}
\|P_{A\cap B}(x)\|_{\circ}\le& \|P_B(x)\|_{\circ}+\|P_{B\setminus A}(x)\|_{\circ}\le 1+a\sum_{n\in B}n^{-\frac{3}{4}}\le 1+4a k_0^{\frac{1}{4}}\le 5,
\end{align*}
whereas, for $n>k_0$, $n\in A$, $|\xx_n^{*}(x)|>n^{-\frac{1}{4}}$, and we have
\begin{align*}
\|P_{A\setminus B}(x)\|_{\circ}\le& \sum_{\substack{n\in A\\ n> k_0}} n^{-\frac{3}{4}}|\xx_n^*(x)|  \le  \sum_{\substack{n\in A\\ n> k_0}} n^{-\frac{1}{2}}|\xx_n^*(x)|^2 \le \sum_{\substack{n\in A}}  n^{-\frac{1}{2}}\log(n+1)|\xx_n^*(x)|^2\le 1.
\end{align*}

As the respective inequalities are immediate for $\|\cdot\|_{\diamond}$ and $\|\cdot\|_{\infty}$, this proves that $\B$ is $6$-quasi-greedy. It only remains to prove that it is conditional. For each $m\in \N$, define 
\begin{align*}
&y_m:=\sum_{n=1}^{m}n^{-\frac{1}{4}}\log^{-1}(n+1)\xx_n & \text{and}\qquad &z_m:=\sum_{n=1}^{m}(-1)^{n}n^{-\frac{1}{4}}\log^{-1}(n+1)\xx_n.
\end{align*}
Then $\|y_m\|_{\infty}=\|z_m\|_{\infty}= \log^{-1}(2)$ and
\begin{align*}
\|y_m\|_{\diamond}=&\|z_m\|_{\diamond}=(\sum_{n=1}^{m}n^{-1}\log^{-1}(n+1))^{\frac{1}{2}}.
\end{align*}
On the other hand,
\begin{align*}
\|y_m\|_{\circ}=&\sum_{n=1}^{m}n^{-1}\log^{-1}(n+1),
\end{align*}
whereas 
\begin{align*}
\|z_m\|_{\circ}=\max_{1\le k\le m}|\sum_{n=1}^{k}(-1)^{n}n^{-1}\log^{-1}(n+1)|\le 2\|z_m\|_{\infty}=2 \log^{-1}(2).
\end{align*}
Thus,
$$\|z_m\|^{-1}\|y_m\|\xrightarrow[m\to \infty] {}\infty, $$
and the proof is complete. 
\end{proof}
Next, we modify Example~\ref{examplewalmostc0} to obtain a $\w^{(1)}$-almost semi-greedy basis that is neither quasi-greedy nor, in any order, a Schauder basis. For the construction, we also adapt some of the arguments from \cite[Theorem 3.13 and Proposition 4.17]{AABBL2021}.

\begin{example}\label{examplewalmostsemic0}
Let $(A_m)_{m\in\N}\subset \N^{<\infty}$ be a sequence of nonempty integer intervals such that $A_m<A_{m+1}$ for all $m$, and 
$$
\big(\sum_{n\in A_m}n^{-1}\log^{-1}(n+1)\big)_{m\in\N} \subset (1,\infty)
$$
is an unbounded sequence. For each $(a_n)_{n\in\N}\in \mathtt{c}_{00}$, define 
\begin{align*}
&\|(a_n)_{n\in\N}\|_{\triangleleft}:=\sup_{m\in \N}|\sum_{n\in A_m}n^{-\frac{3}{4}}a_n|,
\end{align*}
and let $\|\cdot\|_{\diamond}$ be as in Example~\ref{examplewalmostc0}. Let $\X$ be the completion of $\mathtt{c}_{00}$ with the norm 
$$
\|(a_n)_{n\in\N}\|:=\max\{\|(a_n)_{n\in\N}\|_{\diamond}, \|(a_n)_{n\in\N}\|_{\triangleleft},\|(a_n)_{n\in\N}\|_{\infty}\},
$$
and let $\B$ be the canonical vector basis with $\B^*$ its dual basis. Then $\B$ is a normalized $\w^{(1)}$-almost semi-greedy basis that is not quasi-greedy nor, in any order, a Schauder basis. 
\end{example}

\begin{proof}
It is clear that $\B$ is normalized. First, we prove that $\B$ is a Markushevich basis: suppose otherwise, and  fix $x\in \X\setminus \{0\}$ with $\xx_k^*(x)=0$ for all $k\in\N$. Pick a sequence $(x_n)_{n\in\N}\subset [\B]\setminus \{0\}$ so that $\|x-x_n\|\le  n^{-1}$ for each $n\in \N$. Now choose a strictly increasing sequence $(m(n))_{n\in\N}\subset \N$ so that for each $n\in \N$, $A_{m(n)}>\max_{1\le k\le n}\supp(x_k)$, and set $B_n:=\{1,\dots, c_n:=\max(A_{m(n)})\}$. Let $n_1:=1$ and $y_1=x_1$. Given that for each $k\in\N$, 
$$ 
|\xx_k^*(x_n)|\xrightarrow[n\to\infty]{}0<\inf_{n\in \N}\|x_n\|,
$$  
there is $n_2>c_{n_1}$ such that $\|P_{B_{n_1}}(x_{n_2})\|\le 2^{-1}$ and $x_{n_2}\not=P_{B_{n_1}}(x_{n_2})$.  Hence, if $y_{2}:=x_{n_2}-P_{B_{n_1}}(x_{n_2})$, then 
$$
y_2\not=0,\quad \supp(y_2)>B_{n_1}\supset \supp(y_1), \quad \|y_2-x\|\le \|x_{n_2}-x\|+\|P_{B_{n_1}}(x_{n_2})\| \le 2^{-1}+2^{-1}.
$$ 
In addition, for each $m\in \N$, if $A_m\cap \supp(y_2)\not=\emptyset$, then $A_m>\supp(y_1)$. Similarly, we can find $n_3>c_{n_2} $ so that $\|P_{B_{n_2}}(x_{n_3})\|\le 3^{-1}$ and $x_{n_3}\not=P_{B_{n_2}}(x_{n_3})$. Hence, if  $y_3:=x_{n_3}-P_{B_{n_2}}(x_{n_3})$,
$$y_3\not=0,\quad \supp(y_3)>B_{n_2}\supset \supp(y_2), \quad \|y_3-x\|\le\|x_{n_3}-x\|+\|P_{B_{n_2}}(x_{n_3})\|\le \frac{2}{3},$$ 
and for each $m\in \N$, if $A_m\cap \supp(y_3)\not=\emptyset$, then $A_m>\supp(y_2)$.
In this manner, we find an increasing sequence $(n_k)_{k\in \N}$ so that $(y_k:=x_{n_k}-P_{B_{n_k}}(x_{n_k}))_{k\in \N}$ has the following properties: for each $k\in\N$, 
$$y_k\in [\B]\setminus\{0\},\qquad \supp(y_{k+1})>B_{n_k}\supset \supp(y_{k}), \qquad \|y_k-x\|\le 2 k^{-1}$$
and, for every $m, k\in \N$, if $A_m\cap \supp(y_{k})\not=\emptyset$ then $A_m\cap \supp(y_j)=\emptyset$ for all $j\ne k$.\\
For a contradiction, let us show that $(y_k)_{k\in\N}$ converges to $0$ and consequently $x=0$.
First, notice that as $ \supp(y_{k+1})\cap  \supp(y_{k}) =\emptyset$, $\|y_k\|_{\infty} \le \|y_k-y_{k+1}\|_{\infty}$ and $\|y_k\|_{\diamond} \le \|y_k-y_{k+1}\|_{\diamond}$. Also, as each $A_m$ intersects with only one $\supp(y_j)$,  $\|y_k\|_{\triangleleft} \le  \|y_k-y_{k+1}\|_{\triangleleft}$.
Thus,
$$
\|y_k\|\le \|y_k-y_{k+1}\|\le 4k^{-1}\xrightarrow[k\to \infty]{}0.
$$
Therefore, $\B$ is a Markushevich basis. \\
Next, we prove that $\B$ is $\w^{(1)}$-almost semi-greedy: by Proposition~\ref{propositionequivalencesalmostsemigreedy},  it suffices to show that $\B$ is truncation quasi-greedy and $\w^{(1)}$-superdemocratic. To prove the former, fix $x\in \X$, $m\in \N$ and $A\in \G(x,m,1)$. Let $a:=\min_{n\in A}|\xx_n^*(x)|$. For each $m\in \N$ for which $A\cap A_m\not=\emptyset$ and every $\varepsilon\in \EE_A$, by Lemma~\ref{lemmalogcomputation} we have 
\begin{align*}
a|\sum_{n\in A_m\cap A}n^{-\frac{3}{4}}\varepsilon_n|\le& a \sum_{n\in A\cap A_m}n^{-\frac{3}{4}}\le 4 a(\sum_{n\in A\cap A_m}n^{-\frac{1}{2}}\log(n+1))^{\frac{1}{2}}\\
\le&4 (\sum_{n\in A\cap A_m}n^{-\frac{1}{2}}\log(n+1) |\xx_n^*(x)|^2)^{\frac{1}{2}}\le 4\|x\|. 
\end{align*}
We conclude that $\B$ is $4$-truncation quasi-greedy. Now fix $A\in \N^{<\infty}$ and $\varepsilon\in \EE_A$, and let $\|\cdot\|_{e2}$ be the norm on $\mathtt{c}_{00}$ defined in Example~\ref{examplewalmostc0}. By the computations in the aforementioned example, we have
$$
(w^{(1)}(A))^{\frac{1}{2}}=\|\1_{\varepsilon, A}\|_{\diamond}\le \|\1_{\varepsilon, A}\|\le \|\1_{\varepsilon, A}\|_{e2}\le 4(w^{(1)}(A))^{\frac{1}{2}}.
$$
Thus, $\B$ is $4$-$\w^{(1)}$-superdemocratic. \\
Finally, we prove that $\B$ is neither quasi-greedy nor, in any order, a Schauder basis. Fix $\pi:\N\rightarrow\N$ a biyection, and let $\B_{\pi}:=(\xx_{\pi(n)})_{n\in\N}$ be the reordered basis. For each $m\in \N_{\ge 2}$, let $d_m:=|A_m|$, and write $D_m:=\pi^{-1}(A_m)=\{l_{m,1}<\cdots<l_{m,d_m}\}$. Now consider $1\le j_m<d_m$ defined by
$$
j_m:=\min_{1\le k < d_m}\left\lbrace\sum_{j=k+1}^{d_m}\pi(l_{m,j})^{-1}\log^{-1}(\pi(l_{m,j})+1)\le \sum_{j=1}^k\pi(l_{m,j})^{-1}\log^{-1}(\pi(l_{m,j})+1)\right\rbrace,
$$
Then, we have
\begin{align}
\sum_{j=j_m+1}^{d_m}\pi(l_{m,j})^{-1}\log^{-1}(\pi(l_{m,j})+1)\le&  \sum_{j=1}^{j_m}\pi(l_{m,j})^{-1}\log^{-1}(\pi(l_{m,j})+1)\nonumber\\
\le&  1+\sum_{j=j_m+1}^{d_m}\pi(l_{m,j})^{-1}\log^{-1}(\pi(l_{m,j})+1). \label{examplewalmostsemic0dif<1}
\end{align}
For each $m\ge 2$, define 
\begin{align}
z_{m}:&=\sum_{j=1}^{j_m}\pi(l_{m,j})^{-\frac{1}{4}}\log^{-1}(\pi(l_{m,j})+1)\xx_{\pi(l_{m,j})}-\sum_{j=j_m+1}^{d_m} \pi(l_{m,j})^{-\frac{1}{4}}\log^{-1}(\pi(l_{m,j})+1)\xx_{\pi(l_{m,j})},\label{examplewalmostsemic0-zm}
\end{align}
that satisfies
\begin{align}
\|z_m\|_{\infty}\le& 1,\qquad \|z_m\|_{\triangleleft}\le 1,\qquad \|z_m\|_{\diamond}=\big(\sum_{n\in A_m}n^{-1}\log^{-1}(n+1)\big)^{\frac{1}{2}}\label{examplewalmostsemic0forlater0};\\
\|\sum_{j=1}^{j_m}\xx_{\pi(l_{m,j})}^*(z_m)\xx_{\pi(l_{m,j})}\|_{\triangleleft}=&\sum_{j=1}^{j_m}\pi(l_{m,j})^{-1}\log^{-1}(\pi(l_{m,j})+1)\ge 2^{-1}\sum_{n\in A_m}n^{-1}\log^{-1}(n+1).\label{examplewalmostsemic0forlater}
\end{align}
Hence, 
$$
\liminf_{m\to \infty} \frac{\|z_m\|}{\|\sum_{j=1}^{j_m}\xx_{\pi(l_{m,j})}^*(z_m)\xx_{\pi(l_{m,j})}\|}=0,
$$
so $\B_{\pi}$ is not a Schauder basis. Moreover, when $\pi$ is the identity mapping on $\N$, we have $D_m=A_m$ and $E_m:=\{l_{m,1},\dots, l_{m,j_m}\}\in \G(z_m,j_m,1)$, which proves that $\B$ is not quasi-greedy. 
\end{proof}
In our next example, we use the previous one to construct a $\w^{(1)}$-almost semi-greedy conditional Schauder basis that is not quasi-greedy. In the construction, we will use the following lemma. 

\begin{lemma}\label{lemmagoesSchauder} Let $\B=(\xx_n)_{n\in\N}$ be a basis for $\X$ and $\w$ a weight. Define $\Y$ as the completion of $\mathtt{c}_{00}$ with the norm 
$$
\|(a_n)_{n\in\N}\|_{\Y}=\sup_{m\in \N}\sup_{1\le k\le m}\|\sum_{n=k}^{m}a_n\xx_n\|_{\X}, 
$$
 and let $\B_1:=(\yy_n)_{n\in\N}$ be the canonical basis of $\Y$. The following hold: 
\begin{enumerate}[\rm (i)]
\item \label{lemmagoesSchauderbasisnorm}$\B$ is a monotone Schauder basis and $\|\yy_n\|_{\Y}=\|\xx_n\|_{\X}$ for all $n\in \N$. 
\item \label{lemmagoesSchauder1A} For every $A\in \N^{<\infty}$ and every $\varepsilon\in \EE_A$, 
$$
\|\1_{\varepsilon, A}\|_{\Y}\le \max_{B\subset A}\|\1_{\varepsilon, B}\|_{\X}.
$$
\item \label{lemmagoesSchauderdem} If $\B$ is $\C$-$\w$-(super)democratic, so is $\B_1$.  
\item \label{lemmagoesSchaudertqg}If $\B$ is $\C$-truncation quasi-greedy, $\B_1$ is $2\kappa\C^2$-truncation quasi-greedy. 
\end{enumerate}
\end{lemma}

\begin{proof}
\ref{lemmagoesSchauderbasisnorm} and \ref{lemmagoesSchauder1A} are immediate from the definitions. To prove \ref{lemmagoesSchauderdem}, we consider the superdemocracy case, as the democracy one is proven in the same manner.  Fix $A,B\in \N^{<\infty}$ with $w(A)\le w(B)$, $\varepsilon\in \EE_A$ and $\varepsilon'\in \EE_B$. By \ref{lemmagoesSchauder1A},  
$$
\|\1_{\varepsilon, A}\|_{\Y}\le  \max_{D\subset A}\|\1_{\varepsilon, D}\|_{\X}\le \C \|\1_{\varepsilon',B}\|_{\X}\le  \C \|\1_{\varepsilon',B}\|_{\Y}.
$$
Now suppose $\B$ is $\C$-truncation quasi-greedy. Let $(\xx^*_i)_{i\in\N}$ and $(\yy^*_i)_{i\in\N}$ be the dual bases of $\B$ and $\B_1$, respectively. Given $y\in \Y$ with finite support and $A\in \G(y,m,1)$ for some $m\in \N$, set $x:=\sum_{n\in \supp(y)}\yy_n^*(y)\xx_n$. It follows from \ref{lemmagoesSchauder1A} and Remark~\ref{remarkpropc=tqg} that 
\begin{align*}
\min_{n\in A}|\yy_n^*(y)|\|\1_{\varepsilon(y), A}\|_{\Y}\le&  \min_{n\in A}|\xx_n^*(x)|\  \max_{B\subset A}\|\1_{\varepsilon(x), B}\|_{\X}  \le 2\kappa \C^2 \|x\|_{\X} \le 2\kappa \C^2\|y\|_{\Y}.
\end{align*}
This completes the proof of \ref{lemmagoesSchaudertqg} and of the lemma. 
\end{proof}
\begin{remark}\label{remarkgoesschauder}Note that Lemma~\ref{lemmagoesSchauder} does not require a totality hypothesis on $\B^*$. Also, one could note that the basis $\B_1$ obtained by this method inherits from $\B$ several properties studied in the context of greedy approximation in addition to those stated in the lemma, such as quasi-greediness, quasi-greediness for largest coefficients, unconditionality for constant coefficients, and bidemocracy. 
\end{remark}

\begin{example}\label{examplewalmostsemic0schauder}There is a $\w^{(1)}$-almost semi-greedy Schauder basis that is not quasi-greedy. 
\end{example}
\begin{proof}
We will use the construction and notation of Example~\ref{examplewalmostsemic0} unless otherwise specified, with $\X$ the space, $\B$ the $\w^{(1)}$-almost semi-greedy  basis of that example respectively, and $\pi:\N\rightarrow \N$ the identity mapping. Take $(z_m)_{m\in \N}$ as in \eqref{examplewalmostsemic0-zm}. Then for each $m\in \N$, 
\begin{align*}
\|z_m\|_{\triangleleft}\le 1\qquad\text{and}\qquad |\xx_n^*(z_m)|\le 1\;\ \forall n\in A_m=\supp(z_m). 
\end{align*}
Also, it follows from \eqref{examplewalmostsemic0dif<1} and \eqref{examplewalmostsemic0-zm} that 
\begin{align*}
0\le& \sum_{n\in A_m}n^{-\frac{3}{4}}\xx_n^*(z_m)\le 1. 
\end{align*}
From the above inequalities we deduce that there is a bijection $\rho_m:A_m\rightarrow A_m$  such that 
\begin{align}
\max_{\substack{k\in A_m}}|\sum_{n\le k} \rho_m(n)^{-\frac{3}{4}}\xx_{\rho_m(n)}^*(z_m)|\le 1,\label{new}
\end{align}
which entails that
\begin{align}
\max_{\substack{j,k\in A_m\\j\le k}}|\sum_{n=j}^{k} \rho_m(n)^{-\frac{3}{4}}\xx_{\rho_m(n)}^*(z_m)|\le 2.\label{examplewalmostsemic0schauderlessthan2}
\end{align}
Define $\rho:\N\rightarrow \N$ by 
\begin{align*}
\rho(n):=
\begin{cases}
\rho_m(n) & \text{if } n\in A_m;\\
n & \text{if } n\not\in \bigcup_{m\in\N}A_m. 
\end{cases}
\end{align*}
Then $\rho:\N\rightarrow \N$ is a bijection. For each $n\in \N$, let $\zz_n:=\xx_{\rho(n)}$, and let $\B_0:=(\zz_n)_{n\in\N}$. As $\supp_{\B_0}(z_m)=\supp_{\B}(z_m)=A_m$ for each $m\in \N$, it follows from \eqref{examplewalmostsemic0schauderlessthan2} that
\begin{align}
\sup_{m\in \N} \sup_{{\substack{j,k\in A_m\\j\le k}}}\color{black}|\sum_{n=j}^{k} \rho(n)^{-\frac{3}{4}}\zz_{n}^*(z_m)|\le 2.\label{examplewalmostsemic0schauderlessthan2allm}
\end{align}
On the other hand, for each $m\in \N$, it follows from \eqref{examplewalmostsemic0forlater} that 
\begin{align}
\|P_{\rho^{-1}(E_m),\B_0}(z_m)\|=\|P_{E_m,\B}(z_m)\|\ge& 2^{-1}\sum_{n\in A_m}n^{-1}\log^{-1}(n+1).\nonumber
\end{align}
Let $\B_1=(\yy_n)_{n\in\N}$ be the Schauder basis obtained from $\B_0$ by an application of Lemma~\ref{lemmagoesSchauder}. 
Then $\B_1$ is truncation quasi-greedy and $\w^{(1)}$-superdemocratic since, $\B_0$ (which is a reordering of $\B$) has these properties.
Hence, by Proposition~\ref{propositionequivalencesalmostsemigreedy},  $\B_1$ is $\w^{(1)}$-almost semi-greedy. 
\\ 
For each $m\in \N$, let 
$$
y_m:=\sum_{n\in \N}\zz_n^*(z_m)\yy_m.
$$
We have
$$
\|P_{\rho^{-1}(E_m),\B_1}(y_m)\|_{\Y}\ge \|P_{\rho^{-1}(E_m),\B_0}(z_m)\|\ge 2^{-1}\sum_{n\in A_m}n^{-1}\log^{-1}(n+1),\quad \forall m\in \N. 
$$
On the other hand, it follows from \eqref{examplewalmostsemic0forlater0} and \eqref{examplewalmostsemic0schauderlessthan2allm}  that 
\begin{align*}
\|y_m\|_{\Y}=& \sup_{{\substack{j,k\in A_m\\j\le k}}} \|\sum_{n=j}^{k}\zz_n^*(z_m)\zz_n\|_{\X}=\max\{ \|z_m\|_{\infty},\|z_m\|_{\diamond},\sup_{{\substack{j,k\in A_m\\j\le k}}}|\sum_{n=j}^{k} \rho(n)^{-\frac{3}{4}}\zz_{n}^*(z_m)|\}\\
\le& 2+ (\sum_{n\in A_m}n^{-1}\log^{-1}(n+1))^{\frac{1}{2}}.
\end{align*}
It follows that 
$$
\liminf_{m\to \infty}\frac{\|y_m\|_{\Y}}{\|P_{\rho^{-1}(E_m),\B_1}(y_m)\|_{\Y}}=0. 
$$
Since $E_m\in \G_{\B}(z_m, d_m,1)$, we have $\rho^{-1}(E_m)\in  \G_{\B_0}(z_m, d_m,1)= \G_{\B_1}(y_m, d_m,1)$. Therefore, $\B_1$ is not quasi-greedy. 
\end{proof}

\begin{remark}\label{remarkschauderbidem}While not the focus of this paper, we point out that Remark~\ref{remarkgoesschauder} combined with the proof of Example~\ref{examplewalmostsemic0schauder}  suggests a way of using the constructions of \cite[Theorems 3.6, 3.13]{AABBL2021} to obtain bidemocratic Schauder bases that are not quasi-greedy, for a wide range of fundamental functions. 
\end{remark}

Now we turn to the case of $\w\in \ell_{\infty}\setminus \mathtt{c}_{0}$ with $(w_{n}^{-1})_{n\in\N}\not\in \ell_{\infty}$. Here, we have more flexibility constructing conditional bases: given two weights $\w$, $\w'$, define their combined weight $\WW(\w,\w')$ by 
\begin{align*}
W_{2n-1}(\w,\w'):=w_n\qquad\text{and}\qquad W_{2n}(\w,\w'):= w'_n.
\end{align*}
We will show that given $\w\in \mathtt{c}_0$ and $\w'$ seminormalized, there is a $\WW(\w,\w')$-almost greedy conditional Schauder basis, and obtain similar results for the weighted almost semi-greedy property. To that end, we will combine the $\w$-greedy basis $\B_p$ of the space $\X_p$ which we defined using  \eqref{Xp}, with suitable conditional bases. First we need a technical lemma. 

\begin{lemma}\label{lemmawdemdouble}Let $\w, \w'$ be weights, $\C>0$, and $\B_1=(\xx_n)_{n\in\N}$, $\B_2=(\yy_n)_{n\in\N}$ bases for $\X$ and $\Y$ respectively.  Suppose that for all $A,B\in \N^{<\infty}$, $\|\1_{A}\|_{\X}\le \C \|\1_{B}\|_{\Y}$ whenever $w(A)\le w'(B)$. Then, for all $A,B\in \N^{<\infty}$ such that $w(A)\le 2w'(B)$,  
$$
\|\1_{A}\|_{\X}\le (2\C+\lambda_0 \lambda'_0)\|\1_{B}\|_{\Y},
$$
where $\lambda_0$ and $\lambda'_0$ are the maxima between the constants $\lambda$ and $\lambda'$ in \eqref{constants} for $\B_1$ and $\B_2$. 
\end{lemma}
\begin{proof} Fix $A, B\in \N^{<\infty}$ with  $w(A)\le 2w'(B)$. First choose two (possibly empty) sets $A_1,A_2$ as follows: $A_1$ is a subset of $A$ with maximum $w$-measure such that $w(A_1)\le w'(B)$, and $A_2$ is a subset of $A\setminus A_1$ with the same property. Now let $A_3:=A\setminus (A_1\cup A_2)$. We claim that $|A_3|\le 1$. Indeed, if this is false, choose $m_1,m_2\in A\setminus (A_1\cup A_2)$, with $m_1\not=m_2$. Then by our choice of $A_1$ and $A_2$, $w(A_1\cup \{m_1\})>w'(B)$ and $w(A_2\cup \{m_2\})>w'(B)$. Hence, $w(A)>2w'(B)$, a contradiction. We conclude that
$$
\|\1_{A}\|_{\X}\le \|\1_{A_1}\|_{\X}+\|\1_{A_2}\|_{\X}+\|\1_{A_3}\|_{\X}\le (2\C+\lambda_0\lambda_0')\|\1_{B}\|_{\Y}. 
$$
\end{proof}
\begin{remark}Note that Lemma~\ref{lemmawdemdouble} can  be applied to a weight $\w$ and a $\C$-$\w$-democratic basis $\B$, taking $\B_1=\B_2=\B$ and $\X=\Y$. 
\end{remark}

The next lemma forms the basis of our final construction. 

\begin{lemma}\label{lemmasumofspaces}Let $\B_1=(\xx_n)_{n\in\N}$ be a basis for $\X$ and $\B_2=(\yy_n)_{n\in\N}$ a basis for $\Y$. Define $\Z:=\X\oplus \Y$ with the norm $\|(x,y)\|_{\Z}:=\max\{\|x\|_{\X},\|y\|_{\Y}\}$, and let $\B=(\zz_n)_{n\in\N}$ be given by 
\begin{align*}
\zz_{2n-1}:=(\xx_n,0)\qquad\text{and}\qquad \zz_{2n}:=(0,\yy_n). 
\end{align*}
Then $\B$ is a basis for $\Z$, and the following hold:
\begin{enumerate}[ \rm (i)]
\item \label{lemmasumofspacesschauder} $\B$ is a Schauder basis if and only if $\B_1$ and $\B_2$ are Schauder bases. 
\item \label{lemmasumofspacesqg}$\B$ is unconditional (quasi-greedy) if and only if $\B_1$ and $\B_2$ are unconditional (quasi-greedy). 
\item \label{lemmasumofspacestcg} $\B$ is truncation quasi-greedy if and only if $\B_1$ and $\B_2$ are truncation-quasi-greedy. 
\item \label{lemmasumofspaceswdem} $\B$ is $\WW(\w,\w')$-democratic if and only if $\B_1$ is $\w$-democratic, $\B_2$ is $\w'$-democratic, and there is $\C>0$ such that for all  $A,B\in \N^{<\infty}$, we have
$$
\|\1_{A}\|_{\X}\le \C \|\1_{B}\|_{\Y}
$$
if $w(A)\le w'(B)$, and 
$$
\|\1_{A}\|_{\Y}\le \C \|\1_{B}\|_{\X}
$$
if $w'(A)\le w(B)$.
\end{enumerate}
\end{lemma}

\begin{proof}
\ref{lemmasumofspacesschauder} and \ref{lemmasumofspacestcg} are clear. To prove \ref{lemmasumofspacestcg}, suppose that $\B_j$ is $\C_j$-truncation quasi-greedy, $j=1,2$, and fix $z\in \Z$, $m\in \N$, and $A\in \G(z,m,1)$. Let $a:=\min_{i\in A}|\zz_{i}^*(z)|$ with $(\zz^*_i)_{i\in\N}$ the dual basis of $\B$. Write $z=(x,y)$, and set 
\begin{align*}
&A_1:=\{n\in \N: 2n-1\in A\}, &&A_2:=\{n\in \N: 2n\in A\}. 
\end{align*}
Note that $A_1\in \G(x,|A_1|,1)$ and $A_2\in \G(y,|A_2|,1)$. Hence, taking $\varepsilon(z)$,  $\varepsilon(x)$ and $\varepsilon(y)$ with respect to the respective bases, we have
\begin{align*}
a\|\1_{\varepsilon(z),A}\|_{\Z}\le& \max \{a\|\1_{\varepsilon(x),A_1}\|_{\X}, a\|\1_{\varepsilon(y),A_2}\|_{\Y}\} \le\max\{\C_1\|x\|_{\X},\C_2\|y\|_{\Y}\}\\
\le& \max\{\C_1,\C_2\}\|z\|_{\Z}.
\end{align*}
This proves that $\B$ is $\max\{\C_1,\C_2\}$-truncation quasi-greedy. \\
On the other hand, if $\B$ is $\C$-truncation quasi-greedy, it is immediate that so are $\B_1$ and $\B_2$.\\

It only remains to prove \ref{lemmasumofspaceswdem}. Suppose first that $\B$ is $\C$-$\WW(\w,\w')$-democratic, and fix $A, B\in \N^{\infty}$  with $w(A)\le w(B)$. Given that $w(D)=W(\w,\w')(2D-1)$ for all $D\in \N^{<\infty}$, we have
\begin{align*}
\|\1_{A}\|_{\X}=&\|\1_{2A-1}\|_{\Z}\le \C\|\1_{2B-1}\|_{\Z}=\C\|\1_{B}\|_{\X}. 
\end{align*}
Hence, $\B_1$ is $\C$-$\w$-democratic. Similarly, $\B_2$ is $\C$-$\w$-democratic. Now fix $A,B\in \N^{\infty}$ with $w(A)\le w'(B)$. We have
\begin{align*}
\|\1_{A}\|_{\X}=&\|\1_{ 2A-1}\|_{\Z} \le \C\|\1_{2B}\|_{\Z}=\C\|\1_{B}\|_{\Y}.
\end{align*}
The case $w'(A)\le w(B)$ is proven in the same manner.\\
Now suppose that $\B_1$ is $\C_1$-$\w$-democratic and $\B_2$ is $\C_2$-$\w$-democratic, and let $\C$ be as in the statement. Fix $A,B\in \N$ with $W(\w,\w')(A)\le W(\w,\w')(B)$, and define 
\begin{align*}
&A_1:=\{n\in \N: 2n-1\in A\}, &&A_2:=\{n\in \N: 2n\in A\};\\
&B_1:=\{n\in \N: 2n-1\in B\}, &&B_2:=\{n\in \N: 2n\in B\}.
\end{align*}
Since $W(\w,\w')(A)=w(A_1)+w'(A_2)$ and the same holds for $B$, we have
\begin{align*}
w(A_1)+w'(A_2)\le& w(B_1)+w'(B_2)\le 2\max\{w(B_1),w'(B_2)\}.
\end{align*}
Suppose first that the maximum above is $w(B_1)$. Then $\max\{w(A_1),w'(A_2)\}\le 2w(B_1)$. Taking $\lambda$ and $\lambda'$ with respect to the basis $\B$ of $\Z$, an application of Lemma~\ref{lemmawdemdouble} yields
\begin{align*}
\|\1_{A}\|_{\Z}=&\max\{\|\1_{A_1}\|_{\X},\|\1_{A_2}\|_{\Y}\} \le \max\{(2\C_1+\lambda\lambda')\|\1_{B_1}\|_{\X},(2\C+\lambda\lambda')\|\1_{B_1}\|_{\X}\}\\
\le& (\lambda\lambda'+2\max\{\C,\C_1\})\|\1_{B}\|_{\Y}.
\end{align*}
The same argument holds if the maximum is $w'(B_2)$ - we just get $\C_2$ instead of $\C_1$ in the upper bound. Thus, we conclude that $\B$ is  $(\lambda \lambda'+2\max\{\C,\C_1,\C_2\})$-$\WW(\w,\w')$-democratic. 
\end{proof}

\begin{corollary}\label{corollaryexamplescombined} Let $\w,\w'$ be weights, with $\w\in \mathtt{c}_{0}$ and $\w'$ seminormalized. The following hold: 
\begin{enumerate}[\rm (i)]
\item \label{corollaryexamplescombinedwalmost}There exists a conditional $\WW(\w,\w')$-almost greedy Schauder basis. 
\item \label{corollaryexamplescombinedwalmostsemiSch}There exists a conditional $\WW(\w,\w')$-almost semi-greedy Schauder basis that is not quasi-greedy. 
\item \label{corollaryexamplescombinedwalmostsemi}There exists a conditional $\WW(\w,\w')$-almost semi-greedy  basis that is not quasi-greedy nor, in any order, a Schauder basis. 
\end{enumerate}
\end{corollary}

\begin{proof}
Due to the equivalence of the above properties for equivalent weights, we may assume that $\w'$ is constant. To prove \ref{corollaryexamplescombinedwalmost}, pick $1<p<\infty$, and let $\B_p$ be the basis of the space $\X_p$ defined using \eqref{Xp}. Now choose a conditional almost greedy Schauder basis with $\|\1_{A}\|\approx |A|^{\frac{1}{p}}$ for all $A\in \N^{<\infty}$ (for example, apply \cite[Theorem 4.9]{AABBL2021}), and then apply Lemma~\ref{lemmasumofspaces}. Now  \ref{corollaryexamplescombinedwalmostsemiSch} is proven by the same argument as \ref{corollaryexamplescombinedwalmost}, with the only difference that, as the second basis in our construction, instead of a conditional almost greedy Schauder basis we choose an almost semi-greedy Schauder basis that is not quasi-greedy (in light of Corollary~\ref{corollarybidem}, bases with such properties can be found in  \cite[Proposition 3.17]{AABBL2021}). Finally, \ref{corollaryexamplescombinedwalmostsemi} is again proven using the above construction, but in this case, our second basis is almost semi-greedy but neither quasi-greedy nor, in any order, a Schauder basis (for example, we can find such bases in \cite[Theorem 3.13]{AABBL2021}).
\end{proof}

\section{Open Questions}
In Theorem~\ref{theoremc0notl1norming}, we proved that if $\w\in \mathtt{c}_{0}\setminus \ell_1$, $\B^*$ is norming and $\B$ is (weak) $\w$-semi-greedy, $\B$ is also $\w$-almost greedy. However, we do not know whether this holds in general, so a salient question is whether it does. In light of Corollary~\ref{corollaryPropCwA} and \cite[Theorem 4.3]{DKTW2018}, this is equivalent to ask whether such bases are quasi-greedy.  
\begin{question}\label{questionremainingcase} Let $\w\in \mathtt{c}_{0}\setminus \ell_1$. Is every (weak) $\w$-semi-greedy Markushevich basis quasi-greedy, and thus $\w$-almost greedy? 
\end{question}
There are further questions about the case $\w\in \mathtt{c}_{0}\setminus \ell_1$ that arise from our research: We proved that in that case, if $\B$ is a (weak) $\w$-(almost)  semi-greedy basis, it is $\w$-superdemocratic and truncation quasi-greedy, but in the proof we used a non-constructive argument and we were not able to obtain quantitative results, that is upper bounds for the $\w$-superdemocracy and truncation quasi-greedy constants depending on the (weak) $\w$-(almost)-semi-greedy constant and other known  and simple  properties of $\B$. 
\begin{question}\label{questionquantitative} Let $\w\in \mathtt{c}_{0}\setminus \ell_1$ and suppose $\B$ is $t$-$\w$-(almost) semi-greedy. Is there an upper bound for the $\w$-superdemocracy constant and the truncation quasi-greedy constant depending on the $t$-$\w$-(almost) semi-greedy constant and perhaps other, simple properties of the basis?  
\end{question}

The weighted variants of (weak) almost (semi) greediness,  democracy and superdemocracy are all preserved under equivalent weights. On the other hand, given Question~\ref{questionremainingcase}, the following problem remains open.
\begin{question}\label{questionequivalentweights}Let  $\w$ and $\w'$ be equivalent weights, with $\w, \w'\in \mathtt{c}_{0}\setminus \ell_1$, and let $0<s\le 1$.  If $\B$ is $s$-$\w$-semi-greedy, is it $s$-$\w'$-semi-greedy? 
\end{question}

\end{document}